\documentclass[11pt]{amsart}

\usepackage{t1enc}
\usepackage{amsthm,amssymb}
\usepackage{latexsym}
\usepackage{amsmath} 
\usepackage{graphics}
\usepackage{epsfig,multicol}
\usepackage{amsfonts}
\usepackage[latin1]{inputenc}
\usepackage[english]{babel}
\usepackage{ mathrsfs }
\usepackage{enumerate}
\usepackage{dutchcal}

\newtheorem{theorem}{Theorem}[section]
\newtheorem{proposition}[theorem]{Proposition}
\newtheorem{lemma}[theorem]{Lemma}
\newtheorem{definition}[theorem]{Definition}

\newtheorem{corollary}[theorem]{Corollary}
\newtheorem{remark}[theorem]{Remark}

\def\hpic #1 #2 {\mbox{$\begin{array}[c]{l} \epsfig{file=#1,height=#2}
\end{array}$}}
 
\def\vpic #1 #2{\mbox{$\begin{array}[c]{l} \epsfig{file=#1,width=#2}
\end{array}$}}

\newcommand  {\rmn}\romannumeral

\newcommand{\5}{\vskip 5pt}

\begin{document}
\newcommand{\BA}{\vpic{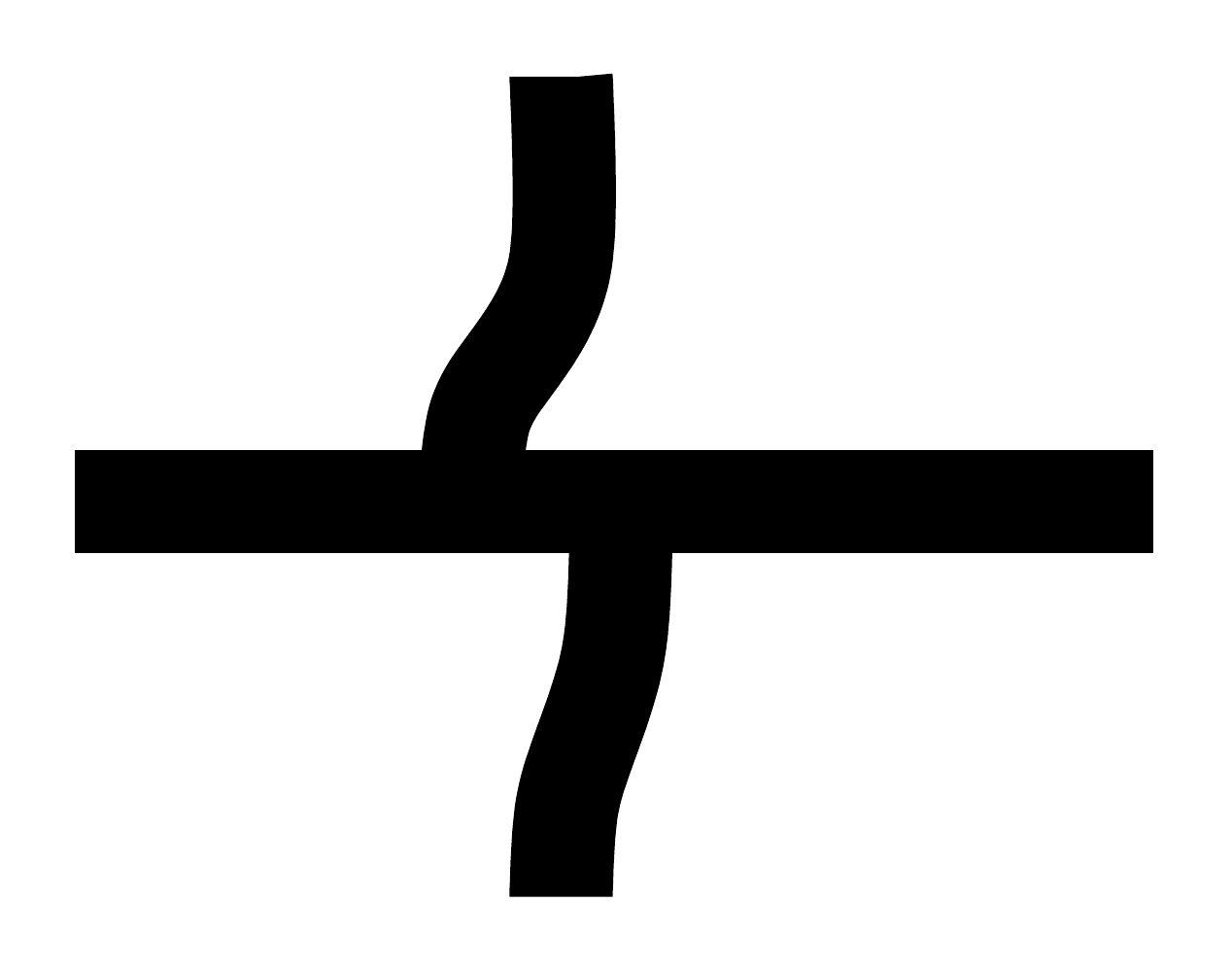} {0.2in}}
\newcommand{\BB}{\vpic{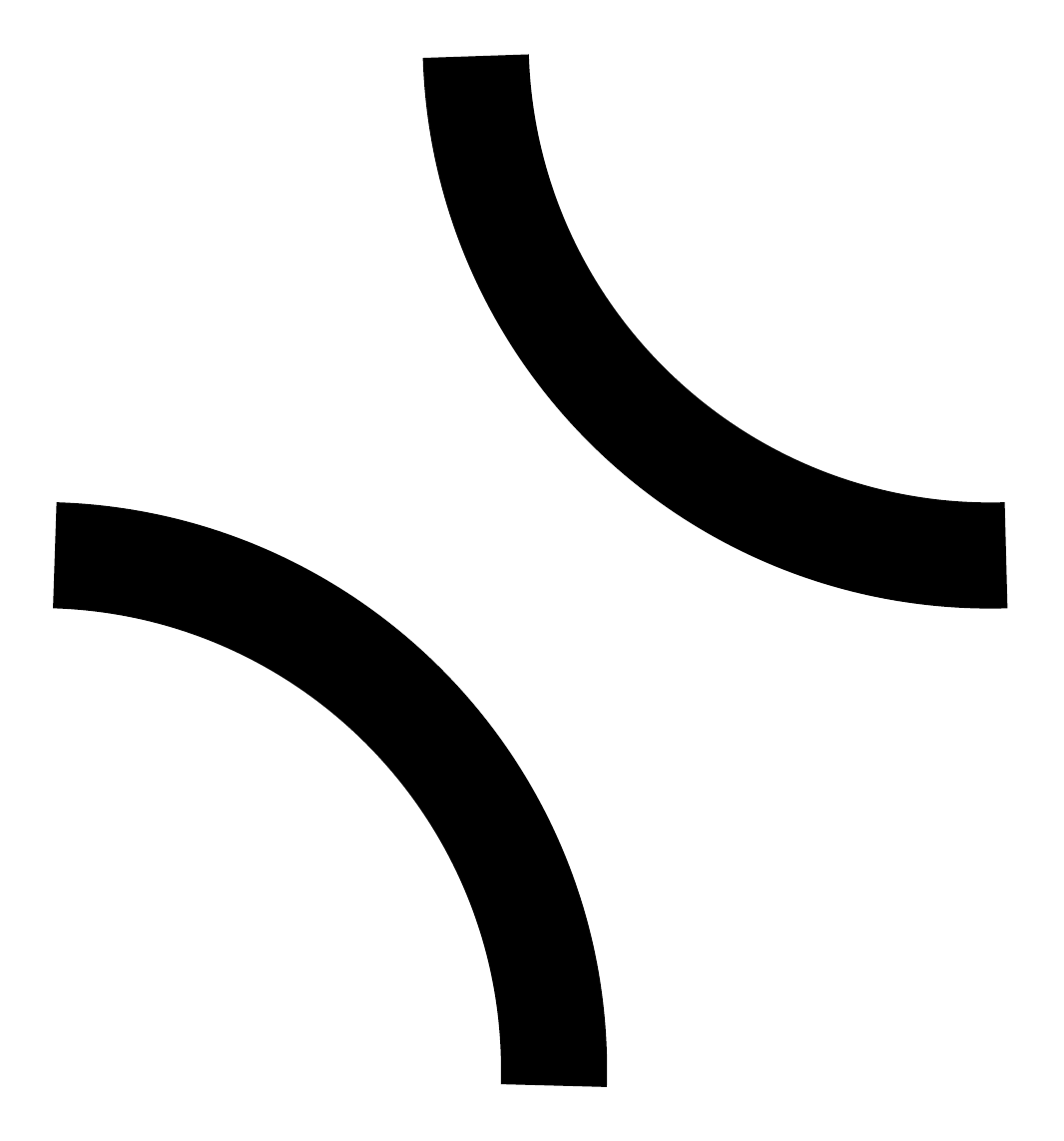} {0.2in}}
\newcommand{\BC}{\vpic{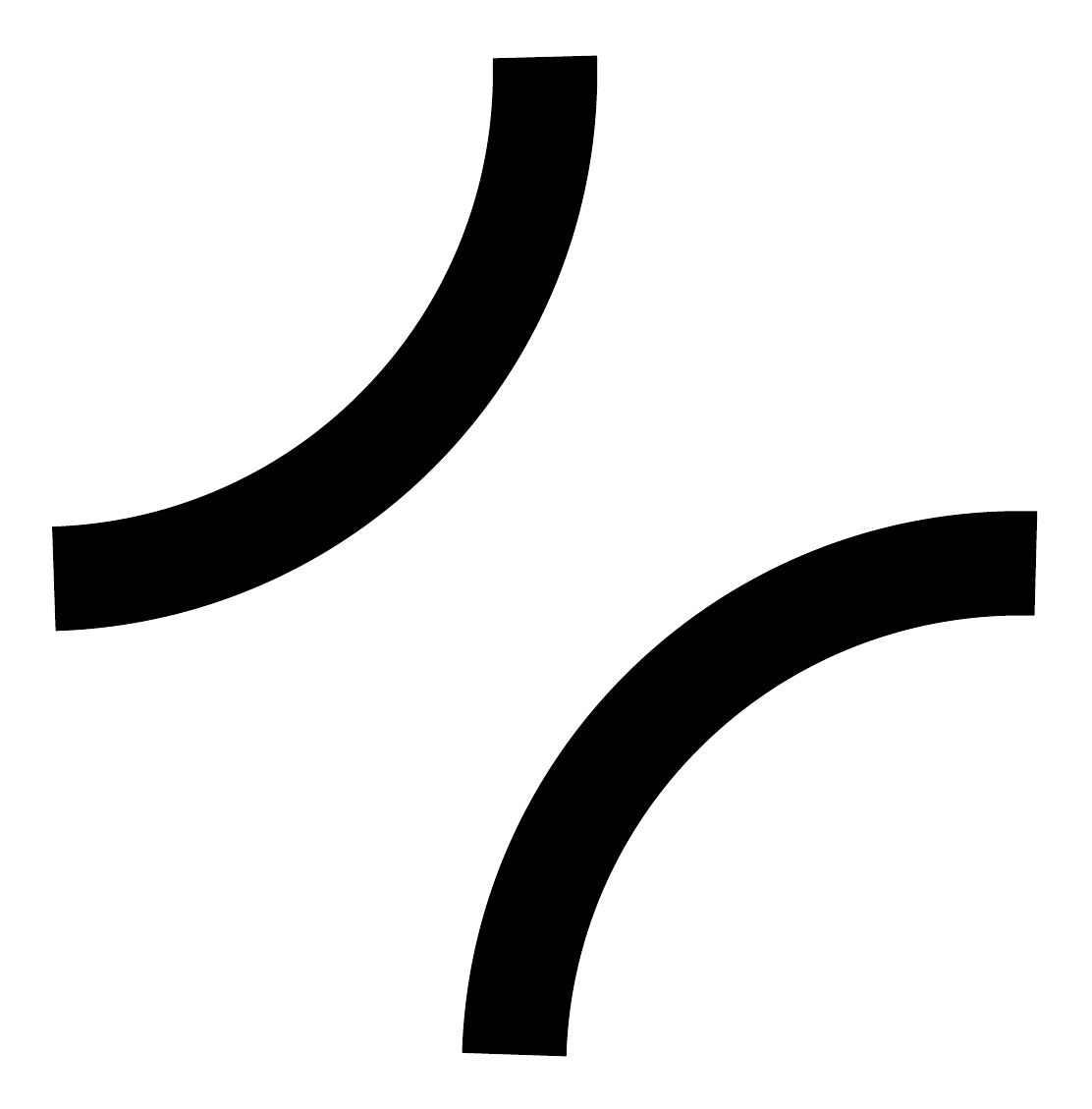} {0.2in}}
\newcommand{\BD}{\vpic{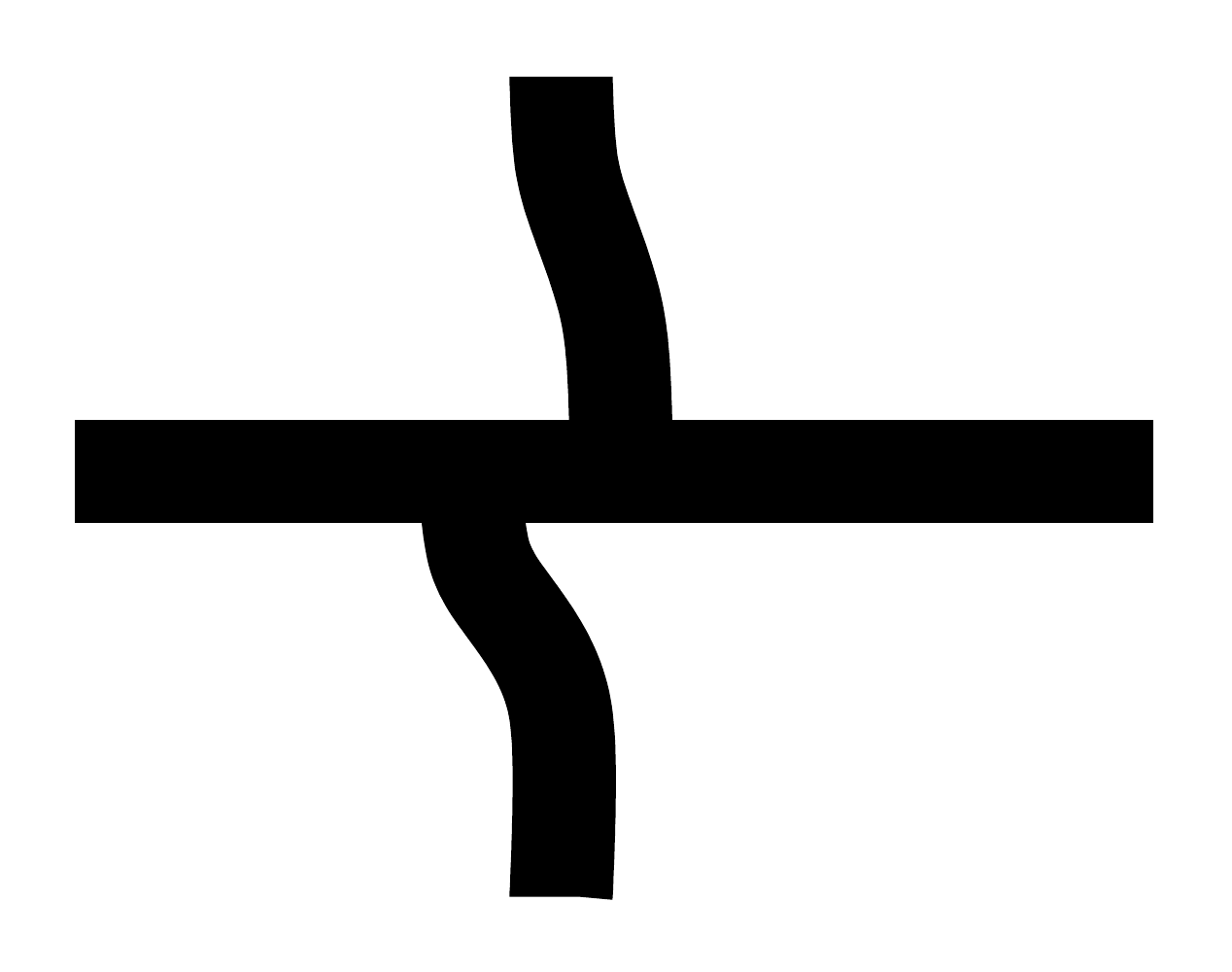} {0.2in}}
\newcommand{\Y}{\vpic{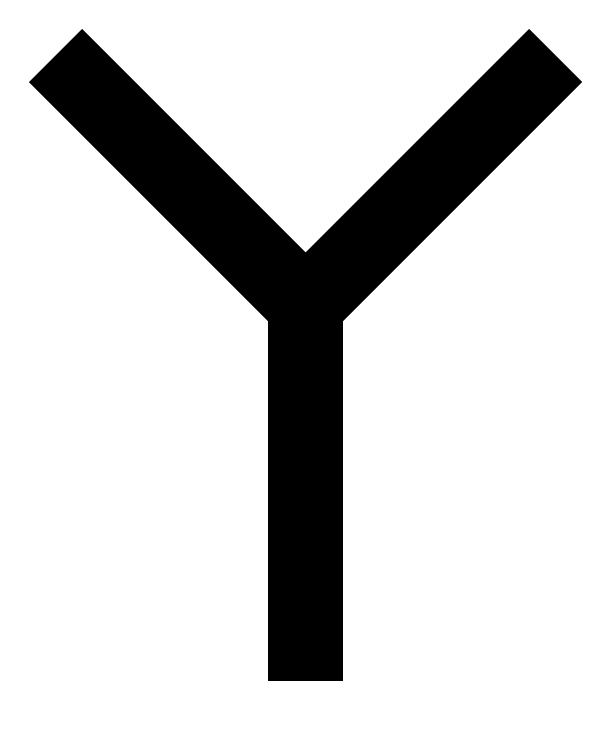} {0.2in}}
\title{Scale invariant transfer matrices and Hamiltionians.}
\author{Vaughan F. R. Jones}
\thanks{V.J. is supported by the NSF under Grant No. DMS-1362138
 and grant DP140100732, Symmetries of subfactors}
\maketitle
\begin{abstract}
Given a direct system of Hilbert spaces $s\mapsto \mathcal H_s$ (with isometric inclusion maps $\iota_s^t:\mathcal H_s\rightarrow 
\mathcal H_t$ for $s\leq t$) corresponding to quantum systems on scales $s$, we define notions of scale invariant and weakly scale
invariant operators. Is some cases of quantum spin chains we find conditions for transfer matrices and nearest neighbour Hamiltonians to be scale invariant or weakly so. Scale invariance forces spatial inhomogeneity of the spectral parameter. But 
weakly scale invariant transfer matrices  may be spatially homogeneous in which case the change of spectral 
parameter from one scale to another is governed by a classical dynamical system exhibiting fractal behaviour.

\end{abstract}

\section{Introduction}
According to dogma, critical phenomena in physics are accompanied by scale invariance-patterns repeat on all scales-
and attendant long-range interactions. In this paper we explore states and observables  of a quantum spin chain that 
exhibit very strict forms of scale invariance, without passing to a continuum limit. The underlying philosophy is that elements of the Thompson groups (\cite{CFP}) 
express local scale transformations on a lattice, which must be supposed infinite for the transformation to 
exist mathematically.  The Thompson group is a replacement for the diffeomorphism group Diff$(S^1)$ (Virasoro algebra) whose
presence in the continuum limit is a consequence of local scale invariance at criticality of a 2-dimensional system.
Thus one the details of the Thompson group representations occurring in various models should
supply both qualitative and quantitative information about physics on the lattice at a critical point. By \cite{Ar} we do
not expect critical behaviour at non-zero temperature so the best place to look for Thompson group symmetry 
is at a \emph{quantum phase transition} where changing some physical parameter besides temperature causes 
an abrupt change of behaviour.

Indeed we will observe three distinct types of behaviour manifested in the asymptotics of the correlation of states with
themselves under lattice rotation/translation by one lattice spacing, as the lattice tends to infinity. This correlation can be 
as simple as an alternation between two values, but most often it  tends rapidly to zero. In this case it is possible to 
rescale the correlation so that it has limits which exist as sesquilinear forms on the pre-Hilbert space of states. 
In the model we investigate there are two sesquilinear forms $S_1$ and $S_2$ and the rescaled correlation tends
to one or the other according to parity. But as the quantum parameter in the model increases, $S_1$ and $S_2$ 
coalesce at a certain critical value after which the convergence is to the common sesquilinear. 

Our approach is wide open to criticism. The states of our "infinite tensor product" (\cite{vNdirect}) have a {\em built in}
long range correlation which forces the impossibility of a continuum limit. This has already been observed by others
and Evenbly and Vidal in \cite{EV} proposed their MERA precisely to overcome this problem. But here we are no longer
trying to construct a continuum limit. 

A perhaps more serious criticism is that the model we present uses exclusively spin-doubling renormalisation which
does not appear particularly physical. To obtain model independent results we should investigate many different models
to see if there are phenomena common to them all. Or introduce bigger groups than the Thompson groups which 
allow more general local scale invariance.

  There have been interesting mathematical developments coming out of this progam-see \cite{jo4}, but the physical relevance of our approach will ultimately be decided by the existence or otherwise of
states with scaling properties in actual physical systems. The spin-doubling operators are no more 
complicated than some of the "gates" in the world of quantum computing (\cite{NC}) so one could in principle
prepare scale invariant states with a machine. But the number of gates required would be rather large. 

It is interesting that the calculation of correlation asymptotics becomes the iteration of a purely classical
dynamical system that may be as simple as a rational function on $\mathbb C P ^1$.  Fixed points, periodic points
and their stability properties thus become the critical values for the quantum system.

Let us give a more precise account of our results. In \cite{jonogo} we proposed the construction of a Hilbert space for the continuum limit of a quantum spin
chain by reversing the process of block spin renormalisation.  We thus obtained Hilbert spaces $\mathcal H_n$ for
a chain of $2^n$ spins with each $\mathcal H_n$ embedded isometrically in $\mathcal H_{n+1}$ by replacing
each spin by 2 copies of itself. The spin doubling isometry is symbolically denoted \Y so that the isometry 
$\mathcal H_{n}\hookrightarrow \mathcal H_{n+1}$ is represented by \Y\Y\Y$\cdots$\Y. We called the inductive limit pre-Hilbert space
$\mathcal H=\underset{\rightarrow} {lim} \mathcal H_n$ the \emph{semicontinuous} limit. The Thompson groups
$F$ (for an open spin chain)  and $T$ (for periodic boundary conditions) act on $\mathcal H$ by unitary transformations
which implement \emph{local scale transformations}.
  
  The result of \cite{jonogo} showed that the continuum limit does not naturally live on the semicontinuous limit or its completion.
  We calculated that for an $SO(3)$ invariant spin $1$ chain rotations by $1\over 2^n$ are hopelessly discontinuous as 
  $n\rightarrow \infty$ for the topology induced on rotations by the circle. (The advantage of this particular spin chain is
  that \Y is unique up to an irrelevant scalar, though in unpublished calculations we have shown the same discontinuity for
  a family of spin-tripling systems where the spin-tripling operator is not at all unique.)

  In the algebraic Bethe ansatz or quantum inverse scaterring method (QISM, \cite{IK,Fad}) one starts with a transfer matrix $T(\lambda)$ depending
 on a spectral parameter $\lambda$ then obtains a nearest neighbour Hamiltonian as the logarithmic derivative of 
 $T(\lambda)$ at some value of $\lambda$ , and other Hamiltonians, and conserved quantities, by further manipulation
 of   $T(\lambda)$.
  
  The calculation of \cite{jonogo} did display a feature common, albeit in a topsey-turvey way, to the semicontinuous limit and the QISM. Namely \emph{the infinitesimal behaviour of space translation is determined  by a transfer matrix with spectral parameter}. (In \cite{jonogo} this was only shown for rotations by $1\over 2^n$ but
  for general rotations there is a more complicated way to manipulate the transfer matrix.) In the QISM the infinitesimal time translation (given by a local Hamiltonian)
  is also given by a transfer matrix with spectral parameter.  These considerations have led us
  to treat the transfer matrix itself as being of fundamental physical significance, being the generator of space translation on the one
  hand and of many Hamiltonians
  and constants of the motion on the other. It thus becomes attractive to look for transfer matrices that are defined on the semicontinuous
  limit. For this we will introduce two notions of scale invariant operators  on $\mathcal H$, the first kind being operators
  on the $\mathcal H_n$ which commute with the inclusion maps $\iota_n^{n+1}$ and the second kind, the weakly scale 
  invariant operators, which commute with the $\iota_n^{n+1}$ in the sense of their  matrix coefficients.
  It was the weakly scale invariant operators that arose in the calculation of \cite{jonogo}. In a simple model coming from 
  the Temperley-Lieb algebra \cite{TL}, we show that scale invariant transfer matrices
  exist in both senses, and that scale invariant nearest neighbour  Hamiltonians exist in the sense of their matrix coefficients.
  
  We should perhaps end by saying that this work began as an attempt to construct chiral conformal field theory on 
  a circle directly from a subfactor (\cite{jo1}), and thus hopefully extending the correspondence begun in \cite{Wa7} to include
  "exotic" subfactors as in \cite{AH}, \cite{EG11},\cite{JMS}.  This has not worked but the intriguing question arises from this paper
  as to whether these exotic subfactors have attendant solutions of our ABC equation of this paper.
  
 \section{Scale invariant transfer matrices.}
 \subsection{Definition}
 Suppose we are given a directed set $(\mathfrak D, \leq)$ (thought of as defining various scales of quantum systems)
 and a direct system $A$ of Hilbert spaces $s\mapsto \mathcal H_s=A(s)$ for $s\in \mathfrak D$ with corresponding isometric
 inclusions $\iota_s^t:\mathcal H_s \rightarrow \mathcal H_t $ for $s\leq t$ satisfying the usual direct system conditions -\cite{jonogo}.
 We let $\mathcal H=\mathcal H_A$ be the direct limit of the $A(s)$. By definition $\mathcal H$ is the set of ordered pairs
 $(d, \xi)$ with $\xi\in \mathcal H_d$, modulo the equivalence relation $(d,\xi)\sim(e,\eta)$ iff there is an $f$ with $f\geq e$ and
 $f\geq d$ with $\iota_d^f(\xi)=\iota_e^f(\eta)$. Since the $\iota$'s are linear isometries, $\mathcal H$ inherits the structure of
 a pre-Hilbert space in the obvious way. Each $\mathcal H_s$ will be identifed with a subspace of $\mathcal H$.
 
 \begin{definition}\label{scaleinvariancedef} 
 With notation as above,
 
 (1) a \emph{scale invariant operator} on $A$ will be a family 
 $T_s:\mathcal H_s\rightarrow \mathcal H_s$ of linear maps such that $$ T_t \circ \iota_s^t=\iota_s^t \circ T_s$$
 
 (2)   a \emph{weakly scale invariant operator} on $A$ will be a family 
 $T_s:\mathcal H_s\rightarrow \mathcal H_s$ of linear maps such that
  $$\langle T_t ( \iota_s^t \xi), \iota_s^t \eta \rangle=\langle  T_s(\xi),\eta \rangle\mbox{   for all   } \xi,\eta \in \mathcal H_s.$$
 \end{definition}
 \begin{remark} Scale invariance implies weak scale invariance but not the other way round  since the weak condition does not force
 $T_t$ to preserve $\mathcal H_s$. 
 \end{remark}
 \begin{proposition} \label{determines}(1)A scale invariant operator $T_s$ determines, and is determined by, an operator $T:\mathcal H\rightarrow \mathcal H$ satisfying $T|_{\mathcal H_s}=T_s$ for all $s$. 
 
 (2) A weakly scale invariant operator determines a sesquilinear (``quadratic'' in the sense of \cite{simon}) form $[,]$on $\mathcal H$ by
 $$ [\xi,\eta]=\langle T_s (\xi), \eta\rangle \mbox{  for any $s$ with  } \xi,\eta \in \mathcal H_s$$
 
 \end{proposition} 
 \begin{proof} These follow immediately from the definition.
 \end{proof}
 \begin{remark} If $n\mapsto s(n)$ is a cofinal sequence in $\mathfrak D$ then we may choose an orthonormal basis
 $\xi_i$ of $\mathcal H$ with $\xi_1,\xi_2,\cdots, \xi_{dim \mathcal H_{s(n)}}$ being a basis of $\mathcal H_{s(n)}$ for
 all $n$. The square matrix $[\xi_i,\xi_j]$ for $1\leq i,j \leq dim \mathcal H_{s(n)}$ is the matrix of $T_{s(n)}$ for this
 basis. 
 
 The form $[,]$ may not define an operator on $\mathcal H$  since there is no a priori control of the size of the
 matrix entries.
 \end{remark}
 
 \subsection{Examples from spin chains.}
 
We will adopt the ``direct limit over trees'' approach of \cite{jonogo} to the semicontinuous limit. This allows us to double
a single spin at a time. Thus we consider the directed set $\mathfrak T$ of planar binary rooted trees
 with $s\leq t$ iff $s$ is a rooted subtree of $t$.

\begin{definition} For each $n$ let   $\mathcal T_n$ be the tree with $2^n$ leaves all at the same distance
from the root. 
\end{definition}

\begin{remark}The $\mathcal T_n$ form a cofinal sequence in $\mathfrak T$.
\end{remark}
If $\mathcal h$ is a (usually finite dimensional) Hilbert space of spin states for a single spin we begin with a ``spin-doubling" operator
$Y:\mathcal h \rightarrow \mathcal h\otimes \mathcal h$. We suppose it is an isometry, i.e. $Y^*Y=id$.

 In Penrose tensor 
notation this condition can be drawn as \vpic {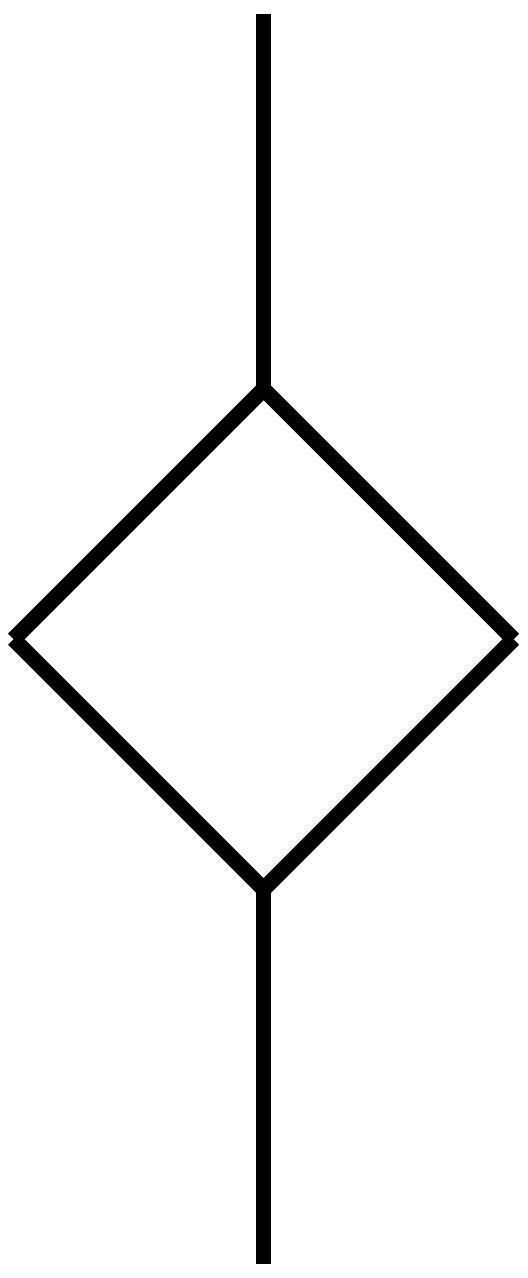} {0.2in} = \hpic {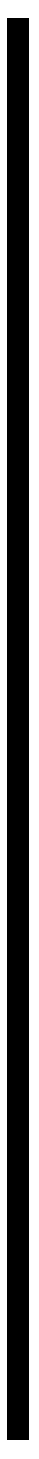} {0.5in} where
 $\Y$ stands for the tensor $Y$ and when it is upside down it represents its adjoint. 
Such a diagram is read from bottom to top, starting with a vector $\xi\in \mathcal h$ which is sent by $Y$ to $\mathcal h\otimes \mathcal h$ then back to $\mathcal h$ by $Y^*$. This kind of notation is very common,see e.g. \cite{CV,Pen}, and is known as ``tensor networks'', and we will generalise it to planar algebras later on.

Now form the direct system $A_Y$ on $\mathfrak T$ where if $t$ has $k$ leaves, $A_Y(t)=\otimes^k \mathcal h$. To define 
the maps $\iota_s^t$ note that any $\leq$ in $\mathfrak T$ decomposes into a sequence of $\leq$'s where one leaf is added a time. So it 
suffices to define $\iota_i :\mathcal H_s\rightarrow \mathcal H_t$ where $t$ is obtained from $s$ by doubling the $ith$ leaf
(from the left). We set
 $$ \iota_i(\eta_1\otimes\eta_2\otimes\eta_3\cdots\otimes\eta_i\otimes\cdots\eta_k)=\eta_1\otimes\eta_2\otimes\eta_3\cdots \otimes Y(\eta_i)\otimes\cdots \otimes\eta_k$$
 or, in tensor network language:
  $\iota_i=\hpic {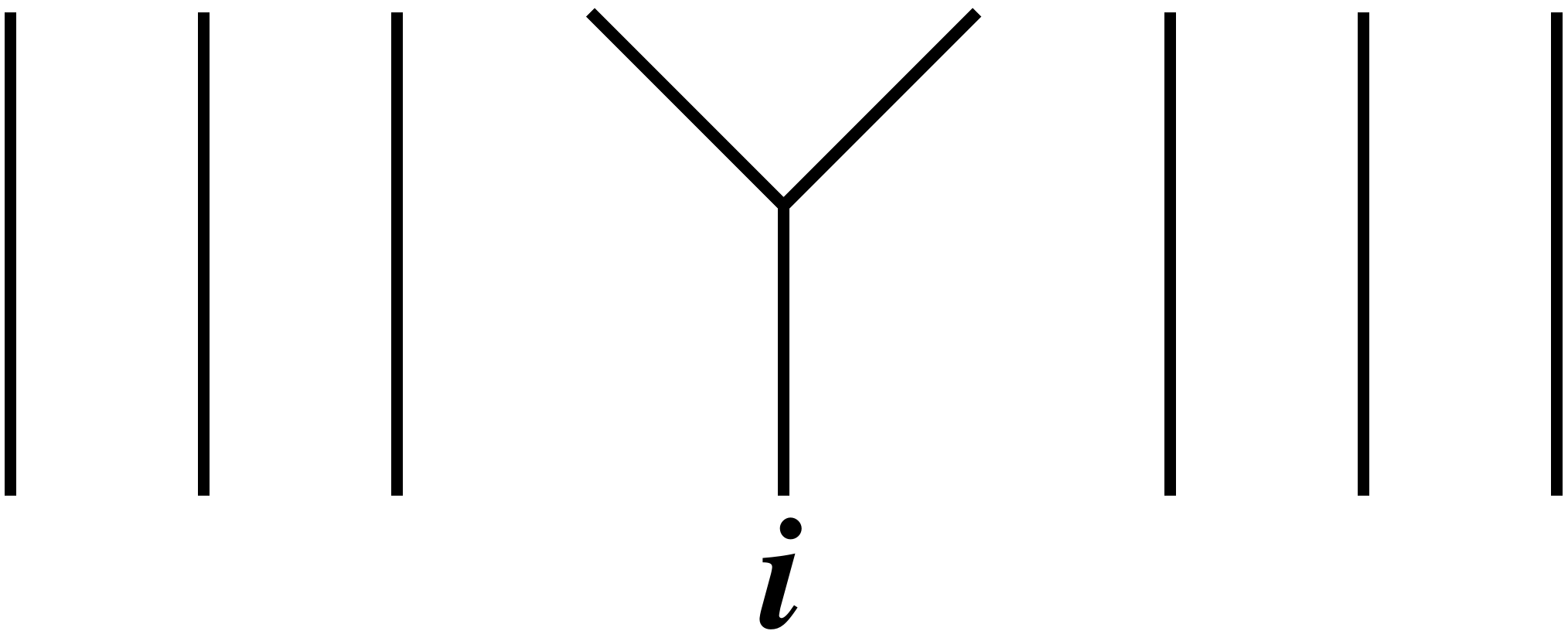} {0.3in} $.\\
  We leave it to the reader to check that these elementary $\iota$'s consistently define a direct system.
  
  Note that the $\mathcal H_{\mathcal T_n}$  are the Hilbert spaces of quantum spin chains with $2^n$ spins
  each having Hilbert space $\mathcal h$. They are embedded one into the next by doubling all the spins with $Y$.
  
  The concept of \emph{transfer matrix} for a spin chain is well known: if we are given a tensor $L$ in 
  $\otimes^4 \mathcal h$ we denote it by \vpic {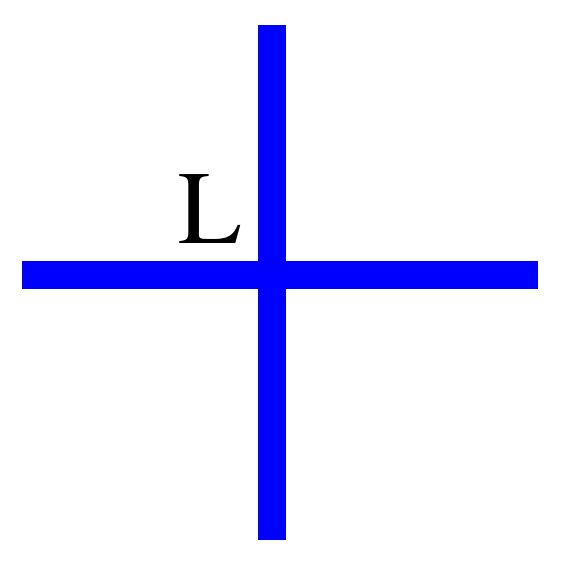} {0.7in} .
  \begin{remark}\label{order} The placement of the $L$ indicates that
   the indices for the tensor $L$ should begin on the string immediately following $L$ in clockwise order, and continue
   in clockwise order.
   \end{remark}
   
    A transfer matrix is then an operator of the form
   $$T(L_1,L_2,\cdots,L_k)= \vpic {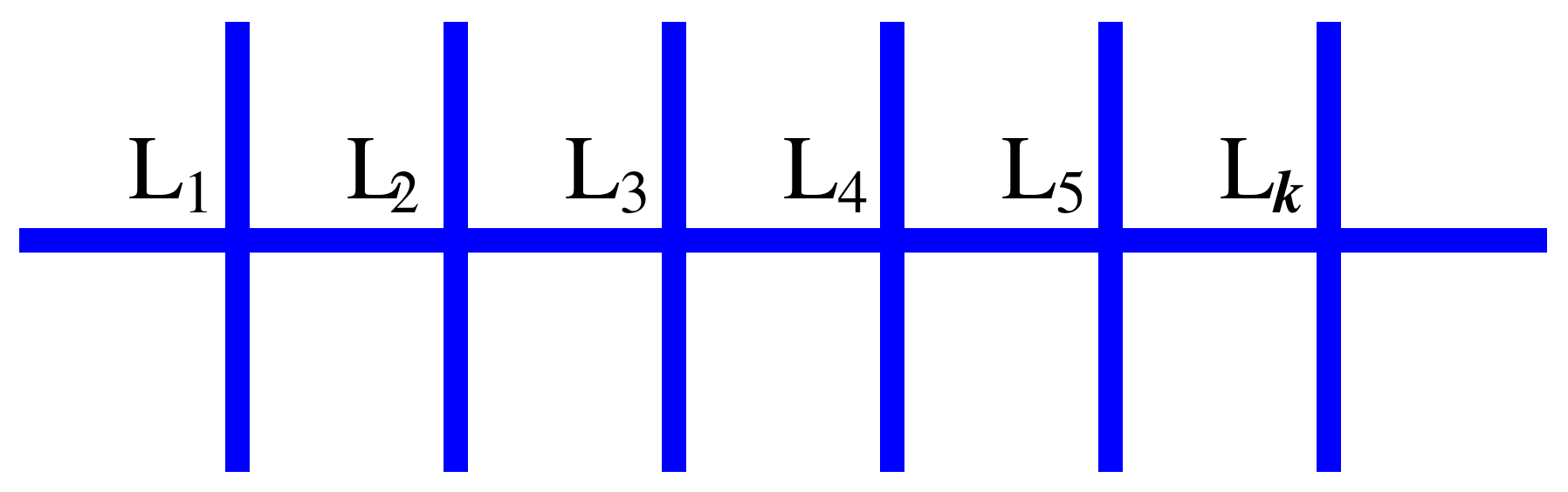} {2.4in}$$  for some choices $L_i$. 
  We need to do something
  about the horizontal boundary to make this picture  an operator from $\otimes^k \mathcal h$ to itself. Let us assume periodic boundary 
  conditions, i.e. we identify the first and last horizontal edges in the picture which may then be thought of as living in an annulus.
  We want to find when $T(L_1,L_2,\cdots,L_k)$ defines a scale invariant operator. For physics we only need to define it on
  the $\mathcal H_{\mathcal T_n}$ and we  could then deduce its values on all the $\mathcal H_s$ by restriction.  But it will
  be just as easy define transfer matrices on each $\mathcal H_t$.
  
     To proceed we introduce an equation 
  which we call the ABC equation, which allows us to extend transfer matrices when a single spin is doubled.

 \begin{definition}
  The ABC equation is the following equation in $\otimes^5\mathcal h$.
  \begin{equation} \label{ABC} 
  \mathscr{D}(A,B)=\mathscr{Y}(C)
  \end{equation}
  
   where $A, B$ and $C$ are unknown tensors in $\otimes^4 \mathcal h$  and \\
    \5
  \5
  $\mathscr{D}(A,B)=$  \vpic {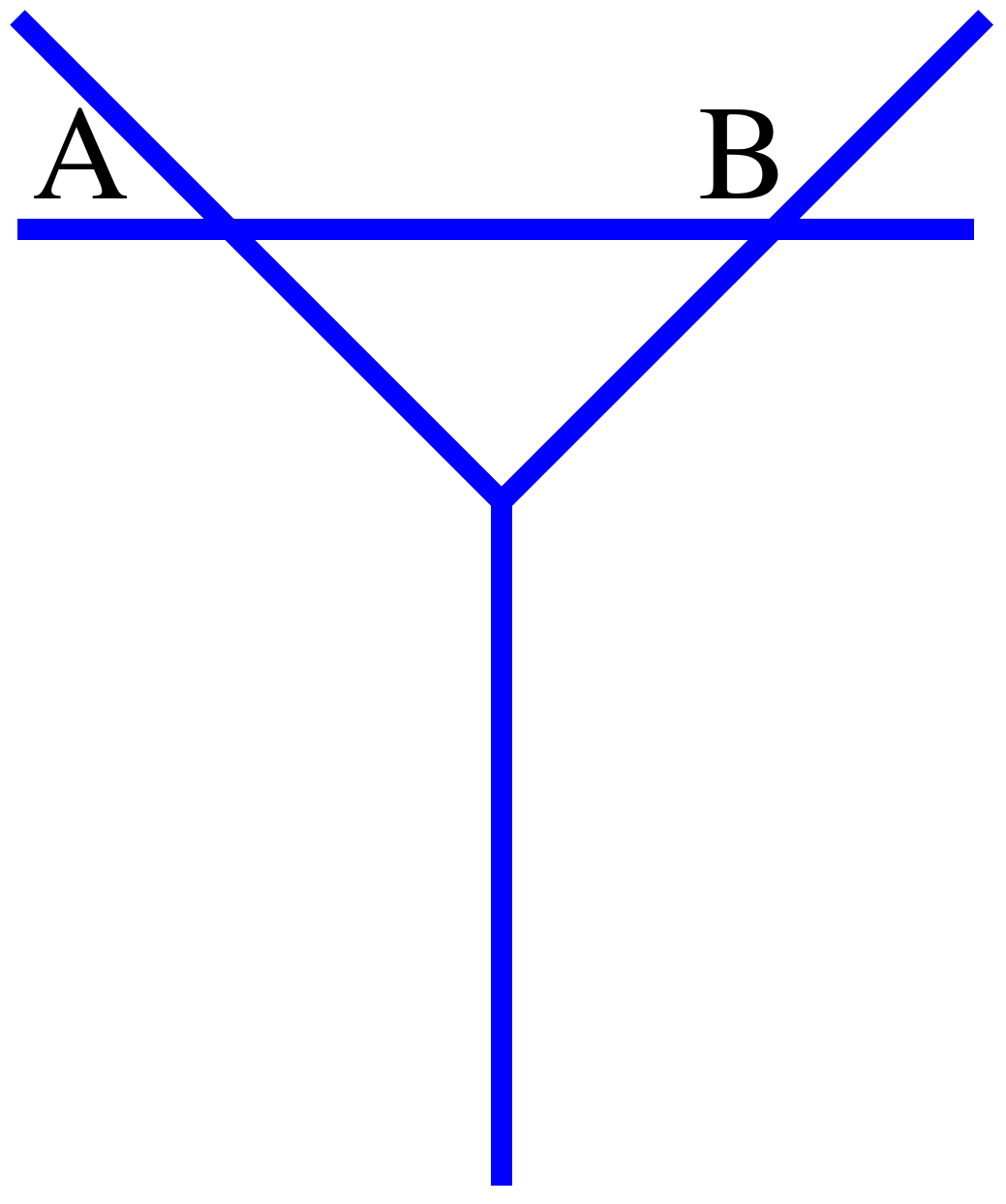} {1in}  $\mathscr{Y}(C)=$  \vpic {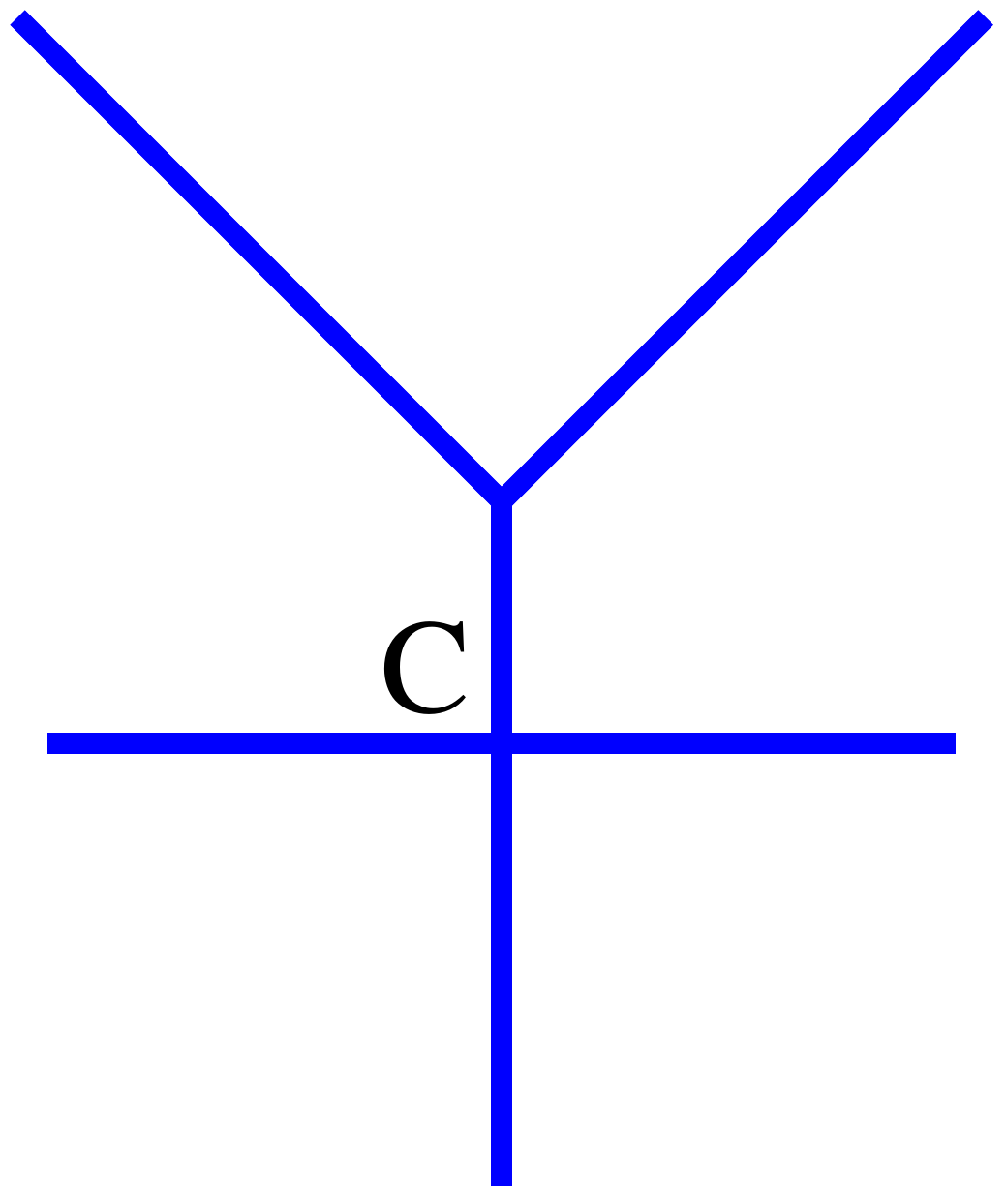} {1in} .
 \end{definition}

    \begin{proposition} Suppose $(A_i,B_i,C_i)$ satisfy the ABC equation for $i=1,2,\cdots, {k}$. Then if 
 $T=T(A_1,B_1,A_2,B_2,\cdots, A_{k},B_{{k}})$, $$T\restriction_{\otimes^k \mathcal h}=T(C_1,C_2,\cdots C_{k})$$
 
 \end{proposition}
 
 \begin{proof} Here $\otimes^k \mathcal H$ is embedded in $\otimes^{k+1}\mathcal H$ by applying
 $Y$ to each tensor product component. So the proof is simply a matter of applying the ABC equation at every trivalent vertex in the diagram for this embedding.
 \end{proof}
 
Given $C$, the ABC equation may or may not have solutions for $A$ and $B$ and if it does it may have many. So in order
to define a scale invariant transfer matrix we need to \emph{choose} solutions if possible. Given such a choice it is easy to define
$T_s$ for any $s\in \mathfrak T$ provided we set up a little notation.

\begin{definition}\label{codetree} For each tree $t\in \mathfrak T$ and  leaf $\mathfrak l$ of $t$, let  $w(\mathfrak l)$ be the word on $\{0,1\}$ read 
from the path on $t$ up from root to leaf, with $0$ on left turns at trivalent vertices and $1$ at right turns. 

\end{definition}
\begin{definition} \label{transfermatrix} 
For any function $w\mapsto L_w\in \otimes^4\mathcal h$ from all words $w$ on $\{0,1\}$ and $t\in \mathfrak T$ with $k$ leaves, define $T^L_t$ to be the transfer matrix (on $\mathcal H_t$) with periodic boundary conditions:

$$T_t^L =\vpic {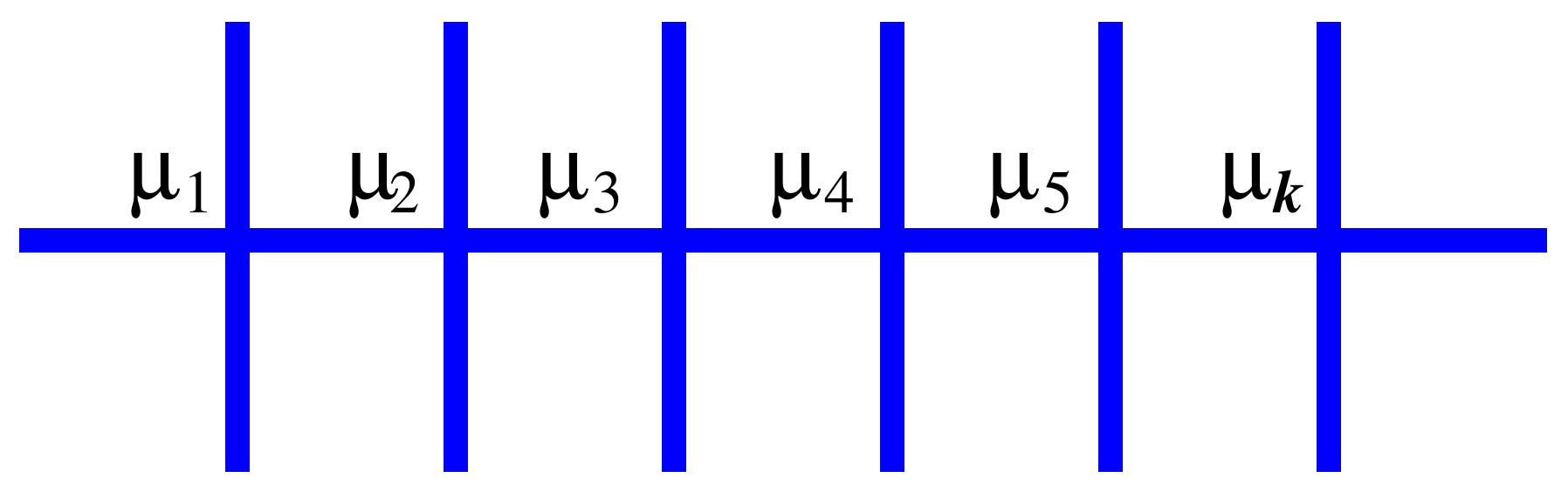} {1.9in} $$ where $\mu_i=L_{w_{i}}$, where the leaves of $t$ are numbered 
 from left to right and $w_i$ is the word coding for the $ith$. leaf.

In the special case $t=\mathcal T_n$ where all the branches of $t$ have the same length (so $t$ has $2^k$ leaves for some $k$), and $L_w$ 
takes the same value $L$, we use $T_L$ for $T^L_t$. Thus \\
for $k=3$, $T_L=$ \vpic {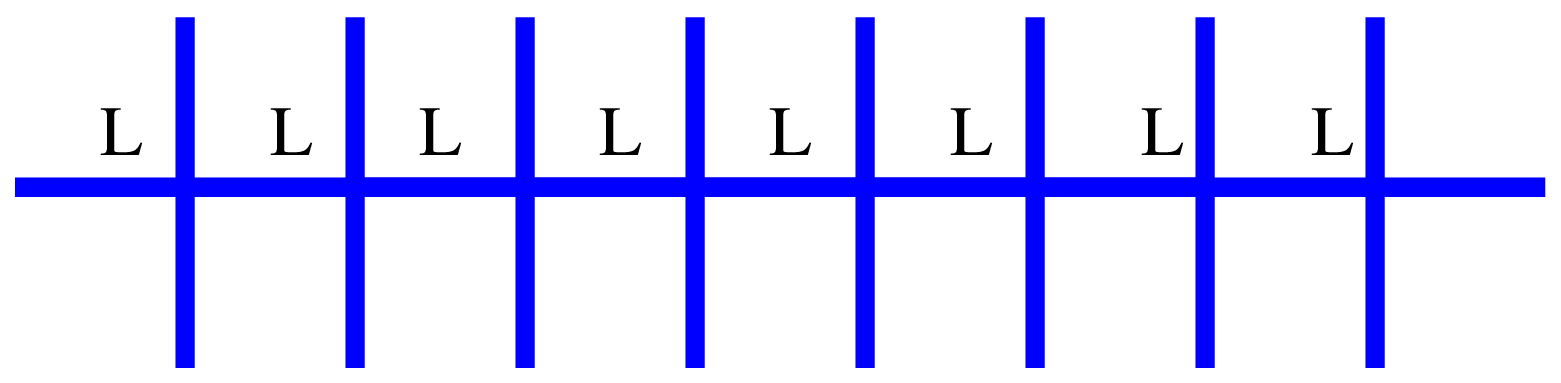} {3in}  .

\end{definition}
\begin{definition} A \emph{coherent choice} $L_w$ of tensors, for all words $w$ on $\{0,1\}$ will be one
such that for each $w$ , $$ \mathscr{D}(L_{w0},L_{w1})=\mathscr{Y}(L_w)$$
i.e. putting $A=L_{Lw0},B=L_{w1}, \mbox{  and  } C=L_w$
gives a solution of the ABC equation.
\end{definition}

\begin{proposition}\label{scaleinvariance}Suppose $w\mapsto L_w$ is a coherent choice of tensors,
then $t\mapsto T_t^L$ defines a scale invariant operator on the direct system $A_Y$.
\end{proposition}
\begin{proof} Just apply the ABC equation every time $t$ differs from $s$ by doubling a single vertex. The 
formulas defining $w$ and $L_w$  take care of the book-keeping for the values of the spectral parameter.
\end{proof} 

Thus in a particular model, to exhibit scale invariant transfer matrices it will suffice to exhibit coherent choices of
tensors.
\begin{remark}\label{moregeneral} In fact, because we are dealing with the direct limit, a coherent choice of tensors does not have to 
be defined for \emph{all} words. Given $t\in \mathfrak T$ one can choose a solution to the ABC equation for every leaf
of $t$, form the corresponding transfer matrix on $\mathcal H_t$  and extend it using coherent choices for each leaf of $t$.
This will define an  operator on $\mathcal H$ which should be considered scale invariant. This requires a slight but 
obvious modificaiton of definition \ref{scaleinvariancedef} which we have not given to avoid confusion.   The restriction of 
this operator  to $\mathcal H_s$ for trees not containing
$t$ will not in general preserve the subspace $\mathcal H_s$ so is not a transfer matrix on $\mathcal H_s$.
\end{remark}
For the convenience of the reader we exhibit a pair $(t,\otimes^7 \mathcal h)$ whose tree has 7 leaves,  and the corresponding transfer matrix for
some choice of $L$'s:\\
\5
\vpic {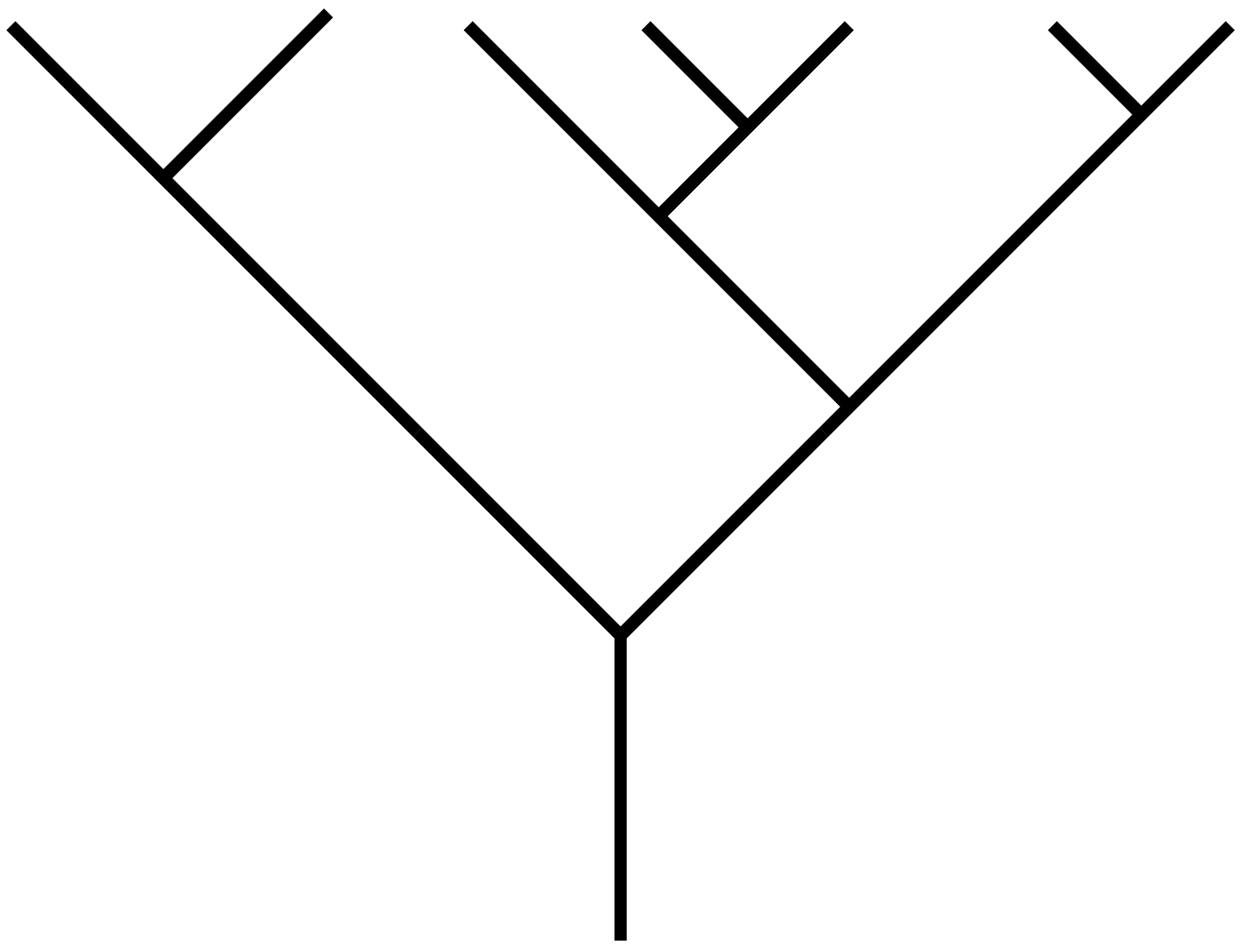} {2.3in} \vpic {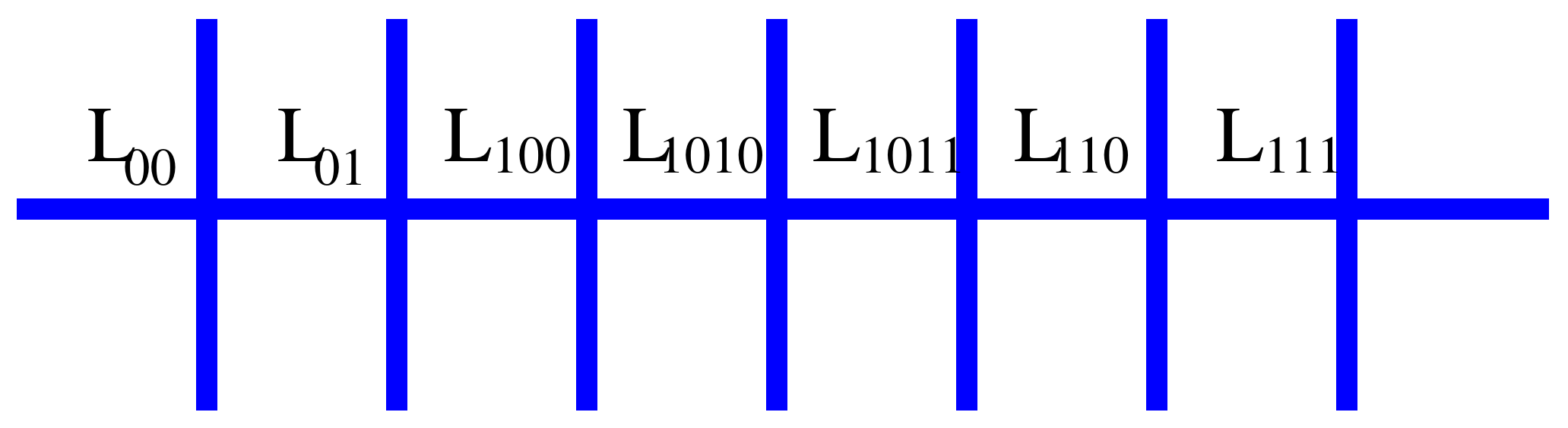} {2.8in}

Before exploring explicit solutions to the ABC equation we point out a few general features. 
  We will  exhibit a symmetry of the ABC equation that uses operations $\alpha$ and
  $\beta$ on certain subsets of $\otimes^4 \mathcal h$. Note that  $\otimes^4 \mathcal h$  becomes a unital associative algebra under (see remark \ref{order})
   $$(AB)_{i,j,k,\ell}=\sum_{p,q}A_{i,p,q,\ell}B_{p,j,k,q}$$ which we will call ``multiplication'' 
  and under $$(A.B)_{i,j,k,\ell}=\sum_{p,q}A_{p,q,k,\ell}B_{i,j,q,p}$$ which we will abusively call ``comultiplication''. Both multiplications have obvious diagrammatic representations.
 $\mathscr F$ will be the rotation by $\pi/2$: $$\mathscr F(A)_{i,j,k,\ell}=A_{\ell,i,j,k}$$ 
  The two multiplications are conjugate: $$\mathscr F(AB)=\mathscr F(A).\mathscr F(B)$$
  Or, $\mathscr F$ gives an isomorphism from $\otimes^4 \mathcal h$ under multiplication to $\otimes^4 \mathcal h$ 
  under comultiplication.
  
  % We will assume for simplicity that $\mathscr F^2=id$
  
  We will use $X^\tau$ for the inverse of $X$ for the comultiplication structure.

  \begin{lemma} \begin{enumerate} We have\\ 
  \item If $X^{-1}$ exists then $\mathscr F(X)^\tau $ does also and $$\mathscr F(X)^{\tau}= \mathscr F(X^{-1})$$.
 \item If $X^\tau$ exists then $\mathscr F(X)^{-1}$ does also and $$\mathscr F(X)^{-1}=\mathscr F(X^\tau)$$.
 \item The same assertions hold with $\mathscr F$ replaced by $\mathscr F^{-1}$.
  \end{enumerate}
  \end{lemma}
  \begin{proof} $\mbox{  }$\\
  \begin{enumerate}
  \item follows simply from the fact that $\mathscr F$ is an isomorphism from multiplication to comultiplication.
  \item is a bit more subtle. Since $\mathscr F^2$ is a comultiplication antiautomorphism,
 $\mathscr F(X^\tau)=\mathscr F^{-1}(\mathscr F^2 (X^\tau))=\mathscr F^{-1}(\mathscr F^2 (X)^\tau))=\mathscr F(X)^{-1}$ ,
 the last equality being because $\mathscr F^{-1}$ is an isomorphism from comultiplication to multiplication.
 \item follows by applying $\mathscr F^{-2}=\mathscr F^2$ to both sides and using its antiautomorphism properties.
 \end{enumerate}
\end{proof} 
 \begin{definition}Let $D=\{X|X^{-1} \mbox{ and } X^\tau \mbox{ exist }\}$.

 \end{definition}
 
Note that the set $D$  is NOT invariant under taking inverses and coinverses. For instance
    the element $(d^2-1)\BB + d\BC$ considered below is invertible and coinvertible but its inverse is not coinvertible. 
    But we can correct inverse and coinverse by $\mathfrak F$ to form $\alpha$ and $\beta$ whose
    domain and range behave appropriately.

 \begin{definition} Let  $X\in \otimes^4 \mathcal h$ be such that both $X$ and $\mathscr F^{-1} (X)$ are invertible. Then define
   $$\alpha(X)=\mathscr F^{-1}(X^{-1})\mbox{  and  } \beta(X)= \mathscr F(X^\tau)$$.
  \end{definition}
Observe that the domain of $\alpha$ is the set of invertibles and the domain of $\beta$ is the set of coinvertibles.
\begin{lemma} If $X\in Dom(\alpha)$ then $\alpha(x)\in Dom(\beta)$ and $\beta(\alpha(X))=X$. Also

if $X\in Dom(\beta)$ then $\beta(x)\in Dom(\alpha)$ and $\alpha(\beta(X))=X$.
And $\alpha$ and $\beta$ commute with $\mathscr F^2$.

\end{lemma}
\begin{proof} The first two assertions follow easily from the previous lemma.
The last assertion is trivial.
\end{proof}
    
    We shall now see how $\alpha$ and $\beta$ arise in the ABC equation.
    
   %But for solutions of equation \ref{ABC} we will be able to 
   %completely determine the intersection of the domains of $\alpha^n$.
 \5  
   \emph{For simplicity we assume from now on that $Y$ is invariant under rotation: $$Y_{i,j,k}=Y_{j,k,i}.$$}
   \begin{proposition}\label{ABCsymmetry}
   Suppose $A\in domain(\alpha),B\in domain(\beta)$. Then 
   $$ \mathscr{D}(A,B)=\mathscr{Y}(C)\iff  \mathscr{D}(\beta(B),\mathscr F^2 (C))=\mathscr{Y}(\mathscr F^2(A))\iff \mathscr{D}(\mathscr F^2(C),\alpha(A))=\mathscr{Y}(\mathscr F^2 (B))$$
   \end{proposition}
   \begin{proof} Suppose $ \mathscr{D}(A,B)=\mathscr{Y}(C)$. Attach  $B^{\tau}$ to the diagrams of $\mathscr{D}(A,B)$ and $\mathscr{Y}(C)$ to obtain the following
   equality:\\
   
   \vpic {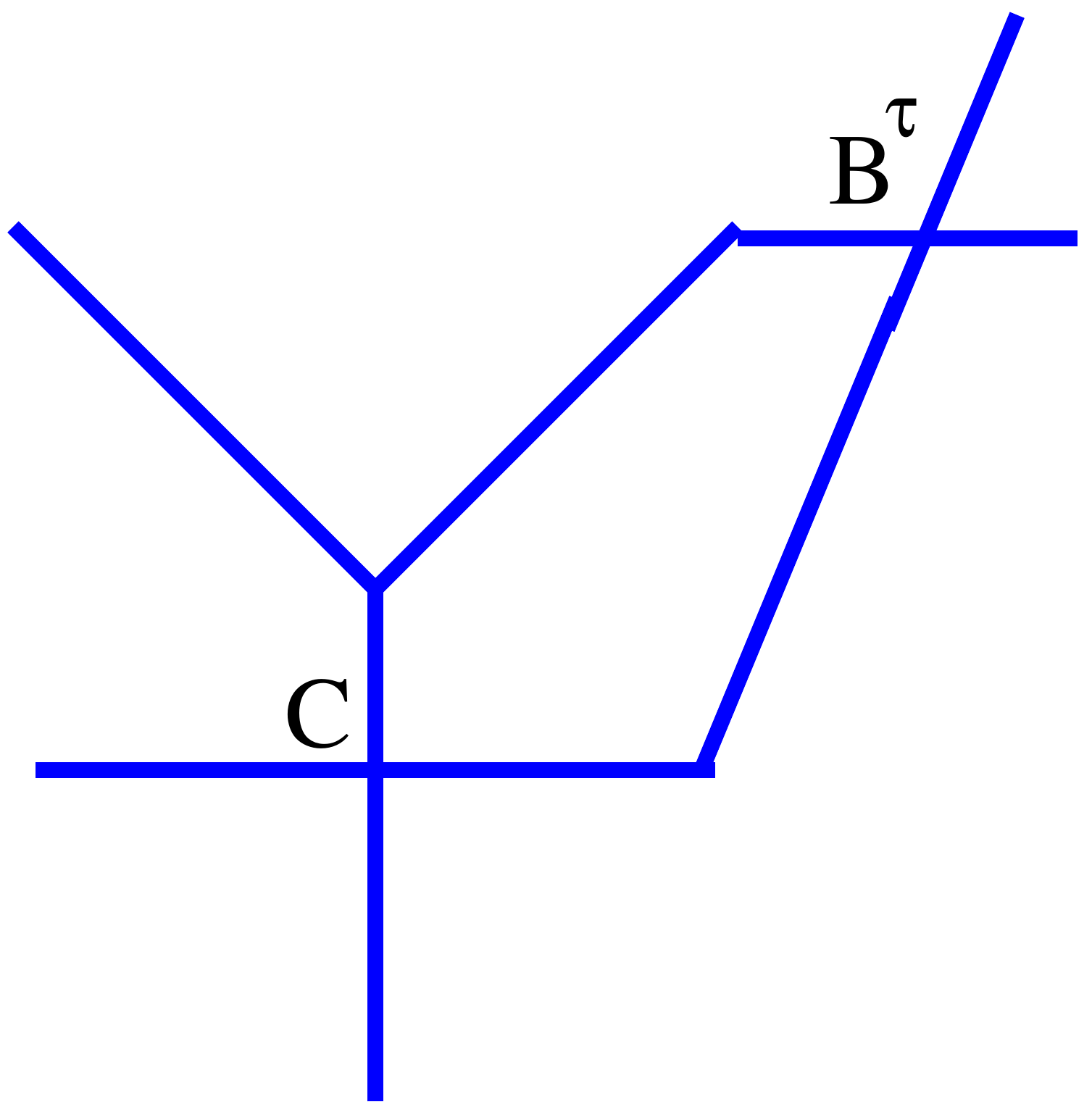} {1.5in}=\vpic {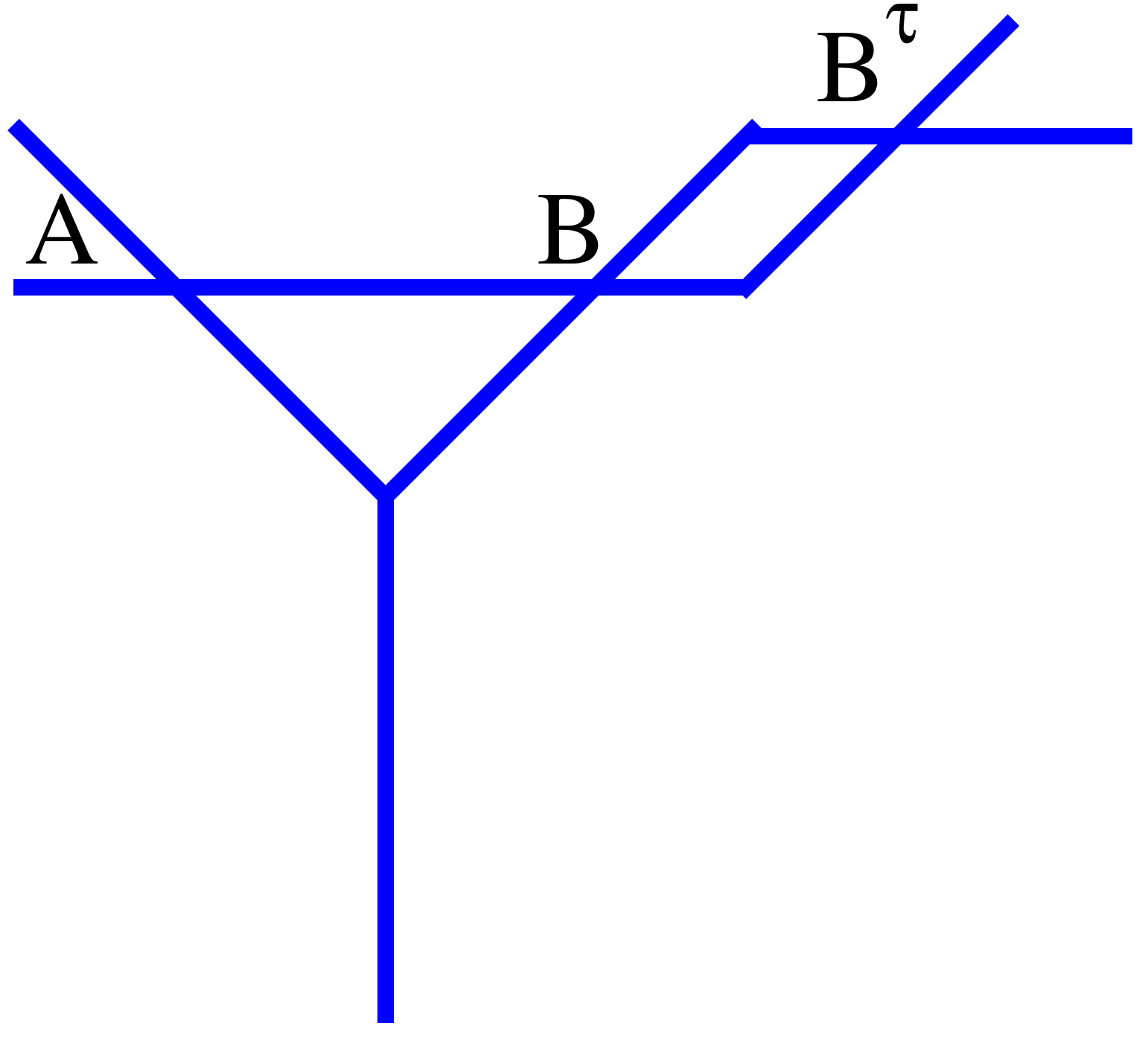} {1.5in}  = \vpic {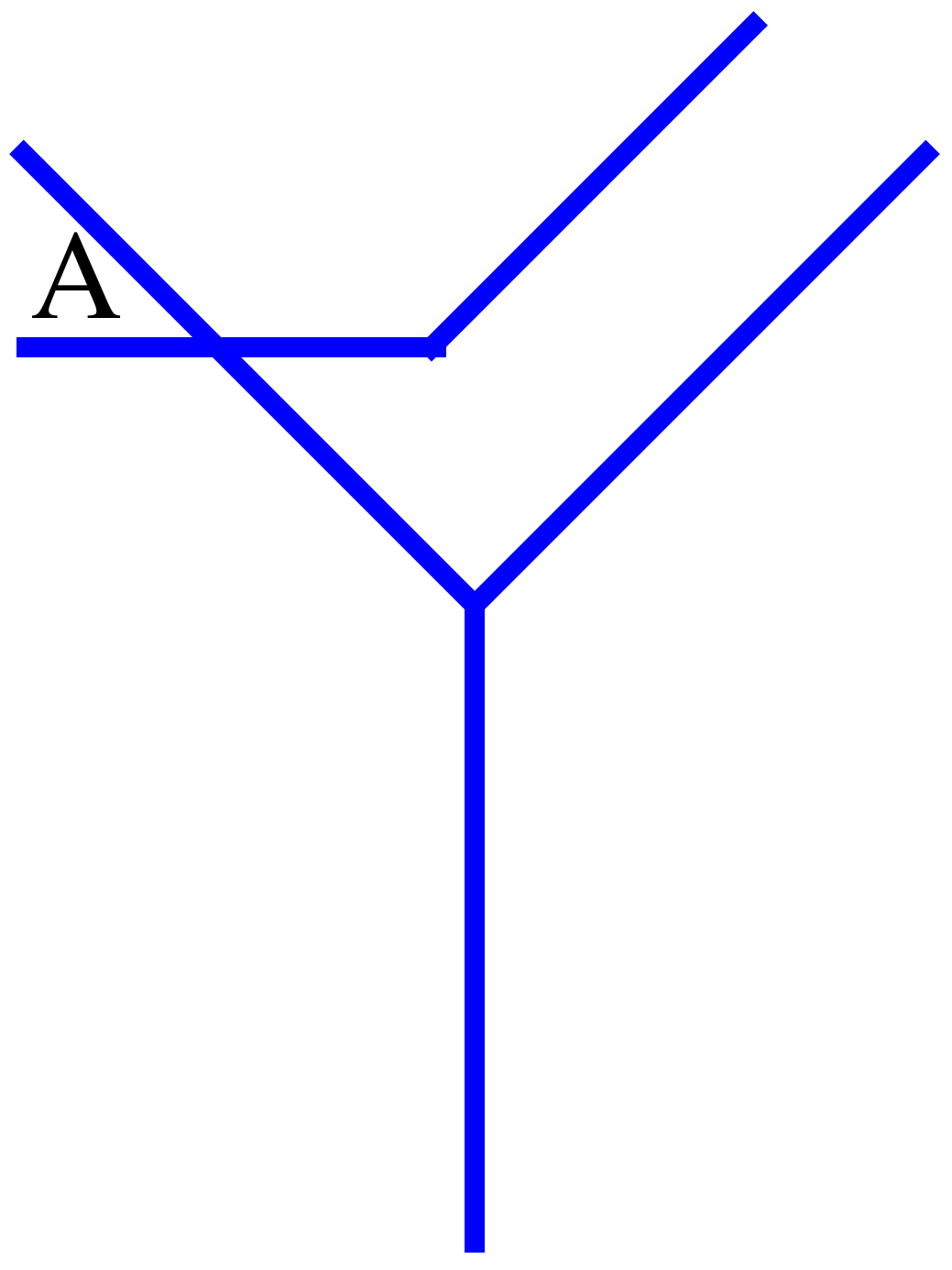} {1.5in} .

   Rotate by $2\pi/3$ and isotope a little to obtain: 
   \5 \5
 \hspace {1in}   \vpic {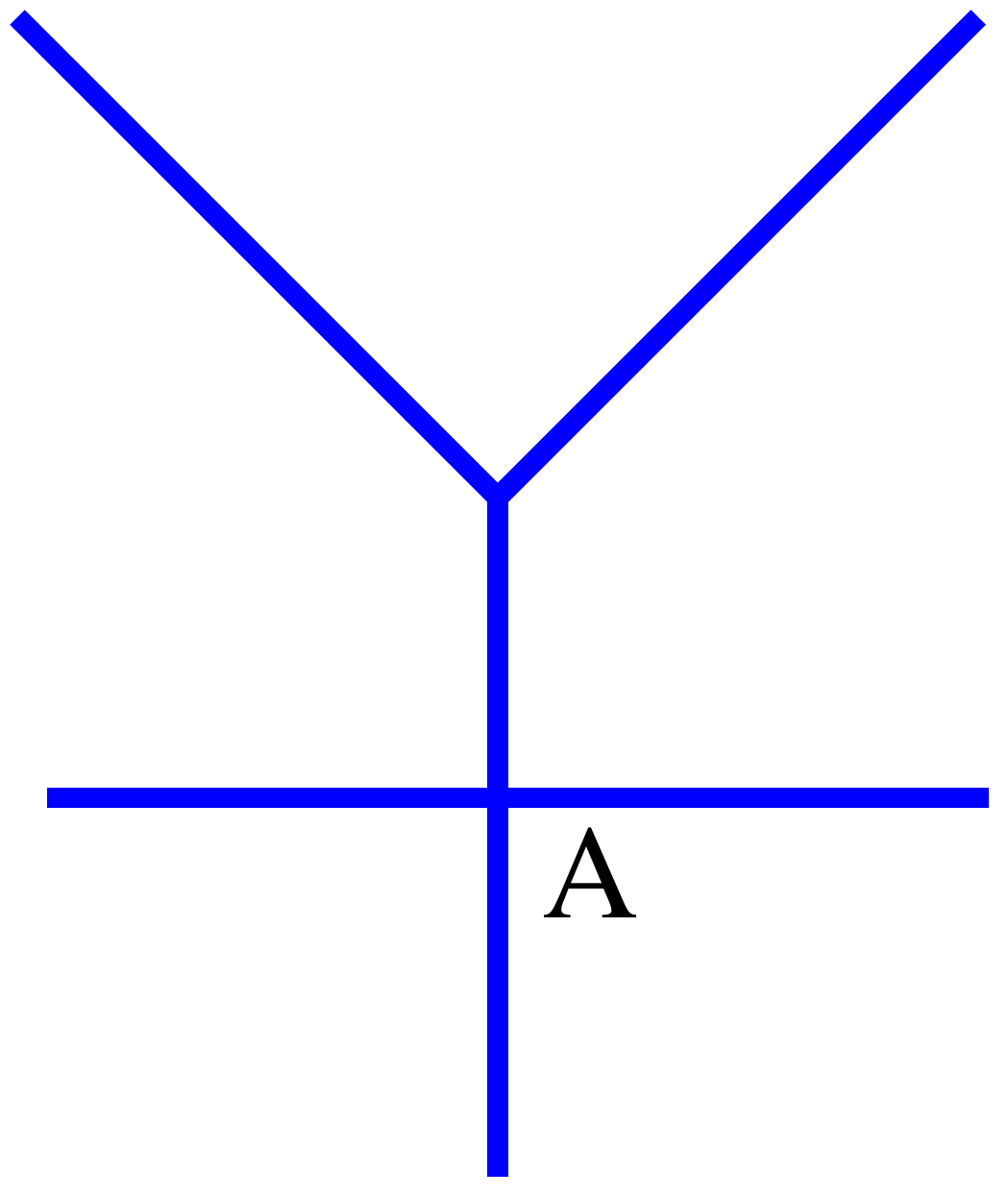} {1.5in}   = \vpic {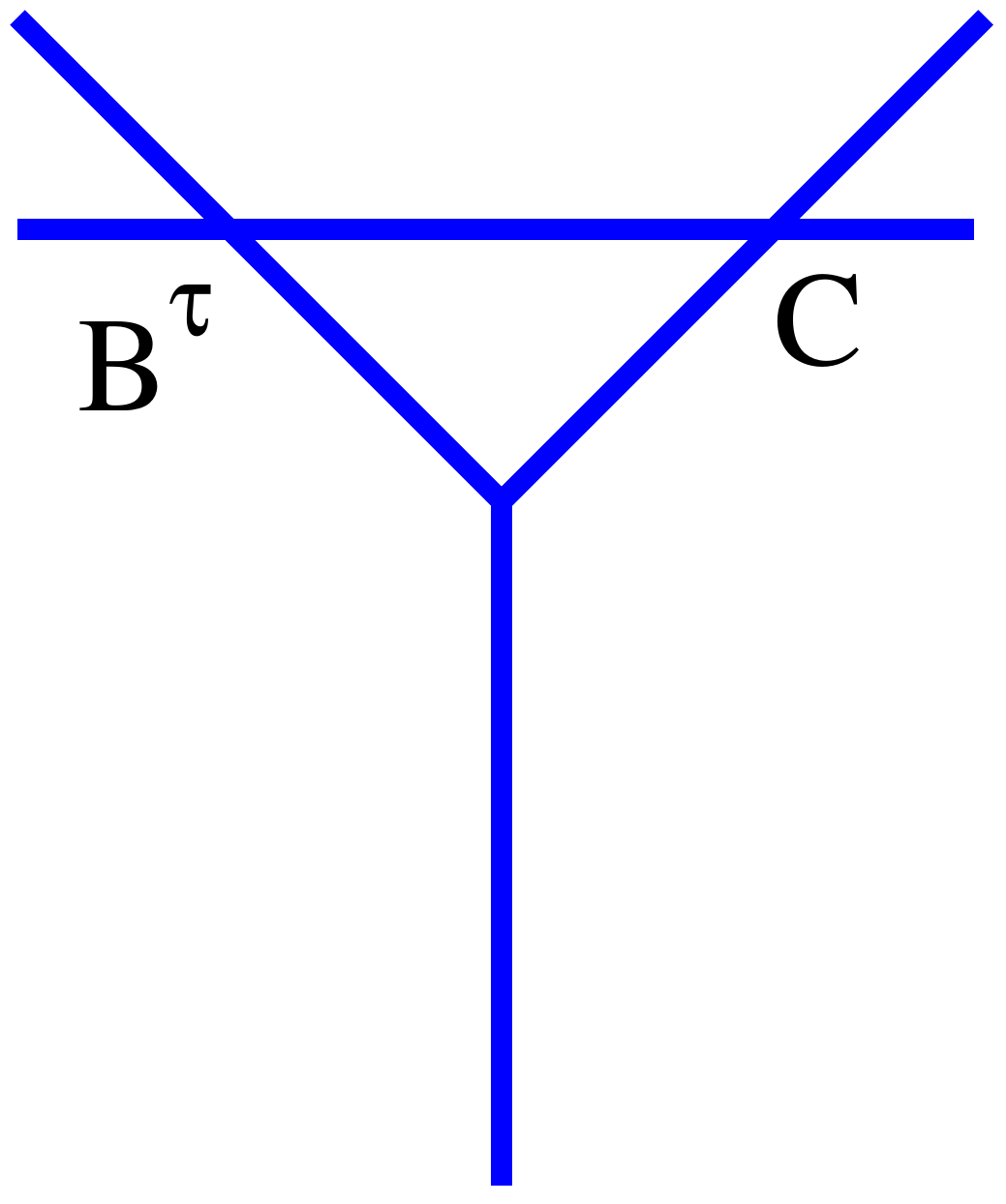} {1.5in} .
 
 Rotating the appropriate tensors  we get
 $\mathscr{Y}(C)=\mathscr {D}(A,B) \implies  \mathscr{D}(\beta(B),\mathscr F^2 (C))=\mathscr{Y}(\mathscr F^2 (A))$.
 
 The process is clearly reversible so $ \mathscr{D}(A,B)=\mathscr{Y}(C)\iff  \mathscr{D}(\beta(B),\mathscr F^2 (C))=\mathscr{Y}(\mathscr F^2 (A)) $, and the other equivalence
 is proved similarly.
   \end{proof}
   
   Unfortunately $domain(\alpha^n)$ is not invariant under $\alpha$ for $n\in \mathbb Z$.

  \begin{definition} \label{scales} Let us say that $(A,B,C)\in \otimes^4 \mathcal h$  ``scales"  if it satisfies the ABC equation and
  $A, B, C \in domain(\alpha^n)$ for all $n\in \mathbb Z$.
  \end{definition}

\begin{theorem} \label{existence} Suppose $\mathscr F^2=id$ and  $(A,B,C)$ scales.
Then define $L_w$ inductively on words on $\{0,1\}$ by $L_0=A, L_1=B,$
$$ L_{w0}=\begin{cases} \alpha^{k+1}(B)&\mbox{ if } L_{w}=\alpha^k(A) \\
                                                         \alpha^k(C)&\mbox{ if } L_{w}=\alpha^k(B) \\
                                                         \alpha^k(A)&\mbox{ if } L_{w}=\alpha^k(C) 
\end{cases}$$ and
$$ L_{w1}=\begin{cases} \alpha^{k}(C)&\mbox{ if } L_{w}=\alpha^k(A) \\
                                                         \alpha^{k-1}(A)&\mbox{ if } L_{w}=\alpha^k(B) \\
                                                         \alpha^k(B)&\mbox{ if } L_{w}=\alpha^k(C) 
\end{cases}$$

Then $L_w$ is a coherent choice of tensors so the $L_w$ determine a scale invariant transfer matrix by proposition \ref{scaleinvariance}.
\end{theorem}
\begin{proof}This follows immediately by induction from \ref{ABCsymmetry}.
This formal proof somewhat obscures what is going on. The idea is that, once $ABC$ is a solution, so are
$\alpha(B)CA, \alpha(C)A\alpha(B),\alpha(A)\alpha(B)\alpha(C),C\alpha^{-1}(A)B$ and so on.
 
\end{proof}

Thus provided we have a solution to the ABC equation (with $\mathscr F^2=id$) the only problem in using it to construct
a scale invariant transfer matrix is the problem of the domains of $\alpha$ and $\beta$.  We will solve this completely in a special model.
 \iffalse

A special case is when $k=0$. Then there is only one solution of the ABC equation, call it $(\gamma,\kappa,\lambda)$, and call the
corresponding operators constructed by the proof of the theorem $T_\lambda$ with $T_{n,\lambda}$ being the transfer matrix which is
the restriction of $T_\lambda$ to $V_{2^n}$. The diagram below is not a planar tangle but a depiction of the $T_{n,\lambda}$ for 
$n=1,2,3,4$. It illustrates the proof of the previous theorem (note that we need to apply periodic boundary conditions):
\5 \5

\hspace{0.5in} \vpic {transferbuild} {4in}
\fi

\subsection{A concrete example: the (quantum)$SO(3)$-invariant spin $1$ chain.} \label{concrete} 
It was explained in \cite {J22} how subfactors and bimodules provide quantum spin chains more elaborate than
with ordinary spins. The spin state space may fractional dimension. In this sense there is for instance a spin
chain for the Andrews-Baxter-Forrester models of \cite{ABF}, on which the transfer matrix acts. Subfactors/bimodules
are known to be described by planar algebras (\cite{jo2,jo3}) and all the calculations we have done in this paper
are diagrammatic, so extend without alteration to (unshaded) positive definite planar algebras. 
The planar algebra is a graded vector space $\mathcal P= (P_n)$ on which the planar operad acts. Elements of
$P_n$ are called $n$-boxes and they may be inserted into the internal discs of a planar tangle. Each $P_n$ is 
equipped with an antilinear involution $*$ for which the sesquilinear form 
$$\langle R, S\rangle =    \vpic {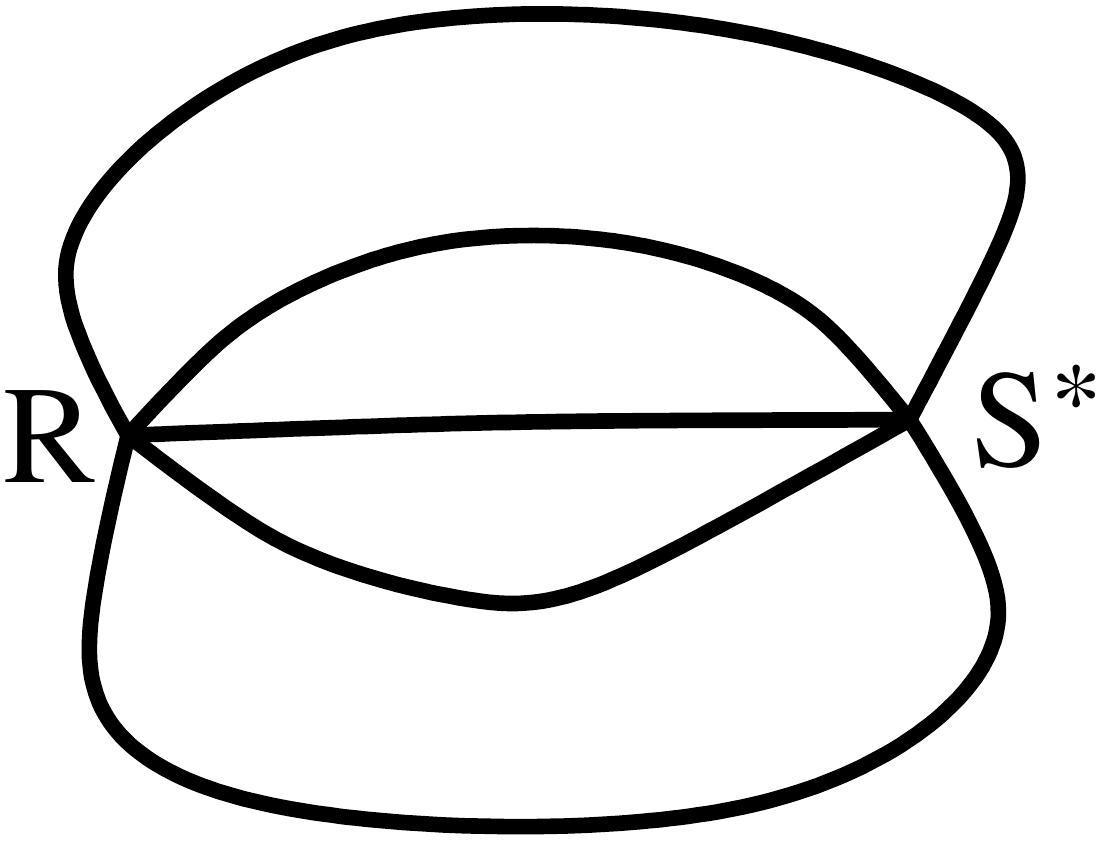} {1in}  $$
 is positive definite. (Illustrated here for $P_5$.) The internal discs of planar tangles have been reduced to points
 and the labels have been placed in regions near that point which correspond to the distinguished interval on 
 the boundary of the disc. The output discs of all tangles have been eliminated but are of course implicit in the diagrams.

The preceeding constructions for tensor networks work equally well for any $\mathcal P$.  One chooses any $Y=$\Y from
$P_3$ satisfying the isometry condition.
The direct system of Hilbert spaces is defined by
$A_Y(t)=\mathcal H_t=P_k$ if $t$ is a tree with $k$ leaves, and the $\iota_s^t$ are diagrammatic. The direct limit is again
denoted $\mathcal H$. Or $\mathcal H_{\mathcal P, Y}$ if necessary. (In fact all that is needed is an "annular" or
"affine" representation of the planar algebra in question-see \cite{gl,jo3,JR}-here we are just using the "trivial"
representation.)

We choose a planar algebra having the advantage that the $3$-box space is one dimensional. It is the planar algebra
for the quantum group $U_qSO(3)$ in its 3-dimensional representation. Alternatively it can be obtained from the
Temperley-Lieb algebra TL($\delta$) (\cite{Kff}) by doubling all strings and reducing by the JW idempotent in the 4-box space of TL.
Obtained in this way, its loop parameter is $d=\delta^2-1$. It is positive definite when $d=4cos^2\pi/n-1, n=6,7,8,\cdots$ or
$d\in \mathbb R,\geq 3$. We will call this planar algebra $\mathcal Q= (Q_n)$. $\mathcal Q$ is described in considerable
detail in \cite{MPS} (which actually gives a list of all the simplest planar algebras generated by a single 3-box).

\begin{remark} We should point out that in this case $\dim Q_1=0$ so the image of \Y in $Q_2$ is zero dimensional and
the spaces of the direct system really only begin at $Q_2$. Thus in forming a scale invariant transfer matrix from a 
solution of the ABC equation there is no constraint on $L_0$ and  $L_1$ which we can choose 
arbitrarily as "C"'s in solutions of the ABC equation and then extend to all of $\mathcal H_{\mathcal Q, Y}$ as in theorem \ref{existence}.
See remark \ref{moregeneral}

%On the other hand $Q_2$ has a privileged element given by
%a single string with two endpoints. We may refer to this vector as the "vacuum vector" and the following diagram of
%it, as an element of $Q_8$, might be useful in understanding the sequel, especially the inner product when applying
%periodic boundary conditions:

\end{remark} 

  There is an "a priori" solution of the ABC equation given by the braiding (which can be established by
  viewing $\mathcal Q$ as a cabled Temperley-Lieb planar algebra). We draw the picture:
  
  The braid solution $A=B=C= \vpic {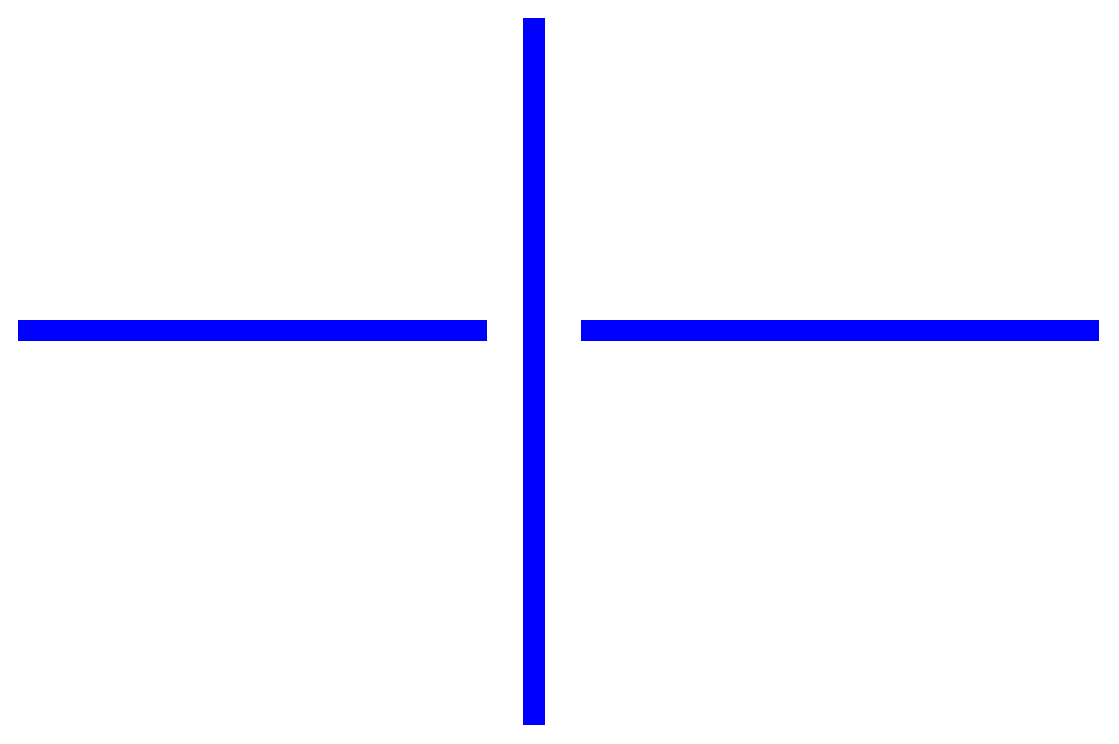} {0.6in}$:
  $$\vpic {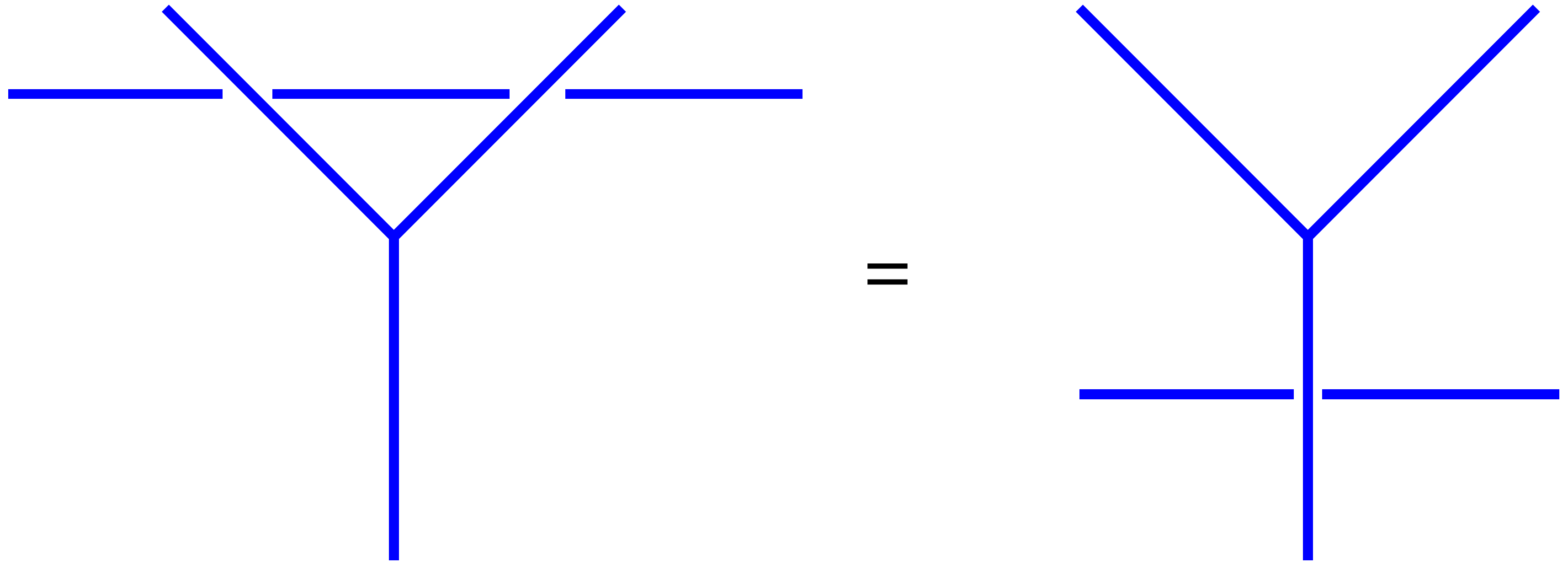} {2.4in}$$
  
  The trivalent vertex  $\Y$ is the unique (up to an irrelevant sign) non-zero self-adjoint element of $Q_3$.
  It is rotationally invariant and it is shown in \cite{MPS} that the following "skein" relations suffice to do all calculations in $\mathcal Q$ (along with positive definiteness):
  \5
  \vpic{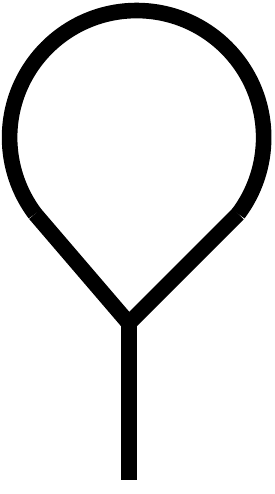} {0.1in} $=0$, \vpic{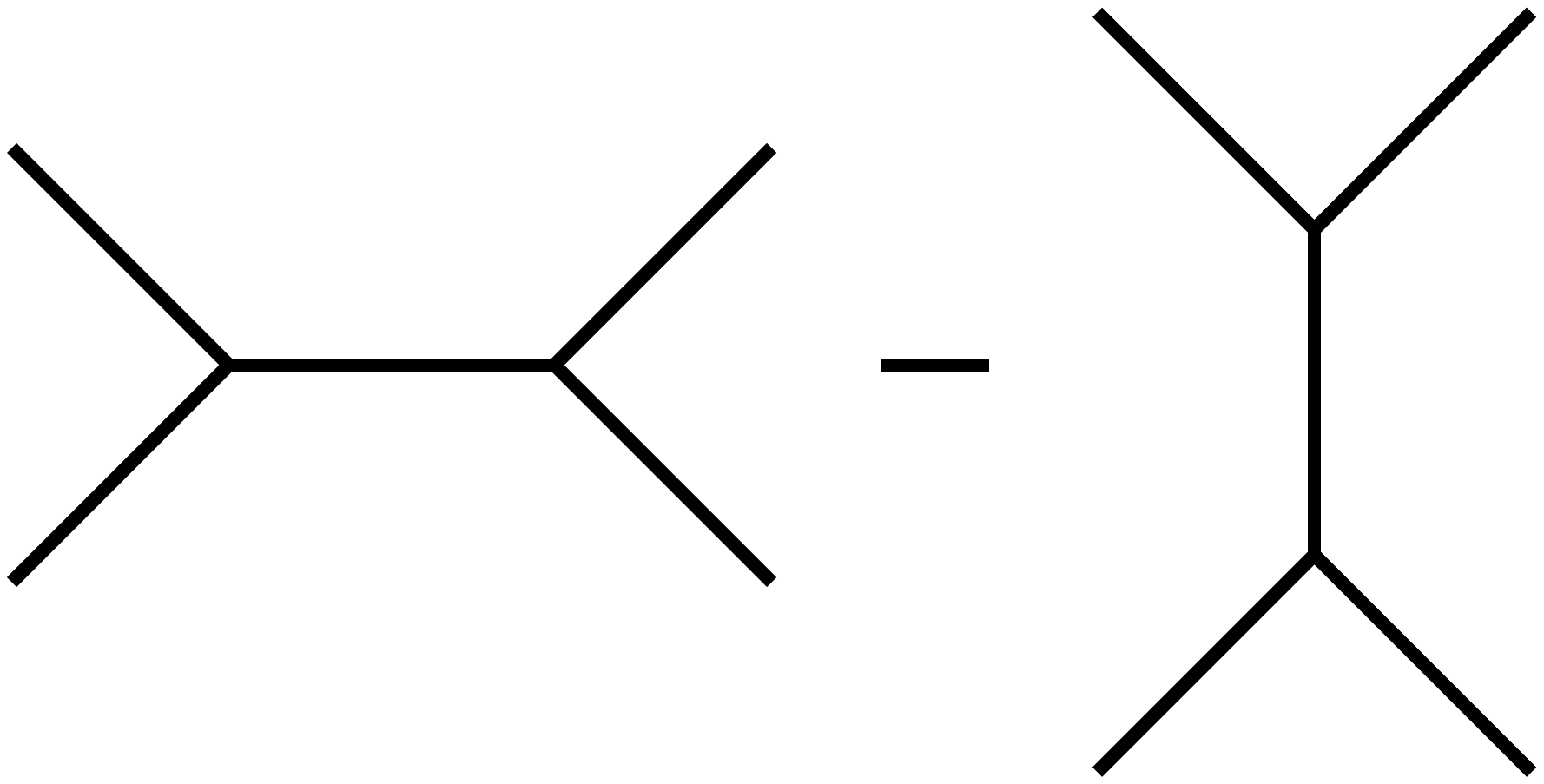} {0.5in} $=\frac{1}{d-1}( \vpic{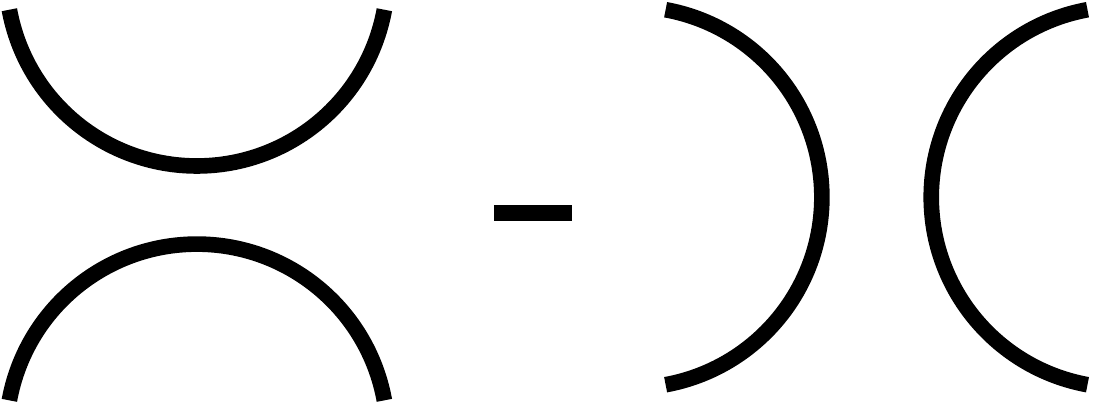} {0.5in} ) $ , and of course unitarity, \vpic{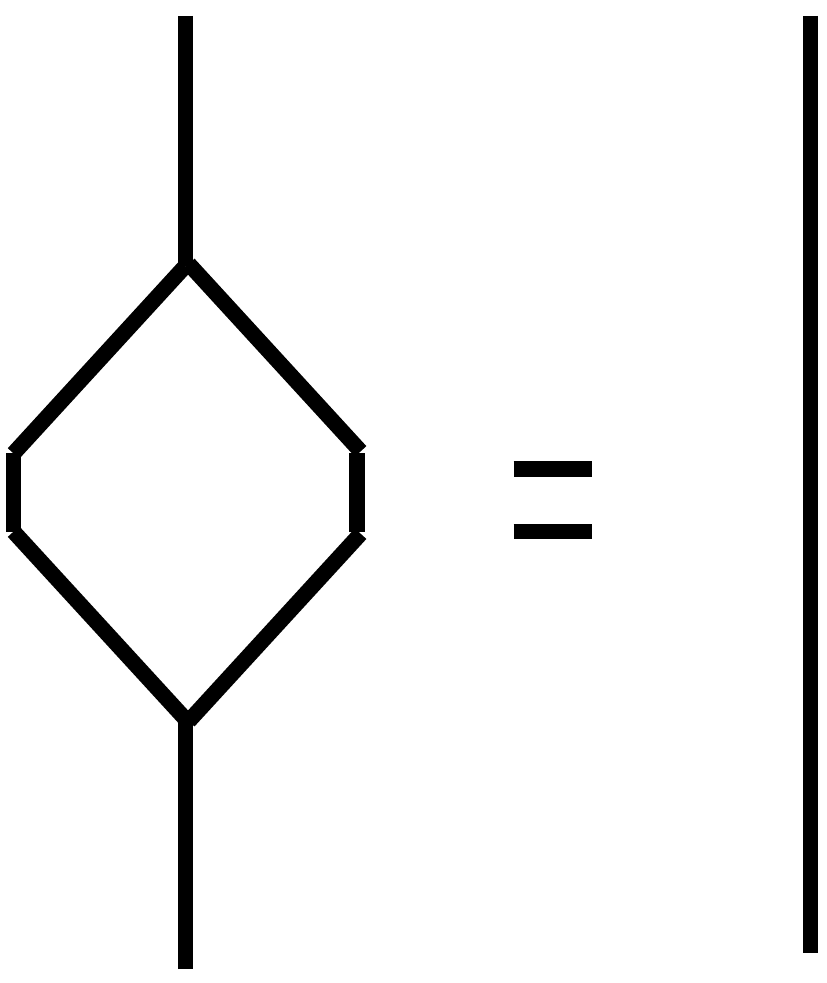} {0.3in} .

These structure constants are all real so we will be mainly thinking of real solutions to the ABC equation. Our answers
will apply to complex solutions though and the braid solutions are in fact complex for $d<3$.
  
  Since $\mathscr F^2=id$ on $Q_4$, we know that a scale-invariant transfer matrix can be formed from a solution of the ABC equation provided all powers $\alpha^n(A)$ (or $B$ or $C$) are invertible for both algebra structures. 
   If $X=p\BA+q\BB+r\BC$ we record the formula
   $$ \alpha(X)=-\frac{p}{q(p+q)}\BA+\frac{ d q r-p q - p r - q r }{q(d-1) (p + q) (q + d r)}\BB+\frac{(d-2) p+(d-1) q}{(d-1) q (p+q)}\BC$$
  We now solve
  the ABC equation using this formula for $\alpha$ to optimise reduction of the number of unknowns from 9 to 7.

   Let $A=a_1\BA+a_2\BB+a_3\BC$, $B=b_1\BA+b_2\BB+b_3\BC$ and $C=c_1\BA+c_2\BB+c_3\BC$, which we will represent
   in the more compact vector notation $A=(a_1,a_2,a_3)$, $B=(b_1,b_2,b_3)$ and $C=(c_1,c_2,c_3)$.
   Since $A$ and $B$ have to be in the domain of $\alpha$, it must be true, by our formula for $\alpha$, that $a_2$ and $b_2$
   are non-zero. Thus all solutions that we are considering come from solutions in which $a_2=1$ and $b_2=1$.

   Let's do a count of equations and unknowns.
The equations happen in the 5-box space which in this case is  (at most) 6 dimensional. It is spanned by
   $6$ explicit tangles, 5 of which are the rotations of a tangle with a single trivalent vertex \Y  and the other
   is any connected tangle with three instances of \Y. The inner product on $Q_5$ is positive definite since it is the restriction
   of the Temperley-Lieb inner product. So by taking inner products with the 6 spanning elements we obtain 6 equations
   in the 9 coefficients, now 7 after putting $a_2$ and $b_2=1$, of $A, B$ and $C$. So we have 6 equations in 7 unknowns
   and expect the solutions to depend on a single parameter.   
   
   Now let us turn to the equations in detail. Attaching a \vpic {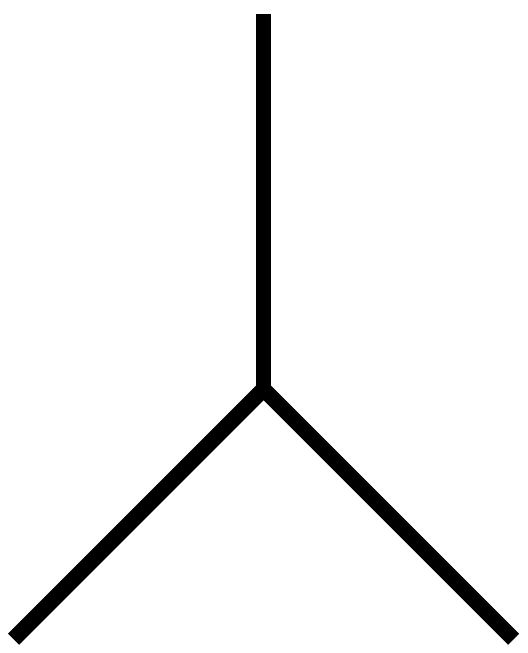} {0.2in} to the top of the ABC equation
   we note the appearance of a tangle from \cite{jonogo} which will be used frequently later. 
   \begin{definition} \label{renormal} The    renormalisation tangle below defines the nonassociative algebra structure $!$ on $Q_4$:\\ 
   $$x!y =\vpic {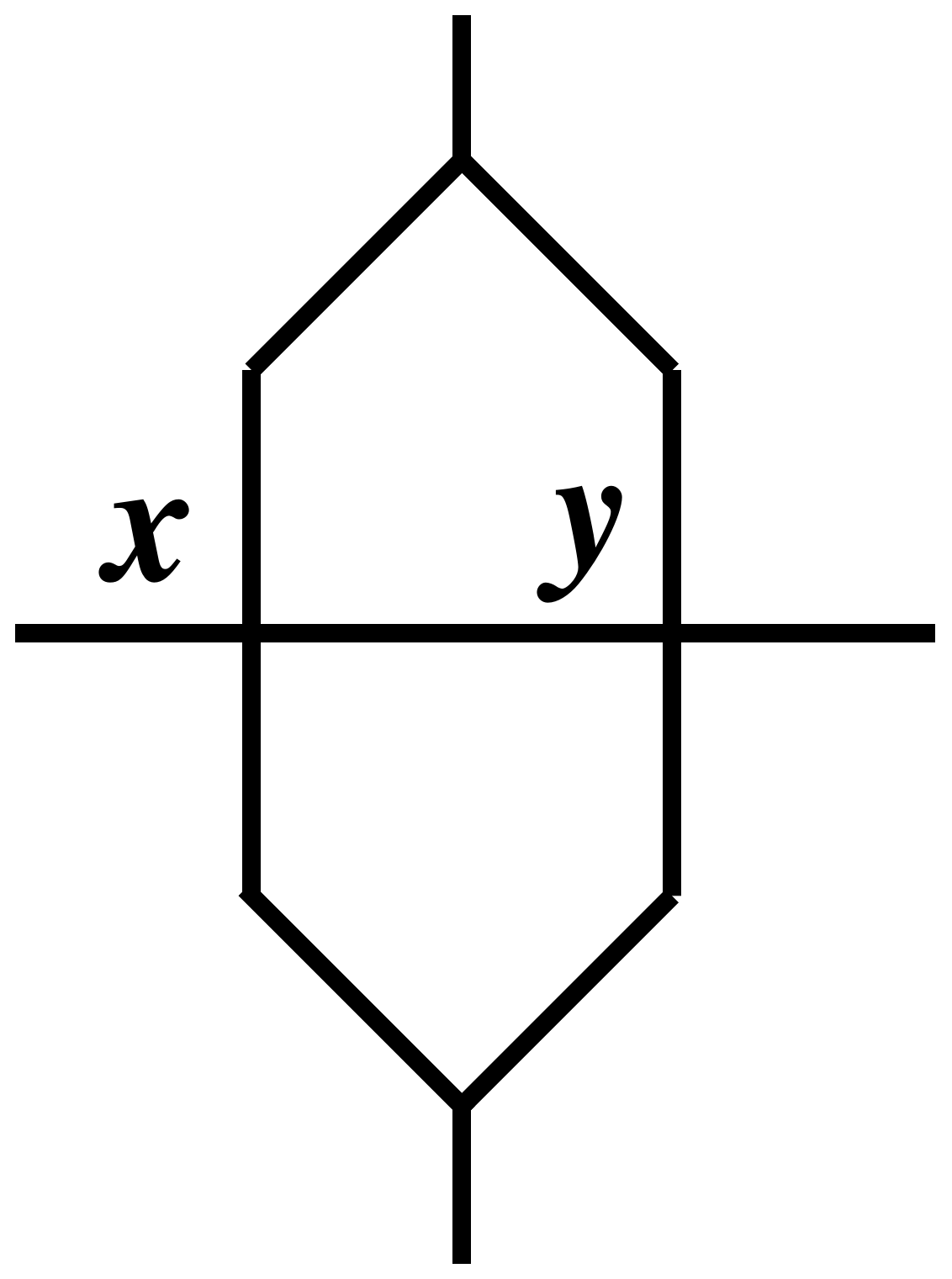} {1.5in} $$
   \end{definition}

   Three of the seven equations for $A,B$ and $C$  are consequences of $A!B=C$ so we are left with a system of  3 equations   in 4 unknowns for $A$ and $B$.
   So we expect  solutions for three of the variables depending on a fourth.
   
   Here are the three equations (before choosing $a_2=b_2=1$):\\
   $$(1) \quad c_1+c_2=(a_2+da_3)(\frac{d-2}{d-1}b_1+b_3)$$
   $$(2)\quad  \frac{d-2}{d-1}c_1+c_3=(a_1+a_2)(b_1+db_2+b_3)$$
   $$(3)\quad  (a_1+da_2+a_3)(b_2+db_3)+(\frac{d-2}{d-1}a_1+a_3)(b_1+b_3)=$$
   $$ \frac{1}{d(d-1)}\{\frac{d(d-2)}{d-1}a_1b_1+d(a_2b_2+a_3b_3+a_1b_3+a_3b_1+d(a_2b_3+b_2a_3))\}$$
   
   We illustrate with the first equation which comes from taking the inner product of both diagrams in the ABC equation
   with \vpic {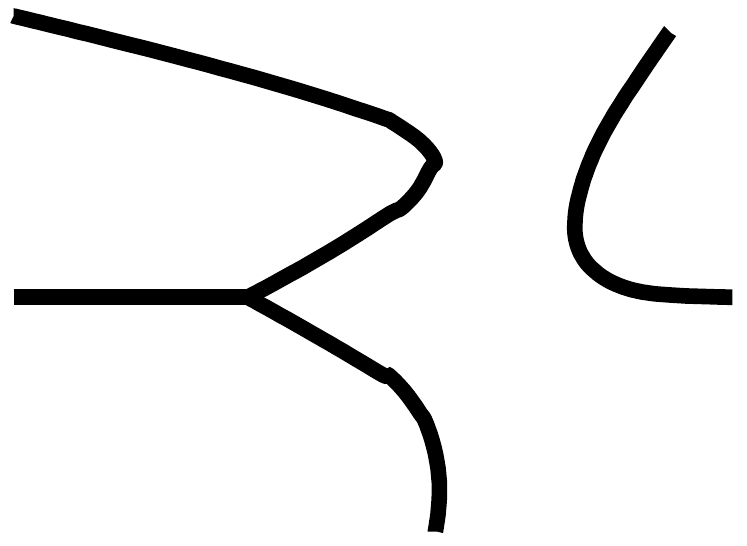} {1in}. Doing this we get the equality:\\
   \5
   \vpic {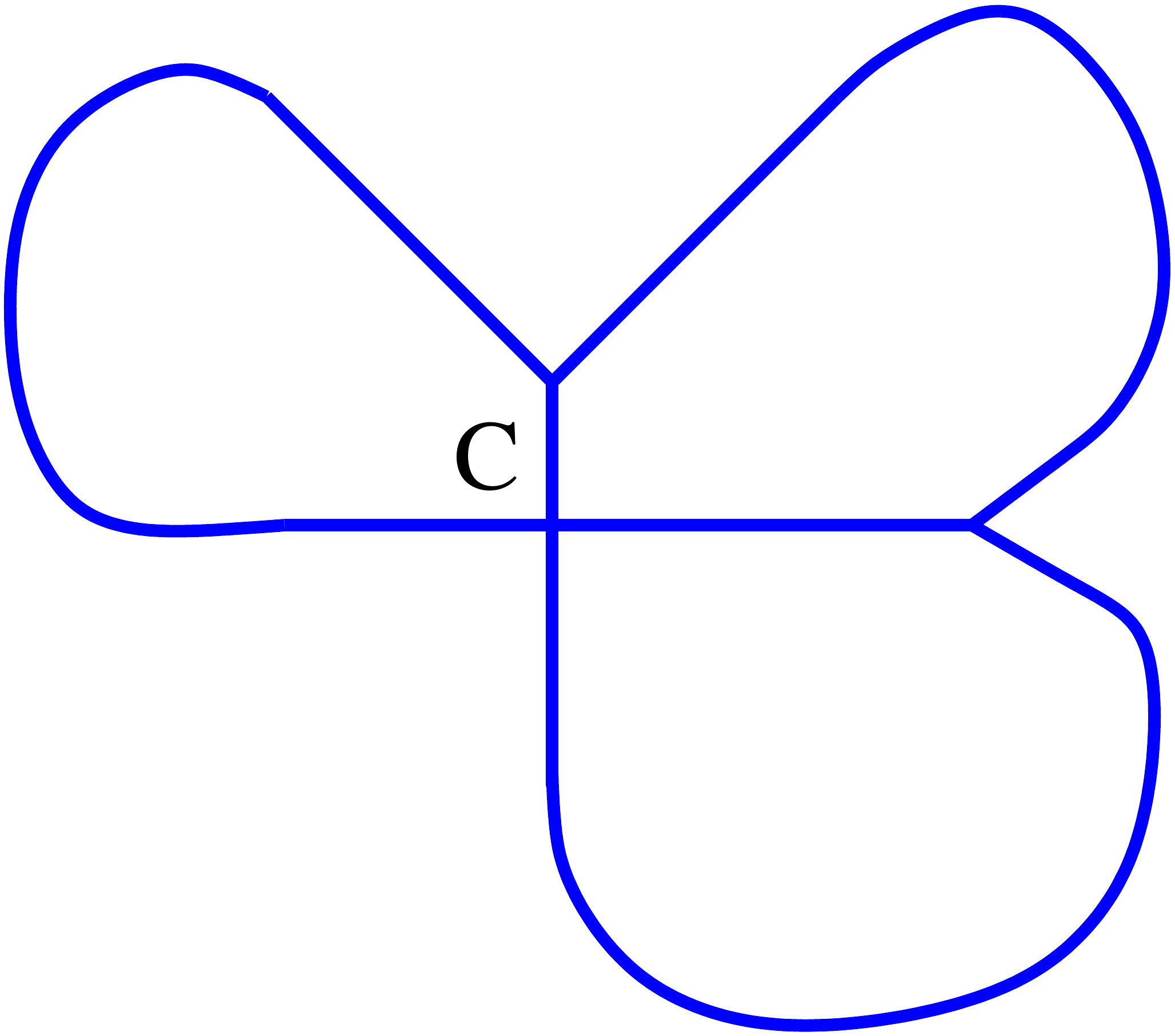} {1.3in}   $=$ \vpic {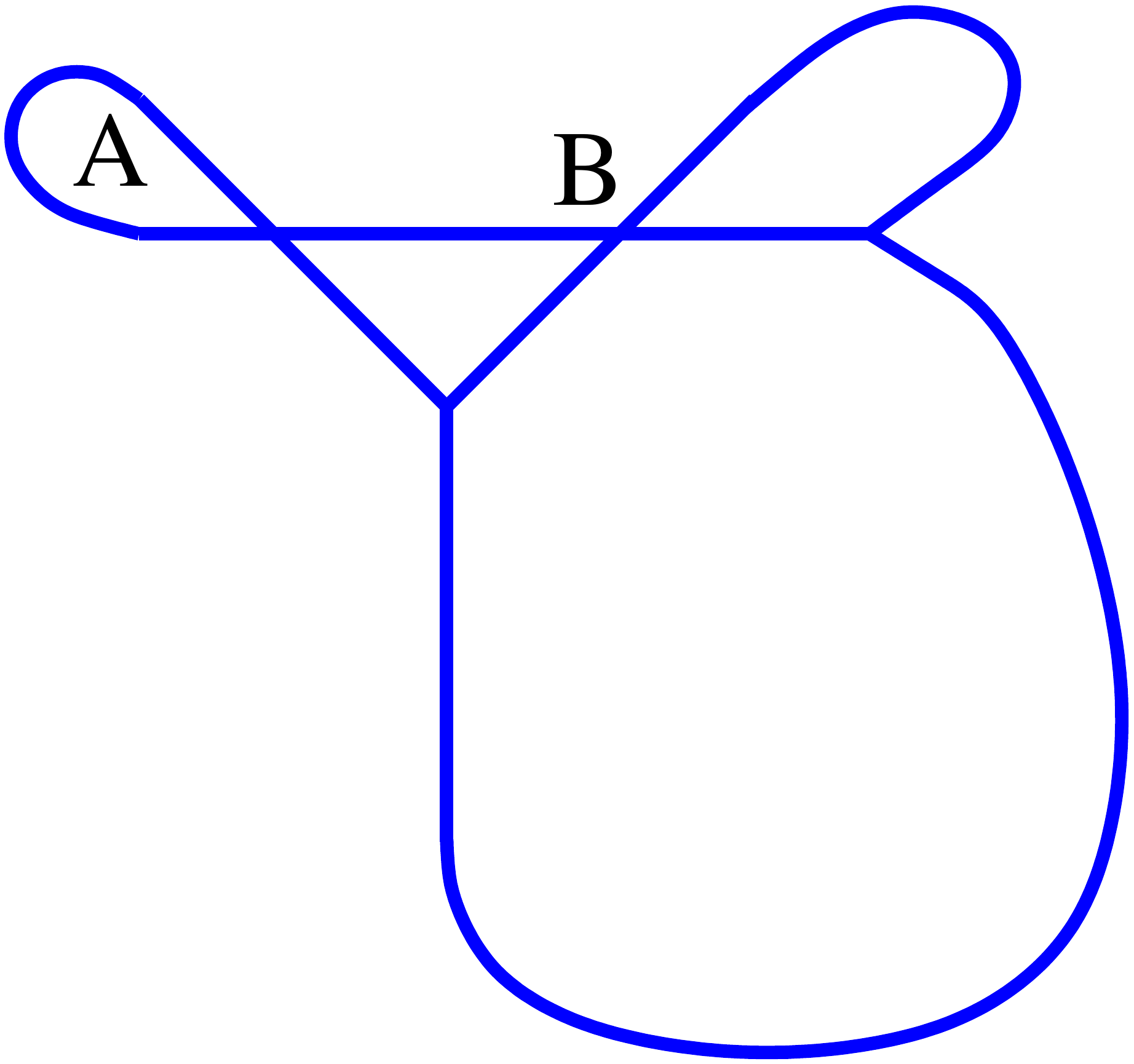} {1in} \\
   which readily yields equation $(1)$ by applying the skein relations of $\mathcal Q$.

       The equations
   are all of the  form ``$\sum_{i,j} y_{i,j}a_ib_j=z"$ for  three sets of constants $y_{i,j}$ and $z$ (remember $C=A!B)$. Thus in principle they
   are easy to solve.
   
   We have solved the equations for $A,B$ and $C$ in terms of $a_1$ which we call $a$, and found (for $d\neq 3$) the unique solution:
   
 \noindent $\displaystyle A=(a,1,\frac{(d-3)a}{a+d^2-3 d+2})$  \\
  \5
 \noindent  $\displaystyle B=(-\frac{ (d-1)^2(a+1)}{a},1,\frac{(d^2-4d+3)(a+1) }{a+d-1})$\\
   \5
   
 \noindent  $\displaystyle C=(\frac{(d-1)^2(a+1)  ((d-2)a+d-1)}{-a (a+d-1))},\frac{(d-1)(a+1) }{a},\frac{(a+1) ((d-2)a +d-1)}{-a})$ \\
      
      These equations are still somewhat messy. The key to further progress is to observe that the transformation $\alpha$ preserves 
the solutions. In fact 
$$\alpha(A)=const(\frac{(d-1) ( (d^2-3d+1)a+d^2-3d+2)}{((d-2)a +d-1)},1,\frac{((d^2-3d+1)a+d^2-3d+2)}{a})$$

where $\displaystyle const= \frac{a ( (d-2)a+d-1)}{(1+a) (d-1) ((d^2-3d+1)a+d^2-3d+2)}$

 so  the effect on the variable $a$ is the linear fractional transformation:

$$\sigma(a)=-\frac{(d-1) ( (d^2-3d+1)a+d^2-3d+2)}{((d-2)a +d-1)}$$

The case $d=3$ is special and we will deal with it in the next subsection.

So if we assume $d\neq 3$ and change variables to $$a=-\frac{(1+\omega^2)}{\omega^2}\frac{ (z+\omega^2)}{ z+1}$$
with $d=\omega+\omega^{-1}+1$, then
$$A=(-\frac{(1+\omega^2) (z+\omega^2)}{\omega^2 (z+1)},1,-\frac{(\omega-1) (z+\omega^2)}{\omega (\omega z-1)}),$$ 
 $$\alpha(A(z))=\frac{\omega^2 (\omega^3-z) (\omega^2+z)}{(\omega^5-z) (\omega^4+z)}A(\frac{-z}{\omega^3}),
  B=A(\frac{z}{\omega^2})\mbox{   and   } A=\frac{\omega(z-\omega)}{z-\omega^3}C(\omega z).$$

%see spectrum.nb
    
\begin{theorem}\label{finalsolution}Suppose $d\neq 3$. Let $A$, $B$ and $C$ satisfy equation \ref{ABC} for some $C(z)$. Then $(A,B,C)$ scales in 
the sense of \ref{scales} provided  $z\neq \pm \omega^n$ for any $n\in \mathbb Z$.
\end{theorem}
\begin{proof} The spectrum of $A$ is easy to work out by changing to the basis of minimal idempotents for $Q_4$. 
It is $$\{\frac{z-\omega^5}{\omega^2 (\omega z-1)},-\frac{z+\omega^4}{\omega^2 (z+1)},1\}$$
So provided $z \neq \pm \omega^n$ for any $n$,  the above formulae show that none of $A, B$ or $C$ nor any
power of $\alpha$ applied to them, has zero in its spectrum. \end{proof}

\begin{corollary}\label{precision}For any two complex numbers $z_0$ and $z_1$ satisfying $z_i\neq \pm \omega^n$ for $i=0,1$, there are 
scale invariant transfer matrices  $T(z_0,z_1)_t$ on $\mathcal H_{\mathcal Q,Y}$ such that, on $\mathcal H_{\Y}$ they are:
$$T(z_0,z_1)_{\Y}=\vpic {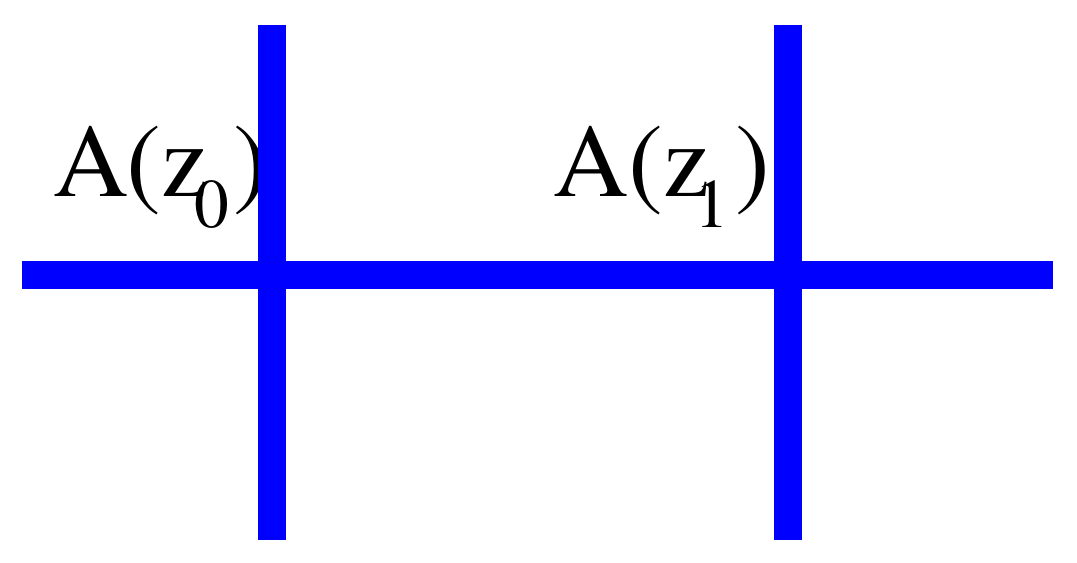} {1.5in}  $$
Moreover  $L_{0w}$ is of the form $\kappa A(\varepsilon \omega^p z_0)$ where $\kappa\in \mathbb C, \varepsilon\in\{\pm1\}$
 and $p\in \mathbb Z$ are all determined by $w$, and similarly for $L_{1w}$
\end{corollary}
\begin{proof} By the formula before theorem \ref{finalsolution} we can begin with either $A,B$ or $C$ and theorem \ref{existence}
shows how to extent to all the $\mathcal H_t$. Explicit formulae for $\kappa\in \mathbb C, \varepsilon\in\{\pm1\}$
 and $p\in \mathbb Z$ could be constructed from theorem \ref{existence} but we will only need the fact that they are determined
 by the leaf coded for by $w$.
\end{proof} 
We have thus  determined all values of $z$ for which this procedure gives  a scale invariant transfer matrix.
Note that a priori  we could have asserted that, for $d=4\cos^2\pi/n -1$ and $n=6,7,8,\cdots$ there is a non-empty open interval 
of real values of $a$ for which the semicontinuous limit transfer matrix exists. This is because the transformation $\alpha$ is then periodic 
so at most a finite set of values of $a$ can be bad.
\begin{remark}
\rm{There are only isolated values of $C$ for which $A=B$. Thus the scale invariant transfer matrix is not
spatially homogeneous for an interval's worth of spectral parameter values.}
\end{remark}

\subsection{Commuting transfer matrices.}

We are unable to think of an a priori reason why the solutions of the ABC equations should have anything to do
with the Yang-Baxter equation (\cite{Ba,Yang}). We will see, however that not only are the solutions we have obtained  a well known solution of YBE, but 
also the spatially inhomogeneous transfer matrices on the semicontinuous limit (on the circle-periodic boundary
conditions) commute for different values of the parameter. Note that this is not automatic-given a solution $R(\lambda)$ of the
YBE , one may not choose $\lambda_i$ arbitrarily and expect $T(R(\lambda_1),R(\lambda_2),\cdots,R(\lambda_n)) $ to commute with
each other.

 % hidden note-these formula can be found in ABCinbraidbasisMay2017.nb see also spectrum.nb
We will identify our solutions with the ``Izergin-Korepin" model. In order to do this we will express the solutions of the
ABC equation that we found in section 2.3 in terms of another basis of $Q_4$- the "braid basis" consisting of a crossing, its inverse
and the identity. If we set $$R=(-\omega - \omega^{-1})\BA+ \omega^{-1}\BB+(\omega - 1)\BC$$
(recall $d=1+\omega+\omega^{-1}$) then the braid equation holds and moreover $R+R^{-1}=-(w+1/w)(\BB+\BC)$
so we are dealing with representations of the BMW algebra \cite{BW,MuJ}.

If we rewrite $A(z,\omega)$ from subsection \ref{concrete} in the braid basis we obtain:
$$A=\frac{1}{(1+z) (\omega-1)}R+z \frac{1+\omega+\omega^3+\omega^4}{(1+z) \omega (z+\omega^2)}id+\frac{z \omega}{(1+z)(1-\omega)}R^{-1}$$

Multiplying by $(1 + z)(\omega - 1)\omega^{-3/2}(1 +\frac{ \omega^2}{z})$ and putting $z=\frac{\omega^2}{x}$ one gets 
$$A'=\frac{1+x}{\omega^{3/2}}R+\frac{(\omega^2-1)(\omega^3+1)}{\omega^{5/2}}id-\frac{(1+x) \omega^{3/2}}{x}R^{-1}.$$
Putting $\omega=e^{-\eta/2}$ and $x=-e^{-\lambda}$ and $R(\lambda, \eta)=A'$, one obtains exactly twice  formula $(III.9)$ for $\check R(\lambda,\eta)$
in \cite{KMN}. This is the Izergin-Korepin R matrix \cite{IK}. It thus satisfies the Yang-Baxter equation (\cite{Ba,Yang}):\\
\5
\mbox{                       }\qquad\qquad  \vpic {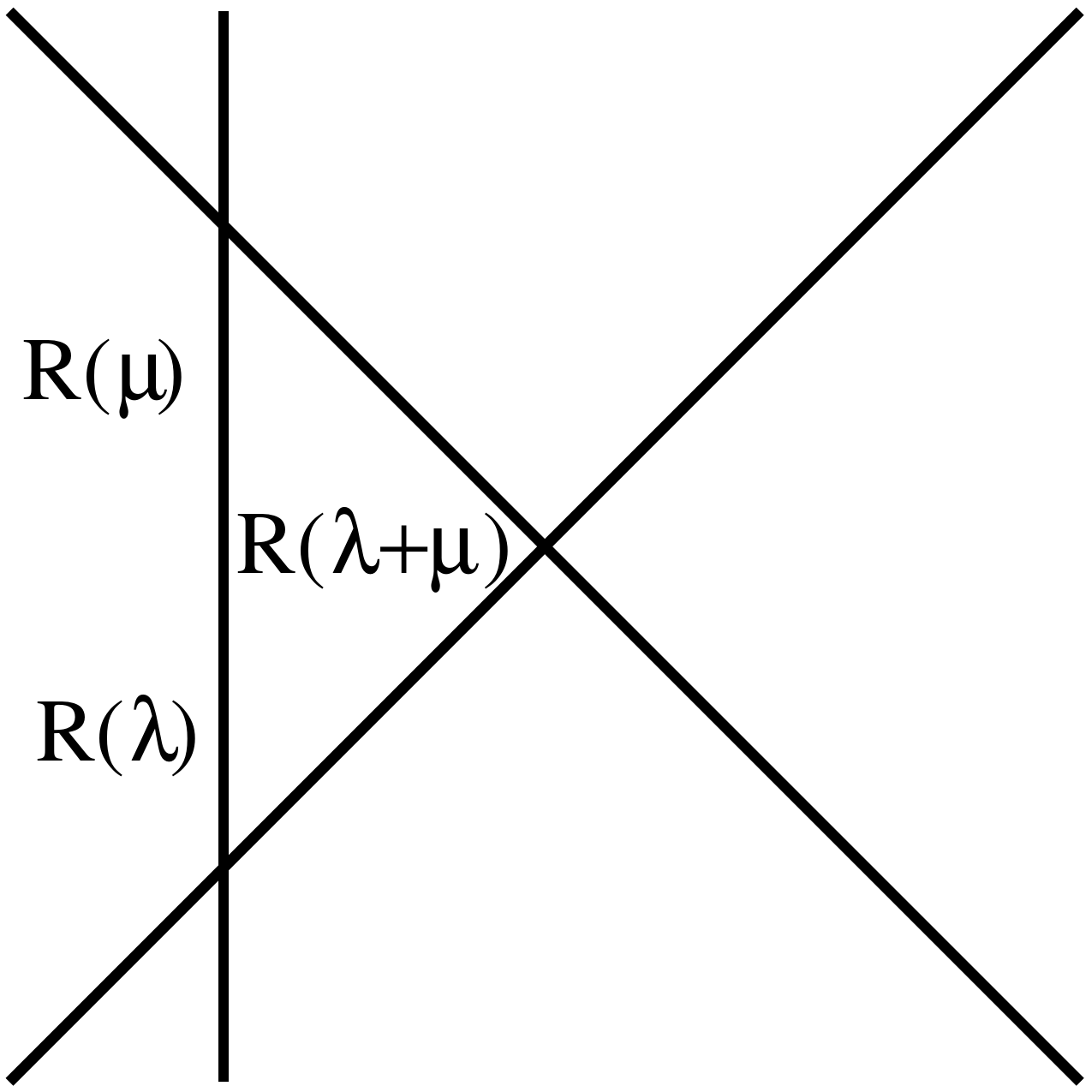} {1.3in}  $=$   \vpic{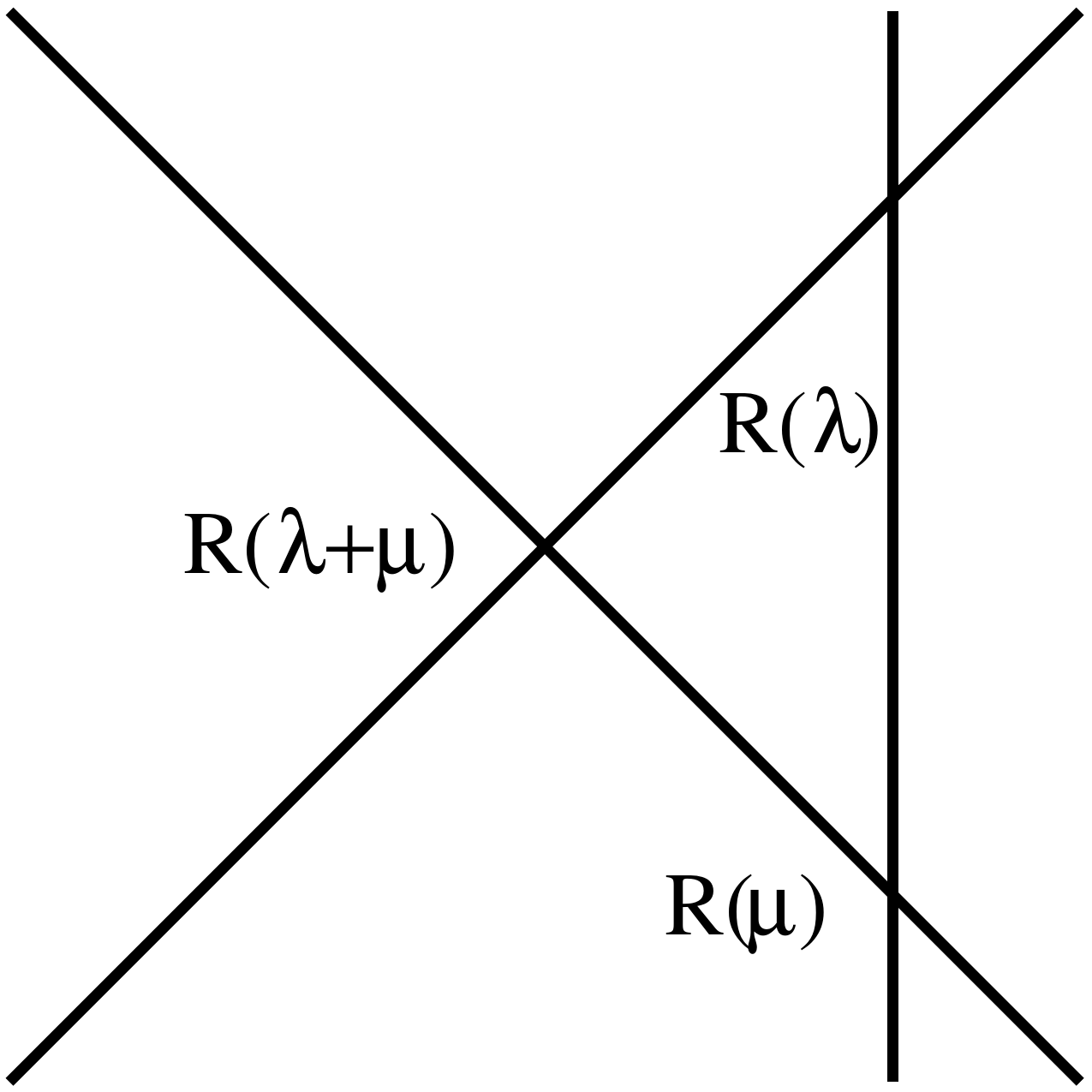} {1.3in}
\5
We will use this in the following form which is easily deduced from the previous equations:\\
\5
\mbox{                       }\qquad\qquad  \vpic {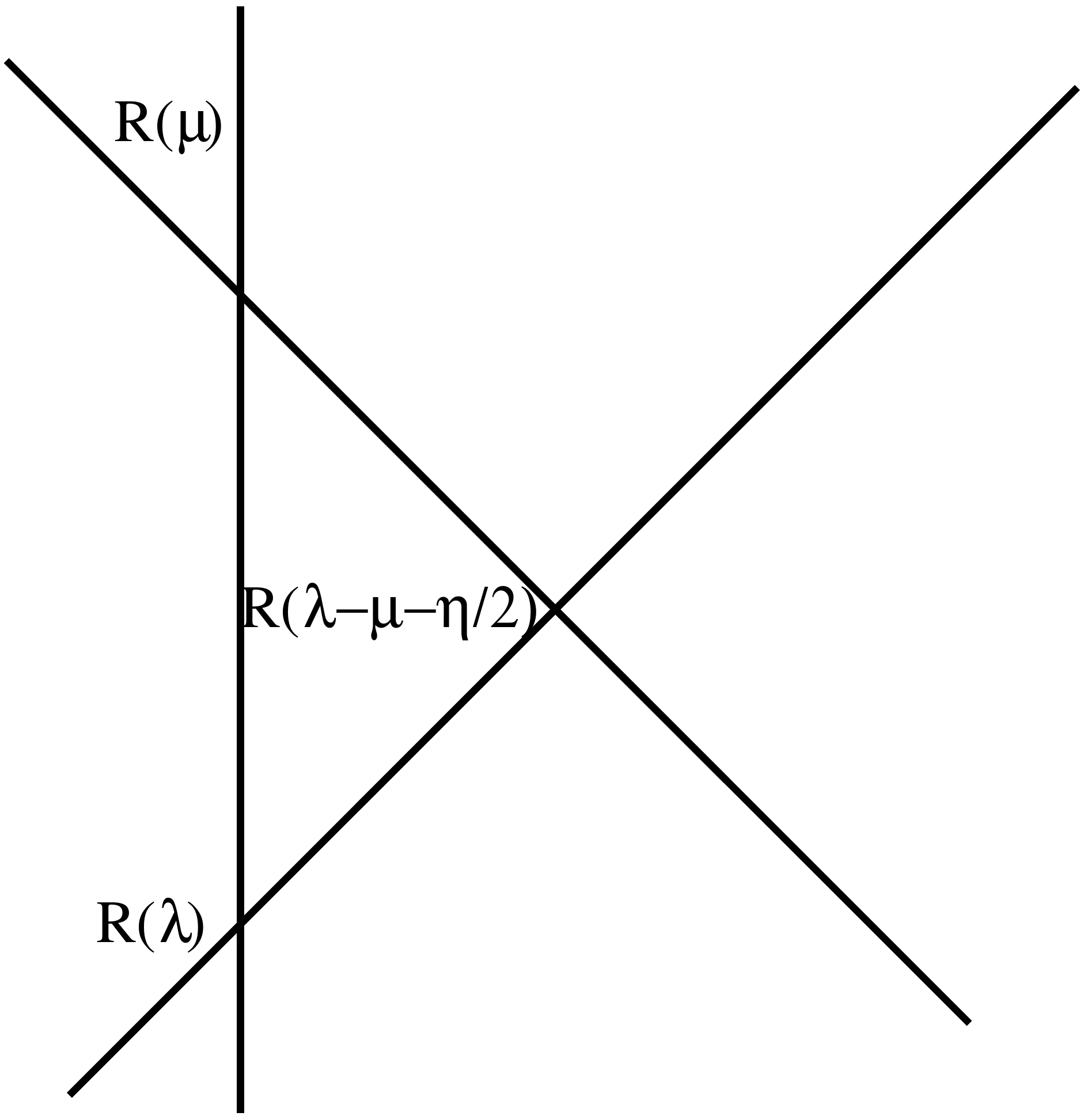} {1.8in}  $=$   \vpic{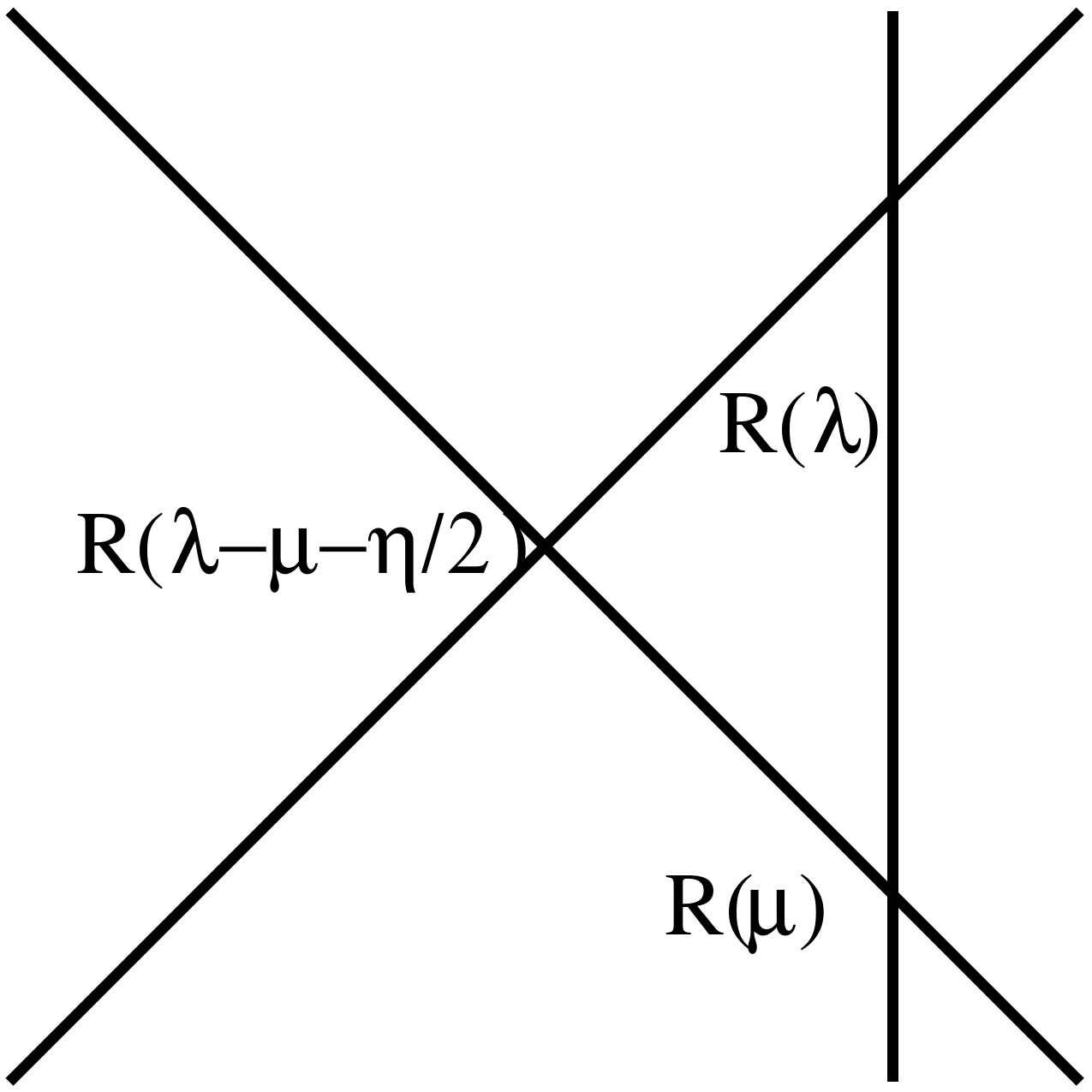} {1.3in}
\5

 \iffalse  At this stage we need to be a bit more careful about what we mean by the transfer matrix with periodic boundary 
   conditions in the case of $\mathcal Q$ for which $dim Q_1=0$ so one cannot use a single $\Y$ or more generally
   a tree to embedd a "vacuum vector" in a higher $Q_n$ since there is no vacuum vector to embedd. But since
   $dim Q_2=1$ there is an essentially unique choice of vacuum vector in $Q_2$ which is a single curve joining two boundary
   points of its disc. This corresponds to a spin chain with two spins. To obtain, for instance, the doubled spin chain with
   four spins one simply attaches a $\Y$ to each of these boudary points
   to obtain the vacuum vector seen in $Q_4$. Continuing in this process one obtains the following picture for the vacuum
   in 
 \fi

 \begin{definition} For all  $\lambda$ and $t$ put $T^{\lambda}=T(-e^{-\lambda},-e^{-\lambda})_t$ 
 (see \ref{precision}).
 \end{definition}
 \begin{theorem}For all $t\in \mathfrak T$ and all $\lambda$ and $\mu$,  $$T^{\lambda}_tT^{\mu}_t=T^{\mu}_tT^{\lambda}_t$$
 \end{theorem}
\begin{proof}

We begin by reviewing
 the well known diagrammatic argument for commuting transfer matrices with periodic boundary conditions and 
spatially homogeneous dependence on the spectral parameter (see \cite{J22}). One starts with
$$T(\lambda)T(\mu)\vpic {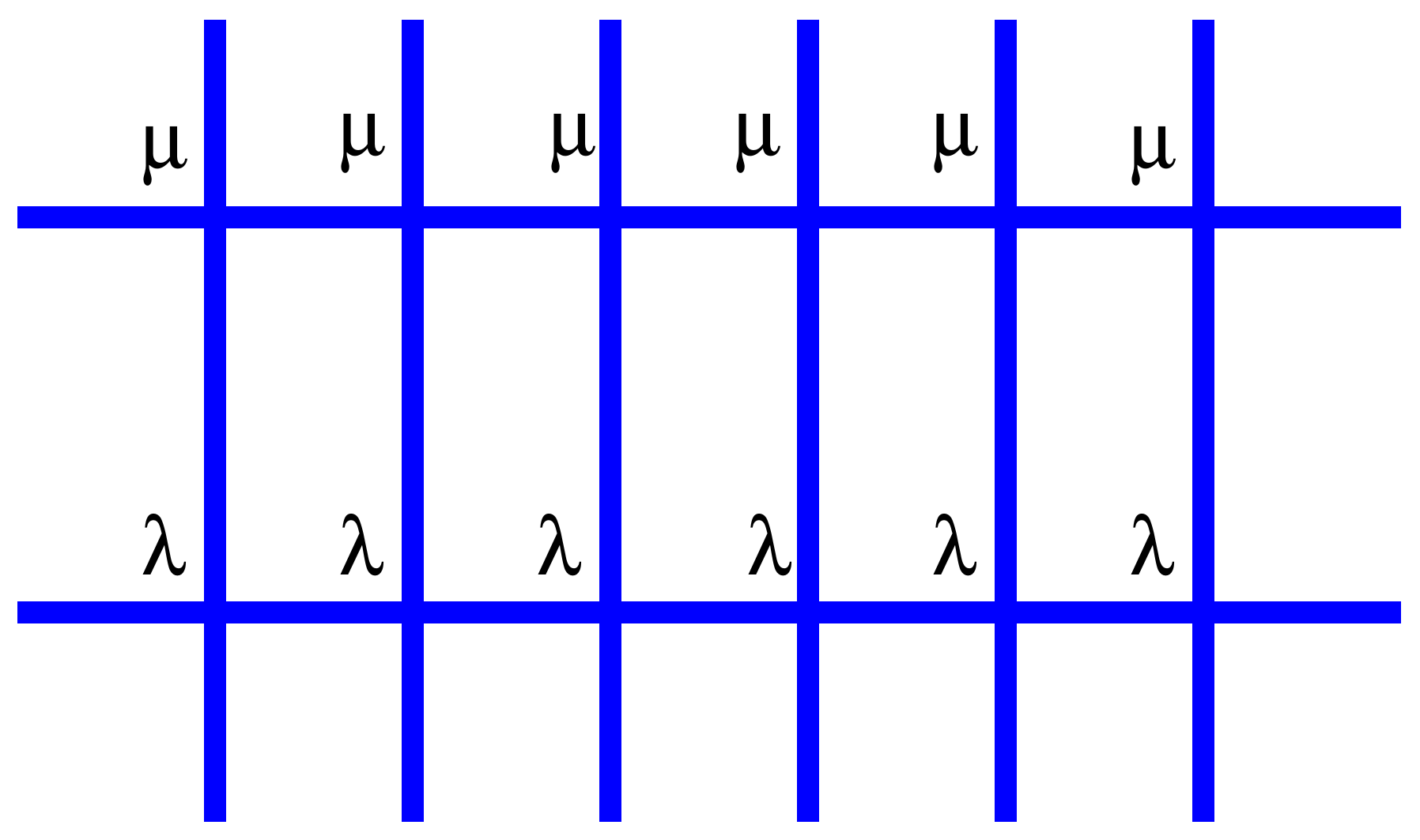} {2in} $$
Then one attaches $R(\lambda-\mu-\eta/2)$ and its inverse (for horizontal multiplication) to
the right of the picture so as not to change the operator $T(\lambda)T(\mu)$. 
Then successively apply YBE, moving $R(\lambda-\mu-\eta/2)$ to the left one step at a time, each time interchanging
an $R(\lambda)$ with an $R(\mu)$. Thus an intermediate step would look like:\\

$$\vpic {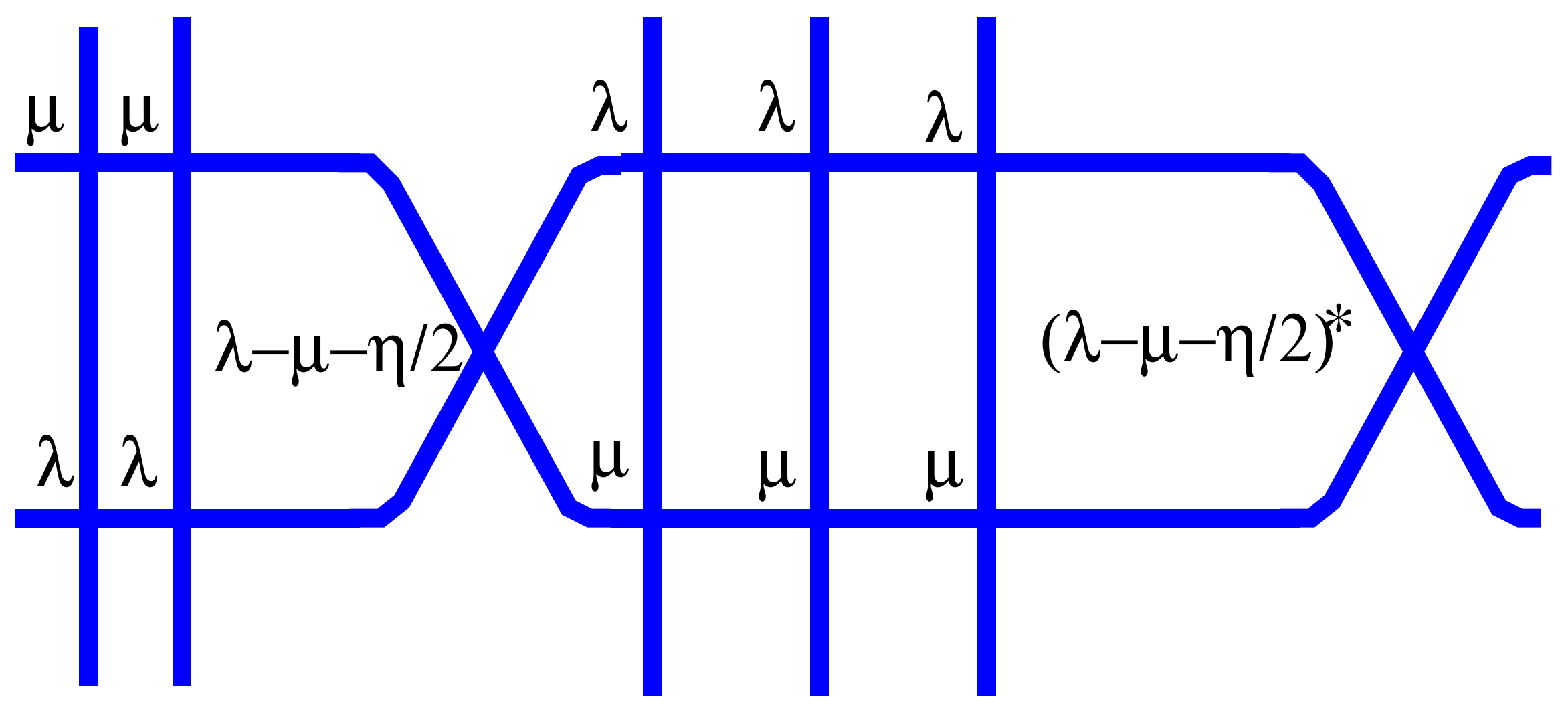} {3in} $$

where we have used a * to indicate the spectral parameter value corresponding to the inverse.

When $R(\lambda-\mu-\eta/2)$ has reached the left side of the picture it meets its inverse because of the periodic
boundary condtions, and disappears, to leave
$$T(\mu)T(\lambda)=\vpic{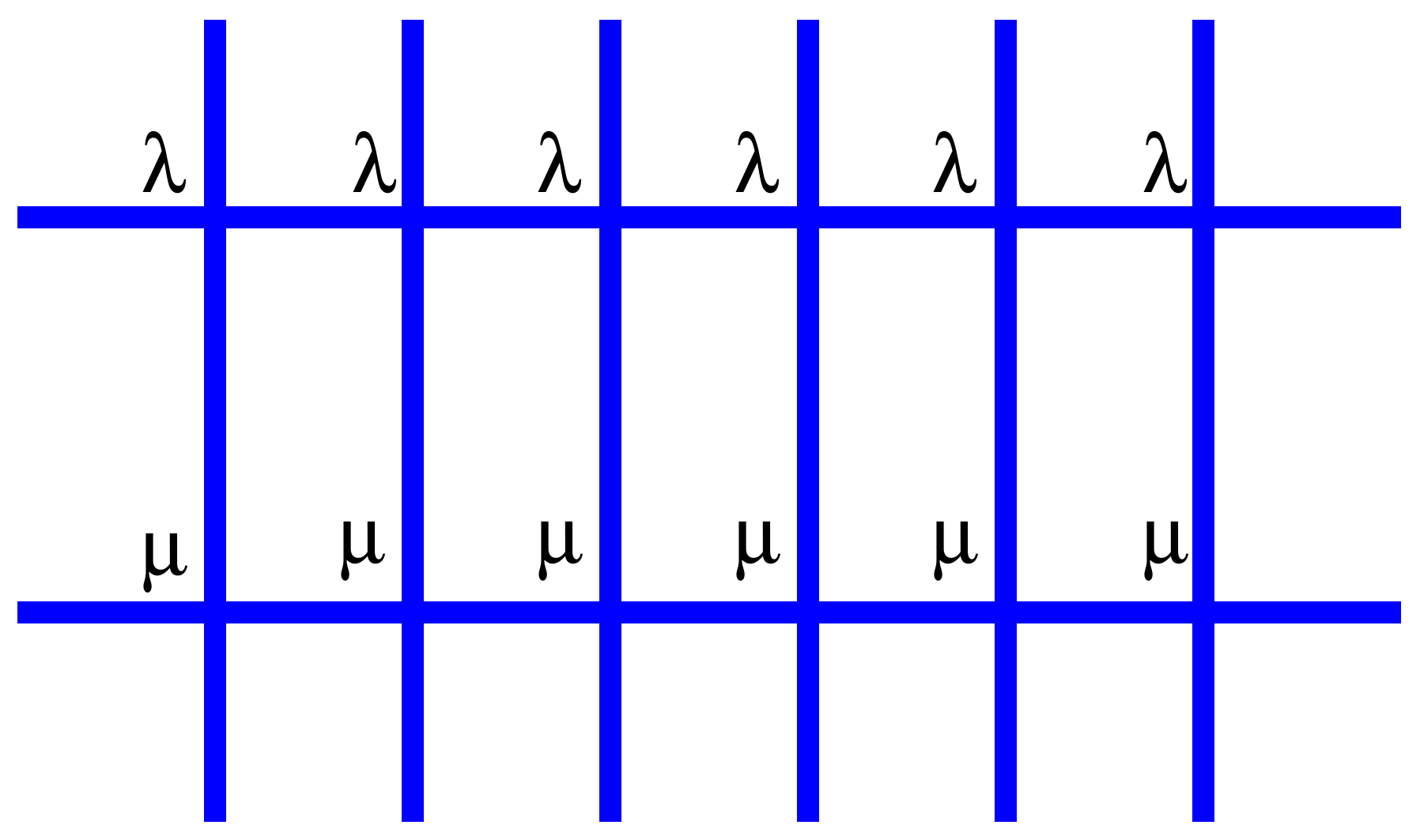} {2in}  $$

This argument goes through almost without alteration for any $T^\lambda_t$ and   $T^\lambda_t$. As
the $R(\lambda-\mu-\eta/2)$ goes to the left, it will meet the leaf coded for by the word $w$ (on $\{0,1\}$), the pair 
$\vpic {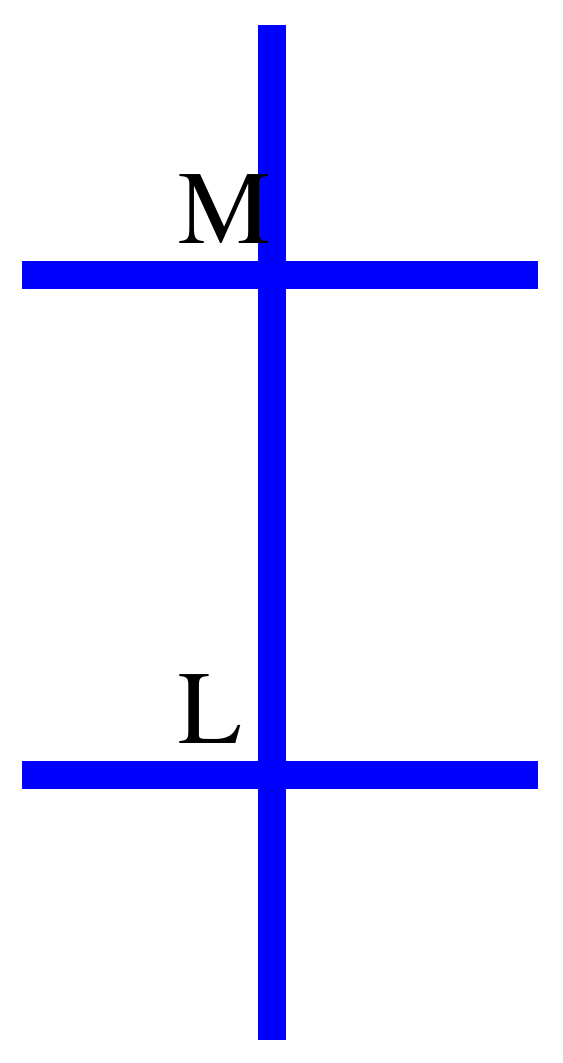} {0.6in} $ where $L$ is, by corollary \ref{precision}, $\kappa R(\lambda+\phi)$ and $M$ is $\kappa R(\mu+\phi)$
for some $\kappa$ and $\phi$ depending only on $w$. The $\kappa$ factors  are the same on both sides of the YBE and
the $\phi$'s cancel in the YBE so that $R(\lambda-\mu-\eta/2)$ goes past the pair, swapping $L$ and $M$.

\end{proof}

\subsection{The case $d=3$, i.e.  $\omega=1$}  Note that in this case the planar algebra $\mathcal Q$ is the tensor
planar algebra based on a 3-dimensional Hilbert space so the vector spaces $\mathcal H_t$ are ordinary spin chains.

If we solve equations $(1)$ and $(2)$ of the previous subsection for $a_1$ and $a_3$ and substitute in equation $3$ we
find that the solution curve for $d\neq 3$ degenerates into a line pair $b_1=-2$ and $b_3=1$. These two lines are 
interchanged by $\alpha$. Here are the two solutions which we parametise by $v$ and $w$ to avoid confusion:\\
$$A_1=(v,1,0),\quad B_1=(-4\frac{1+v}{v},1,0), \mbox{    and    } C_1=2\frac{1+v}{v}(-2,1,-(1+\frac{v}{2}))$$
and
$$A_2=(-2,1,w),\quad B_2=(-2,1,\frac{w}{w+1}), \mbox{    and    } C_2=(-2\frac{2w+1}{w+1},1,0)$$

The transformation $\alpha$ is a little complicated with this parametrisation so we introduce the changes of 
variables $$v(s)=\frac{6s-2}{-3s+2}\mbox{  and  } w(t)=\frac{1}{3t-4}.$$
Then $\displaystyle \alpha(A_1(s))=\frac{3s-1}{3s}A_2(s+1)$ and $\displaystyle \alpha(A_2(t))=\frac{3t-2}{3t-1}A_1(t)$.
So $$\alpha^2(A_1(s))=\frac{3s}{3s-1}\frac{3s+1}{3s+2}A_1(s+1).$$

We have the following other relations among the $A$'s, $B$'s and $C$'s:

\begin{enumerate}
\item $\displaystyle B_1(s)=A_1(\frac{3s+1}{3})$
\5
\item $\displaystyle C_1(s)=\frac{3s-1}{3s}A_2(\frac{3s+2}{3})$
\5
\item $\displaystyle B_2(t)=A_2(\frac{3t+1}{3})$
\5
\item $\displaystyle C_3(t)=A_1(\frac{3t-1}{3})$
\end{enumerate}

Thus any one of $A_1,A_2,B_1,B_2,C_1,C_2$ determines all the others via linear fractional transformations of
the parameters, and multiples which are themselves linear fractional transformations of the parameters. 

We finally turn to the question of when the solution $A_1,B_1,C_1$ scales in the sense of
\ref{scales}. This will be true provided two conditions are satisfied:
\begin{enumerate}
\item The constant factors introduced when applying powers of $\alpha$ must never be zero or infinity.
\item  The spectrum, for both multiplication and comultiplication, of all powers $\alpha^n$ of all $A$'s $B$'s and $C$'s
must not contain zero.
\end{enumerate}

The three values of the spectrum of an element $(x,y,z)$ are determined by linear functions with constant coefficients
of $x$, $y$, $z$. Since applying $\alpha$ just adds $1$ to the argument $s$ We see that there are a finite number of linear fractional transformations $\frac{\beta_i s+\gamma_i}{\delta_i s+\epsilon_i}$ and constants $\kappa_i$ so that
$A_1(s),B_1(s),C_1(s)$ scales provided $$\frac{\beta_i (s+n)+\gamma_i}{\delta_i (s+n)+\epsilon_i}\neq \kappa_i,$$ which
can be written $s\notin z_i+\mathbb Z$ for any $n$ and some finite set of numbers $z_i$. 
Clearly there are non-empty open intervals  of $s$-values for which this is true. The same argument applies to 
$A_2(s),B_2(s),C_2(s)$,

We have thus proved, by theorem \ref{existence} :
\begin{theorem} Suppose $d=3$. There are non-empty open intervals of $s$ values for which the transfer matrix
$$\vpic {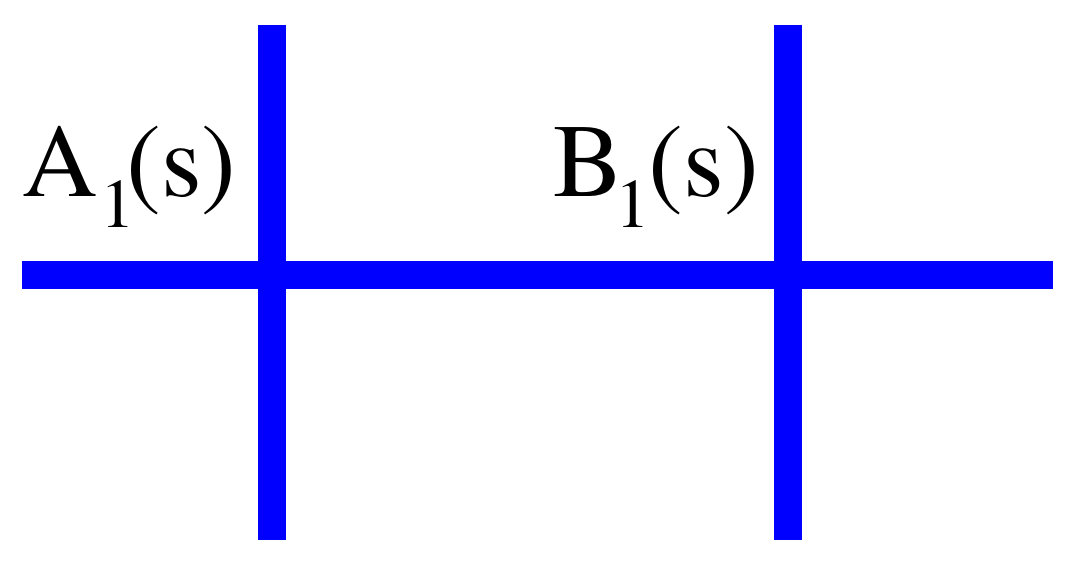} {2in} $$ on $Q_2=\mathcal H_{\vpic {Y} {0.071in} }$ extends to a scale-invariant transfer matrix on 
$\mathcal H_{\mathcal Q, Y }$. (Periodic boundary conditions as usual.)
\end{theorem}

  \section{Spatially homogeneous scale invariant transfer matrix and Hamiltonian via quadratic forms.} 
  \subsection{Generalities.} It is easy to argue on
  physical grounds that scale invariance for an observable should be expressed in terms of its expected values. This leads
  immediately to the notion of weak scale invariance in definition \ref{scaleinvariancedef}. The word "weak" is 
  unfortunate in this context  but is in universal use in functional analysis to describe properties determined by matrix
  coefficients rather than an operator as a whole. A strong motivation for considering weakly scale invariant transfer matrices 
  is that they arose inevitably in \cite{jonogo} in the calculation of the correlation of a state with its rotation by a single lattice site.
  Indeed our construction of weakly scale invariant transfer matrices will allow us to deepen the study of the behaviour of the
  rotation by one lattice site.
  
  We immediately give the criterion for weak scale invariance, invoking the algebra structure $x!y$ of definition \ref{renormal}.
  Let $L, t$ and $w$ be as in definitions  \ref{codetree} and \ref{transfermatrix}.
 
 \begin{proposition} The map $t\mapsto T^L_t$ defines a weakly scale invariant operator on  the direct system $A_Y$ provided 
 $$ L_{w0}!L_{w1}=L_w$$ for all $w$.
 \end{proposition}
 \begin{proof}
 This follows just as in proposition \ref{scaleinvariance}
 \end{proof}
 
 At the risk of labouring the point, here is the picture of the equation $$ L_{w0}!L_{w1}=L_w$$ with $A=L_{w0}, B=L{w1}$ and $C=L_w$:
 $$\vpic {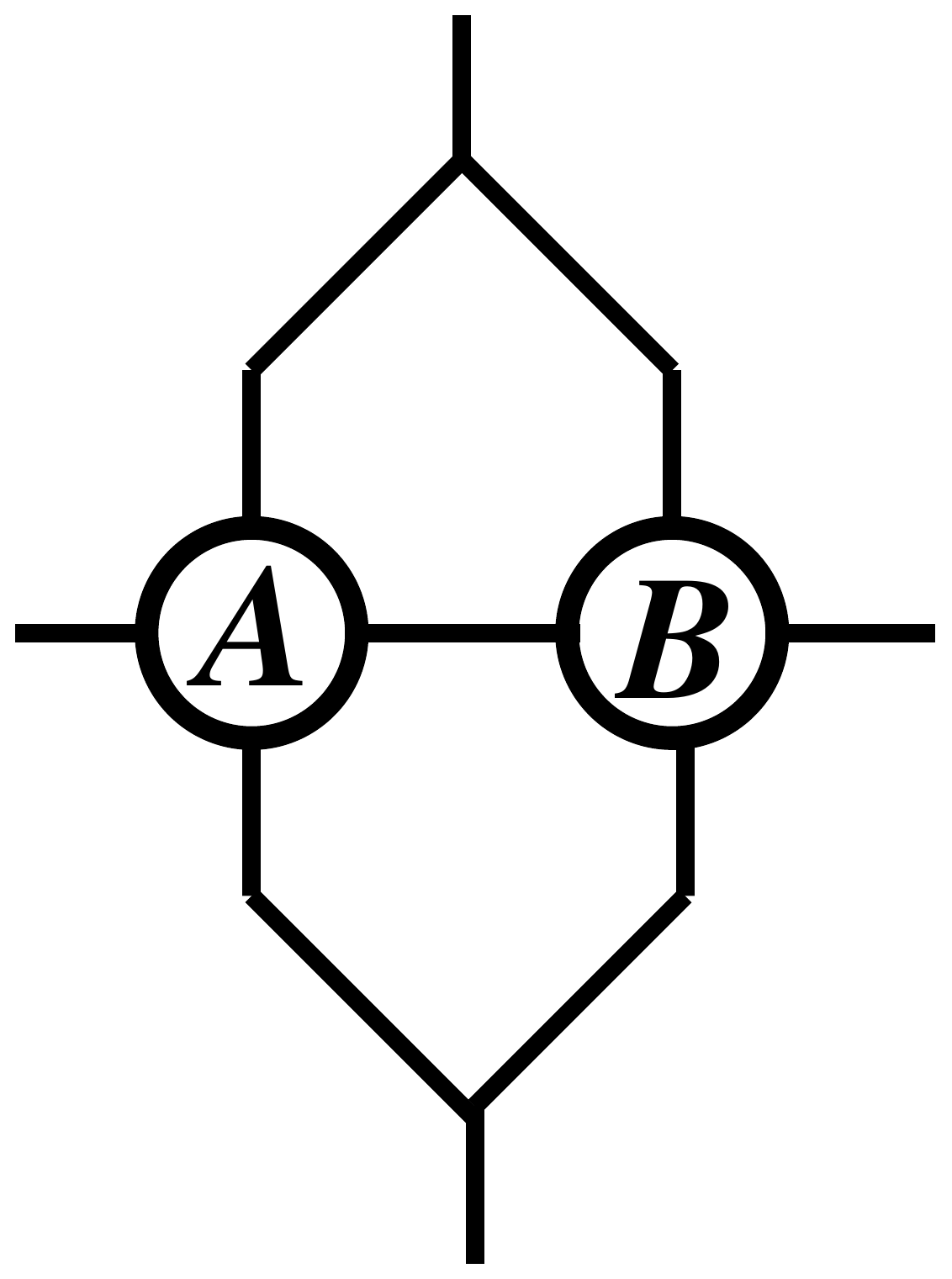} {1.5in} \qquad = \qquad \vpic {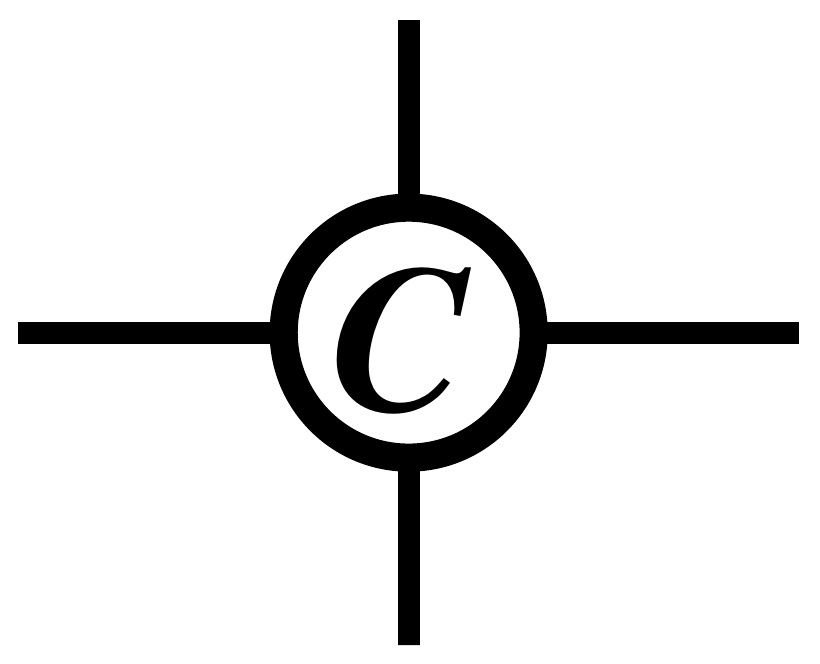} {1.2in} $$
 
 This makes it obvious that a solution of the ABC equation provides one of the equation $A!B=C$ (scale invariance implies
 weak scale invariance). But there are in general many more solutions and weakly scale invariant transfer
 matrices can exhibit quite arbitrary dependence on the lattice point. We shall thus focus on  \emph{ spatially homogeneous ones} which
 are  provided by solutions of the equation $X!X=X$. 
 \begin{definition} If $\mathcal P$ is a planar algebra and $Y=\Y\in P_3$  we call the quadratic map
  $\mathcal R:P_4\rightarrow P_4$,  $$\mathcal R(X)=X!X$$ the \emph{renormalisation dynamical system} of $\mathcal P,Y$.
 \end{definition}
 
 \begin{corollary} If $X=\{X_i, i=1,2,3,\cdots\}$ is a sequence of elements of $P_4$ obtained by back-iterating $\mathcal R$, i.e. such
 that $\mathcal R(X_{i+1})= X_i$, then we get a weakly scale invariant transfer matrix by putting
 $$L_w=X_{length (w_{\mathcal l})}$$ for all leaves $\mathcal l$ of all trees. 
 \end{corollary} 
 In this case the sesquilinear form is \emph{spatially homogeneous} since if all the leaves of a tree have the same
 length, the same value of the spectral parameter is used.

 \begin{definition} 
 We call $[,]_X$ the sesquilinear form on $\mathcal H$ determined by this weakly invariant transfer matrix (see 
 proposition \ref{determines}). If $X$ is a fixed point for $\mathcal R$ we let $X$ mean the constant sequence with all terms equal to $X$.
 \end{definition}
 Let us make a couple of easy general remarks. From now on we will use $\mathcal R'$ to denote $\mathcal R$ on the projective space
 $\mathbb P P_4$ and, for $X\in P_4$, $\mathcal R'(X)$ to denote $\mathcal R'$ of the class of $X$ in $\mathbb P P_4$. 
 \begin{enumerate}
 \item  A sequence $X_i'$ of 
 back iterates of $\mathcal R'$ can always be lifted to back iterates $X_i$ of $\mathcal R$. The liftings are unique up to signs.
 \item A sequence $X_i$ as above is by definition an element of the projective or inverse limit of the inverse system over $\mathbb N$
 all of whose spaces are $P_4$ and whose connecting maps are all $\mathcal R$. Thus up to signs the space of spatially homogeneous
 scale invariant transfer matrices is a subspace of the inverse limit of  $\mathbb P P_4$  with connecting maps $\mathcal R'$, a compact space.
 \item
 Back iterating from a given $X_1$ may or may not be possible since $\mathcal R$ may not be surjective.A special case where
 back iteration is always possible  is obtained from a fixed point or more generally a   \emph{periodic point} $W$ for $\mathcal R$ (i.e. there exists
 a $p\in \mathbb N$ for which $\mathcal R^p(W)=W$).One may then put $X_1=W$ and choose the back iterates  $X_i$
 to be $\mathcal R^j(W)$ for the appropriate choice of $W$. We let $[,]_W$ be the sesquilinear form for this choice of
 back iterates.
 
  \item In an entire  neighbourhood of a repelling fixed point of $\mathcal R$ back iteration is always possible:
 \begin{definition} A fixed point $X$ for $\mathcal R$ is called \emph{repelling} if there is a norm $||-||$ on $P_4$ and  $\epsilon>0$ 
 such that $||\mathcal R(Y)-X||> ||Y-X||$ whenever $||Y-X||<\epsilon$.
 \end{definition}
 \begin{theorem}\label{repelling} If $X$ is a repelling fixed point for $\mathcal R$ there is a neighbourhood $V$ of $X$ such that 
 $\mbox{for all  } Y \in V, \mbox{ there are }Y_i\in V, i=1,2,\cdots, \mbox{ with } Y_1=Y, \mathcal R(Y_{i+1})=Y_i$ for all $i$.
 
 The corresponding weakly invariant invariant transfer matrix is weakly analytic as a function of $Y$.
 
 \end{theorem}
 \begin{proof}By the repelling property $\mathcal R$ can be inverted locally to give $\mathcal R^{-1}$ which
  takes the ball of radius $\epsilon$
 inside itself. So choose the back iterates by iterating $\mathcal R^{-1}$. The inverse is analytic so that if $\xi$ and $\eta$ are fixed
 $[\xi,\eta]_{\{Y_i\}}=\langle (\mathcal R^{-1})^n\xi,\eta\rangle$ for some fixed $n$ which is as analytic as $\mathcal R^{-1}$.
 \end{proof}

 \end{enumerate} 
  \subsection{Topsey-turvey momenta.}
  One can rephrase Stone's theorem on one parameter unitary groups as follows:
  "Given a strongly continuous one parameter unitary group on Hilbert space, $t\mapsto u_t$, 
  $\langle u_t \xi, \eta \rangle-\langle \xi,\eta\rangle$ tends to zero as $t\rightarrow 0$, but one may
  renormalise $\langle u_t \xi, \eta \rangle$ by dividing by $t$ to obtain a sesquiliear form $[,]$ on a dense subspace such
  that $$\frac{\langle u_t \xi, \eta \rangle-\langle \xi,\eta\rangle}{t}\rightarrow [\xi,\eta].$$
  
  Thus in quantum mechanics one obtains energy and momentum from time evolution and space translation.
  
       If $\rho_n$ is the rotation (=translation) by $\frac{1}{2^n}$ on the direct limit Hilbert space we 
 saw in \cite{jonogo} that $\langle \rho_n\xi,\eta\rangle$ is determined by iterating $\mathcal R$ starting with $\BA$. $\mathcal R$ being a homogeneous quadratic, there is always a neighbourhood of zero within which all points iterate rapidly
 towards zero under it. The result of \cite{jonogo} followed simply by showing that $\BA$ is in such a neighbourhood
 after a few iterates. In this paper we will take inspiration from Stone's theorem and \emph{renormalise} $\langle \rho_n\xi,\eta\rangle$ so that it has  limits as $n\rightarrow \infty$. These limits
will be expressed in terms of $[\xi,\eta]_X$ for some $X$'s.
  Unfortunately we may not always know the exact rate at which $\langle \rho_n \xi, \eta \rangle$ approaches zero but it will be
  possible to create the quadratic form without that knowledge. The situation is a little more complicated than in Stone's theorem as
  there will be two quadratic forms rather than one that control the behaviour of $\langle \rho_n \xi, \eta \rangle$.

  Since $\rho_n$ is a unitary given by spatial translation (by a single lattice site) we will call these quadratic forms the 
  topsey turvey momenta.

  \begin{lemma} Suppose $X$ is a fixed point for $\mathcal R$ and that  $X_n$ is a sequence in $P_4$ with
  $\lim_{n\rightarrow \infty} X_n= [X]$. Then there exist constants $c_n\in \mathbb C$ such that, for any $\xi,\eta\in \mathcal H$, $$\lim_{n\rightarrow \infty} c_n\langle T_{X_n}\xi,\eta\rangle=[\xi,\eta]_X$$ 
  \end{lemma} 
  \begin{proof} The action of $\mathbb C^\times$ on $P_4\setminus \{0\}$ gives a locally trivial fibre bundle with base space
  $\mathbb P P_4$. Thus in a trivialising neighbourhood of $X$ one can just lift the $X_n$ to $c_n X_n$ so they have the same
  vertical coordinate as $X$ in this trivialisation. Fix $\xi$ and $\eta$ in some $P_{2^k}$. By the estimate of theorem 4.0.1 of
  \cite{jonogo}, $||\lambda_nX_n-X||$ tends to zero so
   $c_n\langle T_{X_n}\xi,\eta\rangle\rightarrow \langle T_{X}\xi,\eta\rangle=[\xi,\eta]_X$.
  \end{proof}
  Now we can get the result we want:
  \begin{theorem}Suppose $X$ is a fixed point for $\mathcal R^p$ and that  $W\in P_4$ is such that
  $\lim_{n\rightarrow \infty} \mathcal (R')^{np+i}(W)= (\mathcal R')^i (X)$ for $i=0,1,\cdots, p-1$. Then there exist constants $c_m\in \mathbb C$ such that, for any $\xi,\eta\in \mathcal H$, and $i=0,1,\cdots, p-1$, $$\lim_{n\rightarrow \infty} c_{np+i}\langle T_{\mathcal (R)^{np+i}(W)}\xi,\eta\rangle=[\xi,\eta]_{\mathcal R^i(X)}$$ 
  
  \end{theorem}
  \begin{proof}Apply the argument of the previous lemma to the sequence $\mathcal R^{np}(W)$ then apply $\mathcal R$ $p-1$ times
  to the conclusion.
  \end{proof} 
  \begin{remark}\rm{An explicit choice of $c_n$ can generally be made by choosing $\xi=\eta=\Omega=$some unit vector
  in $P_1$ (or $P_2$ if $dim(P_1)=0$) to obtain the choice
  $$c_n=\frac{[\Omega,\Omega]_X}{\langle T_{X_n}\Omega,\Omega\rangle}$$ but it would be better to have an exact expression
  for this, at least asymptotically. (Note that the numerator is just a simple quadratic in the coefficients of $X$.}
  \end{remark} 
  
  As an immediate corollary we get the topsey-turvey momentum operators:
  \begin{theorem}Suppose there is an $X\in P_4$ so that $\mathcal R^2$ converges to $X$ on iteration starting at $\BA$.
  Then there are two sesquilinear forms $[,]_\pm$, and numbers $c_n$ and $d_n$ such that, for $\xi,\eta\in P_{2^{2k}}$,
  $$\lim_{n\rightarrow \infty} c_n\langle  \rho_{2n}(\xi),\eta\rangle=[\xi,\eta]_+$$ and
  $$\lim_{n\rightarrow \infty} c_n\langle  \rho_{2n+1}(\xi),\eta\rangle=[\xi,\eta]_-$$
  \end{theorem}
  We will see in the next subsection that the existence of $X$ holds in every example that we examine. It might be true universally.
 \subsection{The case $\mathcal P=\mathcal Q$.}
 
 In this case $\mathcal P=\mathcal Q$, the dynamical system is a quadratic map from $\mathbb C^3$ to $\mathbb C^3$ so
 should properly be thought of as a dynamical system on $\mathbb C\mathbb P^2$.  Once
 again the braid solution is an important one which will be missed over the reals, at least for $d<3$, but in keeping with this paper we will tend to 
 think of $\mathcal R$ as a map on $\mathbb R \mathbb P^2$ which means we can fix a circle at infinity and draw pictures in the plane. 
 Beware that this is \emph{not} the plane of $\mathbb C \mathbb P^1$ much beloved of dynamical systems people. In
 particular $\mathcal R$ is not surjective. We have seen that both
 iteration and back-iteration of this dynamical system are relevant. Note that fixed points on $\mathbb C\mathbb P^2$
 and $\mathbb C^3$ are the same apart from $0$ and $\infty$ since a point in $\mathbb C^3$ that is fixed up to a scalar
 can be rescaled to an absolute fixed point.
 
Let $$X=p\BA+q\BB+r\BC=(p,q,r)$$ be an arbitrary element of $Q_4$.
 \begin{proposition}
%\begin{align} 

\begin{align*}\mathcal R(X)=(\frac{d^{2} - 5d + 7}{(d - 1)^2}p^2 + 2pq + 2\frac{d - 2}{d - 1}pr + 
 q^2 + r^2)\quad & \BA \\
 +(-\frac{1}{(d - 1)^3} p^2- \frac{1}{d-1}(2pq + q^2)) \quad&  \BB \\
 +\frac{d^2-3d+3}{(d-1)^3}p^2 + \frac{1}{d-1}(2pq + q^2))\quad & \BC   \end{align*}

 \end{proposition}
 \begin{proof}
 This is just a calculation using the skein theory of $\mathcal Q$.\end{proof}
 
 We want to understand the dynamical systems $\mathcal R$ for different values of $d$. At this stage there is
 insufficient justification for doing this for all values of $d$ so we will content ourselves with presenting 
 representative  results for 4 different regimes of behaviour: $d= 2, 2<d<3, d=3$ and $d>3$. For a value chosen
 in each of these cases we give a portrait of the significant points in $\mathbb R\mathbb P^2$ with values rounded off 
 to a few decimal places. We have chosen
 to send the circle $p=0$ to infinity since this makes $\BA$ the origin in the portrait.
 
 \begin {enumerate}
 \item $d=$ golden ratio, $\frac{1+\sqrt 5}{2}$.
 
 This case was not covered in \cite{jonogo} though it is actually easier than anything there. In this case $Q_4$ has dimension
 equal to $2$ and is spanned by $\BB$ and $\BC$ (see \cite{MPS}). The relation $\BA=\BB-\frac{1}{d} \BC$ holds. It is extremely
 easy to calculate $$\mathcal R (q\BB+r\BC)=(r^2-\frac{q^2}{d})\BA +(q^2-\frac{r^2}{2})\BB$$.
 One checks immediately that $\langle \rho_n\xi,\eta\rangle\rightarrow 0$ very rapidly as $n\rightarrow \infty$ just as in \cite{jonogo}.
Since $dim(Q_4)=2$,  the projective version of $\mathcal R$ is actually a rational dynamical system on $\mathbb C\mathbb P^1$
but it is rather boring-the Julia set is the unit circle $\mathcal R$ interchanges the inside and outside. It has an attracting orbit
of period two to which all points not on the unit circle converge under iteration.
 \item $d=2$ This case does not merit a portrait since the entire plane is mapped to the single line $r=-q$ and
 $\BA$ is a fixed point for $\mathcal R^2$ so that the sesquilinear forms $[,]_\pm$ are just given by $[,]_{\BA}$ and
 $[,]_{\mathcal R(\BA)}$. The only fixed points for $\mathcal R'$ are the braid crossings and the real point $(-1/2,1/2)$.
 
 \item $d=1+\sqrt 2$\\
 
 Portrait: \\
 \vpic {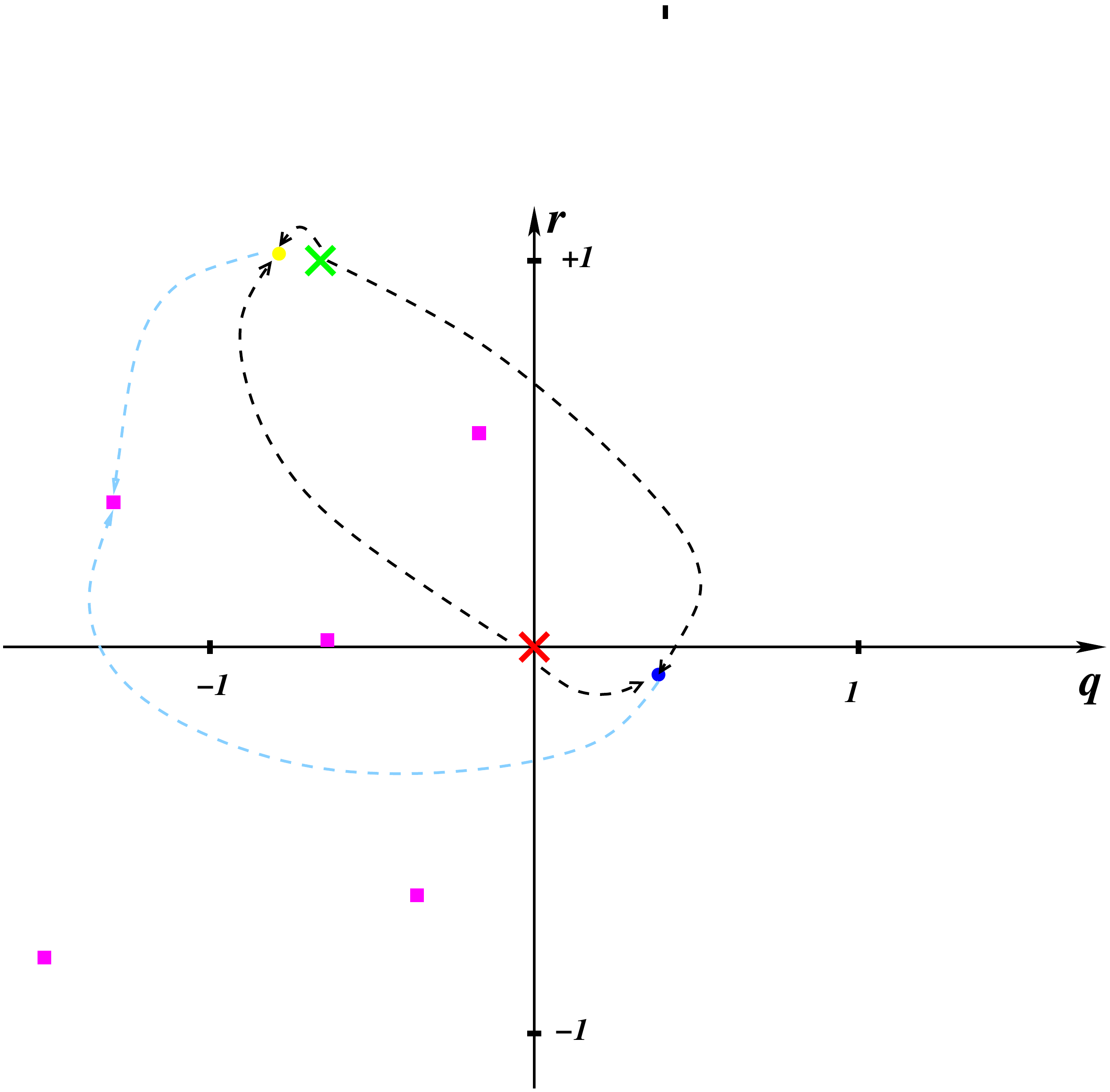} {4in}
 
 \5 Legend:  \\
 \vpic {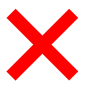} {0.15in}  The element \BA of $P_4$  $=(0,0)$ \\
 \vpic {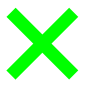} {0.15in} The element \BD of $P_4$  $=(-0.707,1)$ \\
 \vpic {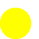} {0.1in}  \vpic {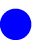} {0.1in} Stable points of period 2. $(-0.825,1.022)$ and $(0.315,-0.118)$\\
 \vpic {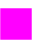} {0.1in}  All the real fixed points for $\mathcal R$. $(-1.557,-0.850)$, repelling, $(-1.332,0.332)$,repelling,\\
 $(-0.61,0.098)$, unstable, 
 $(-0.375,-0.625)$, repelling and $(-0.186,0.521)$, unstable.\\
 \vpic {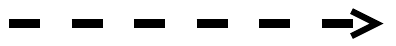} {0.6in}   Iterates to. \\
 \vpic {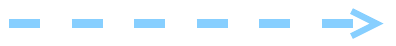} {0.6in}  Backiterates to.

 Missing are two complex fixed points, the positive and negative braid crossings.
 
 \item $d=3$\\
 
 Portrait: \\
 \vpic {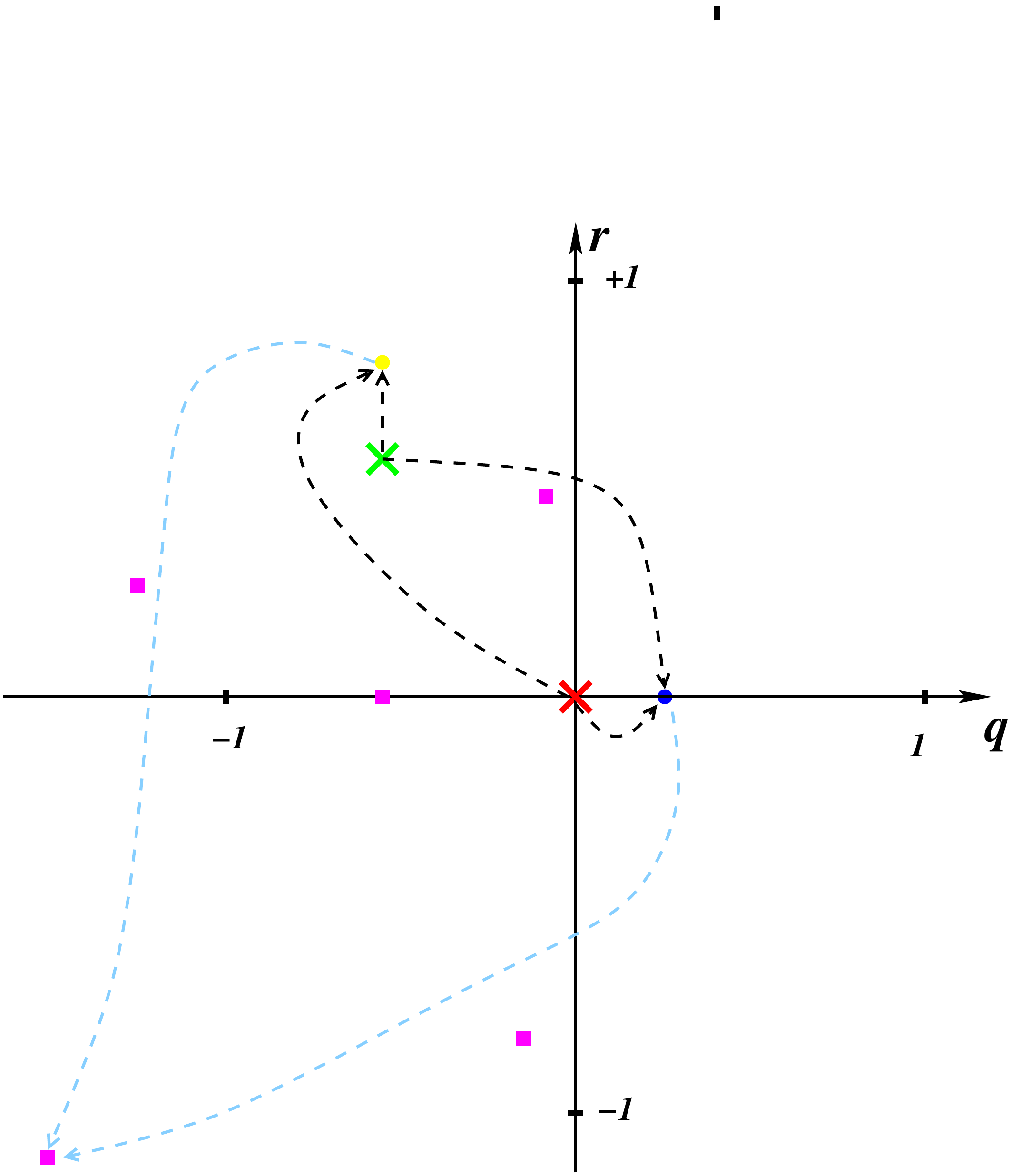} {4in}
 
 \5 Legend:  \\
 \vpic {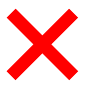} {0.15in}  The element \BA of $P_4$  $=(0,0)$ \\
 \vpic {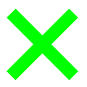} {0.15in} The element \BD of $P_4$  $=(-0.5,0.5)$ \\
 \vpic {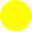} {0.1in}  \vpic {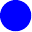} {0.1in} Stable points of period 2. $(-0.5,0.781)$ and $(0.281,0)$\\
 \vpic {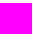} {0.1in}  All the real fixed points for $\mathcal R$. $(-1.675,-1.175)$, repelling, $(-1.309,0.309)$,repelling,\\
 $(-0.5,0)$, neutral-this is the braid solution, 
 $(-0.191,-0.809)$, repelling and $(-0.075,0.425)$, unstable.\\
 \vpic {iteratesto} {0.6in}   Iterates to. \\
 \vpic {backiteratesto} {0.6in}  Backiterates to.
 
 No complex fixed points this time. The braid solution is symmetric and real.
 
\item For $d$ between $3$ and just under 3.52783, the braid solution at $d=3$ splits into three fixed points, two of
 which are the braid solutions, but apart from this the picture looks much like at $d=3$. 
 
 \item Something more interesting
 happens at 3.52783. The two periodic points that are the limits of $\BA$ and $\BD$ under iteration coalesce. \\
 Here is the portrait for large $d$:\\
 
 $d=26.04$\\
 
 \vpic {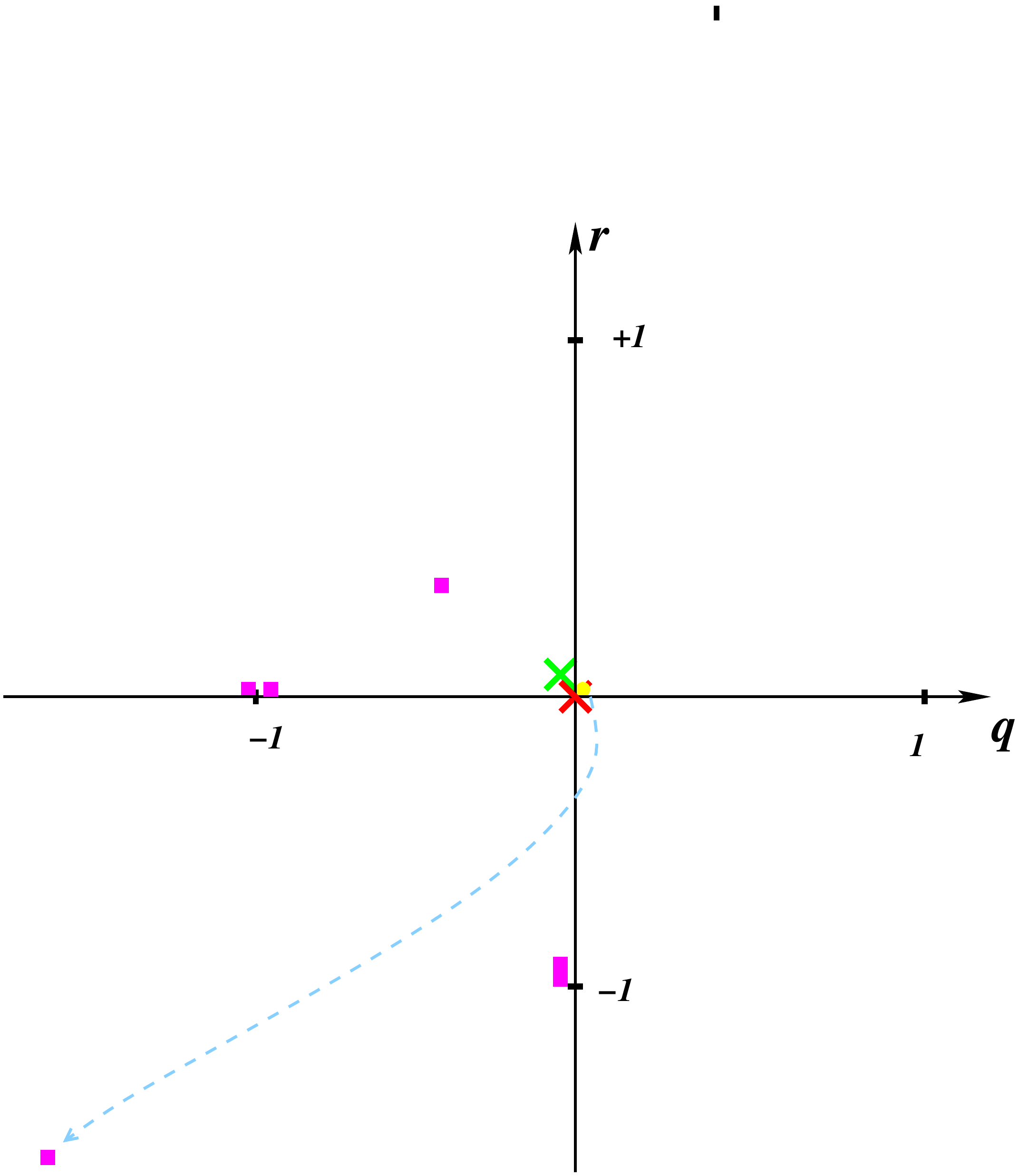} {3.5in}
 
 \5 Legend:  \\
 \vpic {redcross} {0.15in}  The element \BA of $P_4$  $=(0,0)$ \\
 \vpic {greencross} {0.15in} The element \BD of $P_4$  $=(-0.0399,0.0399)$ \\
 \vpic {yellowdot} {0.1in}  Unique attracting fixed point.  Limit of  \vpic {redcross} {0.15in} and 
  \vpic {greencross} {0.15in} under iteration.\\
 \vpic {purpledot} {0.1in}  All the real fixed points for $\mathcal R$\\
 \vpic {backiteratesto} {0.6in}  Backiterates to.

 \end{enumerate}
 \begin{remark}\rm{ The fixed and periodic points other than the limits under iteration and backiteration do
 not yet have any interest for the scale invariant physics but they might be of interest for topology as
 they provide coefficient-like functions on the Thompson group.}
\end{remark}
 
 \section{Weakly scale invariant nearest neighbour Hamiltonians.}
 By a nearest neighbour Hamiltonian we mean an operator of the form:
 $$\sum_{i=1}^n id\otimes id\otimes \cdots \otimes h_i \otimes \cdots \otimes id$$
 where $h_i$ are self-adjoint linear maps on $\mathcal h\otimes \mathcal h$  and $h_i$ comes after
 the $(i-1)$th. tensor product symbol, with periodic boundary conditions. It is spatially homogeneous if $h_i$ is independent of $i$.
 We will call this map $\mathfrak H_n$, being deliberately ambiguous about the tensor power of $\mathcal H$ on which it acts.
 (The extension to general planar algebras is obvious.)
 
 The analogue of the ABC equation has no solutions for $\mathcal Q$, spatially inhomogeneous
 or not, but if we only require weak scale invariance there are solutions. We restrict our attention to
 the spatially homogeneous case. We will see that the equation only invokes a \emph{linear} 
 renormalisation so is very easy to solve.
 
 \begin{definition} The \emph{scale invariance} map $\mathcal S: End(\mathcal h\otimes \mathcal h)\rightarrow End(\mathcal h\otimes \mathcal h)$
 is the map $$\mathcal S(h) =\vpic {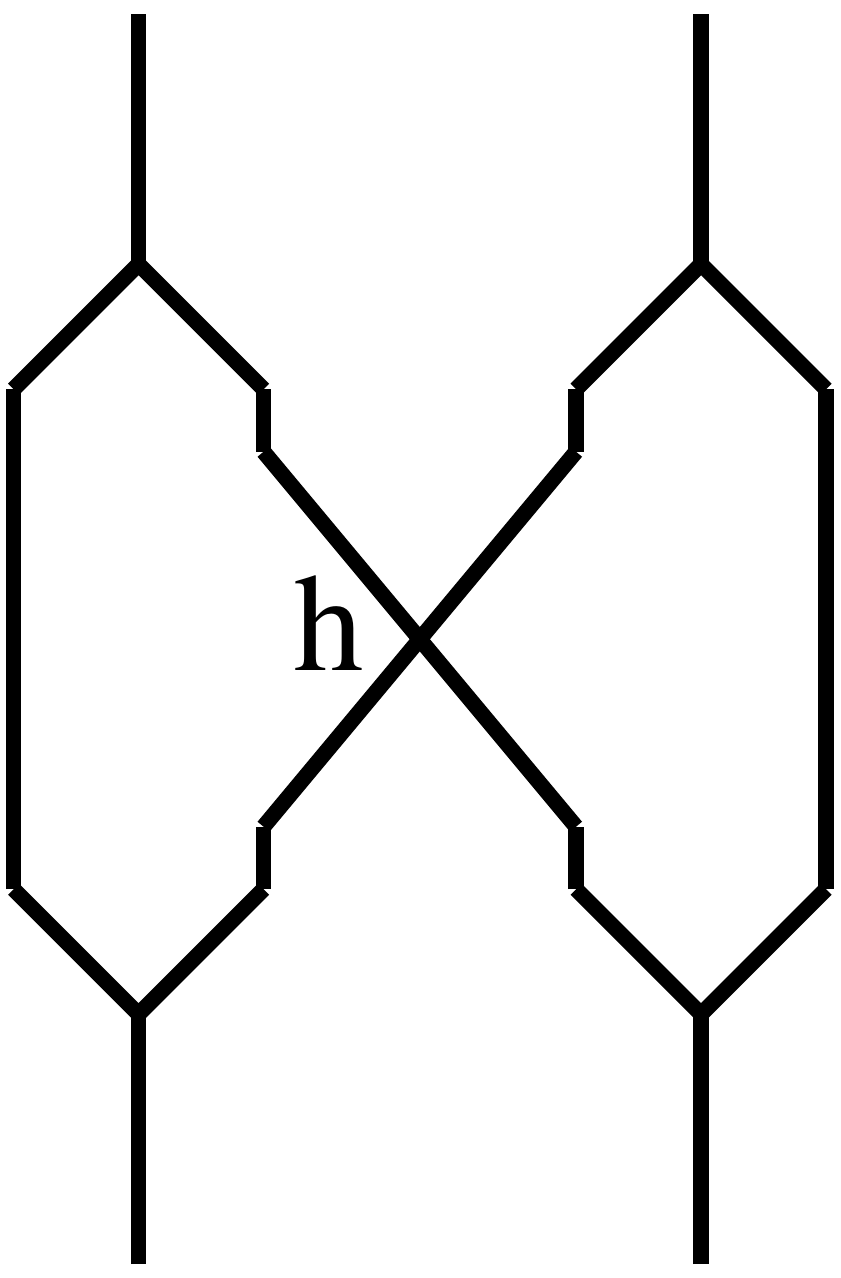} {1in}\qquad + \qquad \vpic {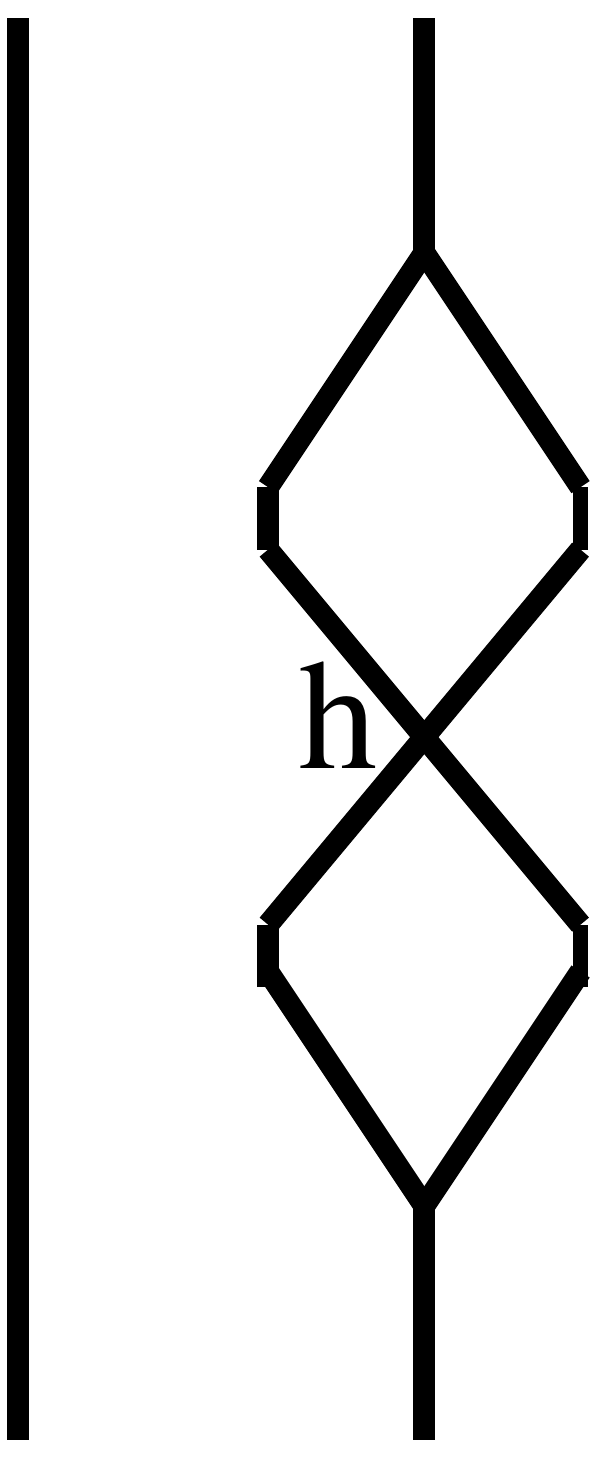} {0.6in} $$
 
 \end{definition}
We assume the notation of section 2.2 with the direct system $t \mapsto  A(t)=
\otimes^n \mathcal h$, ($t$ being a tree with $n$ leaves)  $Y$ being used to define the embedding maps $\iota_s^t$ for $s\leq t$.

 We will content ourselves to use the scale invariance map to construct nearest neighbour Hamiltonians on
 for the cofinal sequence $\mathcal T_n$ for which the vector space is $A(\mathcal T_n)=\otimes^{2^n} \mathcal h$, leaving it to the reader
 to sort out the restrictions to an arbitrary $A(t)$.
 
 Recall that for $s\leq t \in \mathfrak T$, the inclusion 
 \begin{theorem} Given $h\in End(\mathcal h \otimes \mathcal h)$ consider $\mathfrak H_h$ on the Hilbert space $A(\mathcal T_n)$. Then for 
 $\xi, \eta \in A(\mathcal T_n)$, 
 $$< \mathfrak H_{h}(\iota_{\mathcal T_n}^{\mathcal T_{n+1}}(\xi)),\iota_{\mathcal T_n}^{\mathcal T_{n+1}}(\eta)>=<\mathfrak H_{\mathcal S(h)}\xi,\eta>$$ so that  $\langle\mathfrak H_h \xi,\eta\rangle$ on $\otimes^{2^n} \mathcal h$ extends to a sesquilinear form on $\mathcal H$ provided 
 $\mathcal S ^{-n}(\{h\}) \neq \emptyset$ for all $k$.
 \end{theorem}
 \begin{proof} In evaluating $< \mathfrak H_{h}(\iota_{\mathcal T_n}^{\mathcal T_{n+1}}(\xi)),\iota_{\mathcal T_n}^{\mathcal T_{n+1}}(\eta)>$ there are two kinds of terms according to the parity of $i$ in the sum defining $\mathfrak H_{h}$. In one kind $h$ is surrounded by \vpic {scaleinvarianthameqn1} {0.7in} and in the other kind by \vpic {scaleinvarianthameq2} {0.5in}. Grouping the terms together and
 using periodic boundary conditions one obtains the result. 
 \end{proof}
 
 Thus the eigenvectors of $\mathcal S$ can be used to construct very scale-invariant nearest neighbour hamiltonians.
 
 Finally we calculate $\mathcal S:Q_4\rightarrow Q_4$ in our example to make sure it is generic enough for our usual
 linear algebra intuition to hold.
 
 \begin{proposition}  The matrix of $\mathcal S$ with respect to the basis $\{\vpic {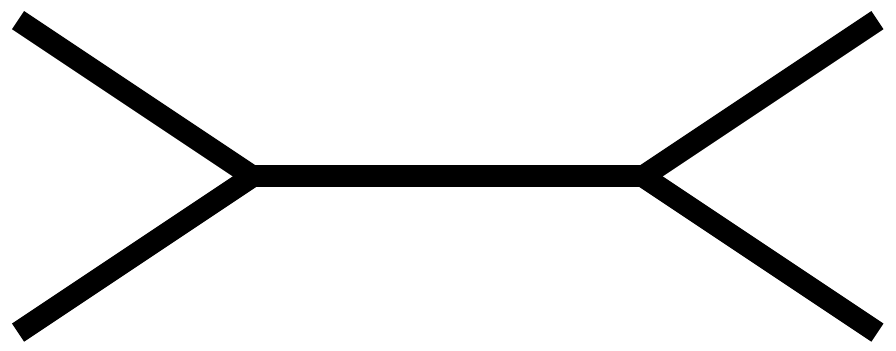} {0.3in} ,\vpic {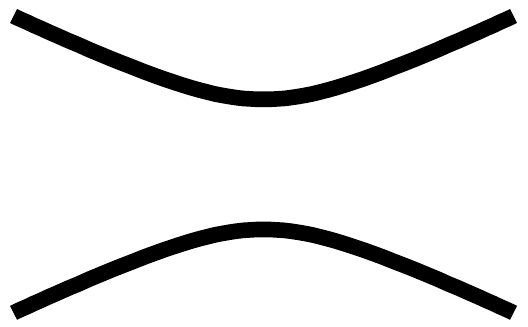} {0.3in} ,\vpic {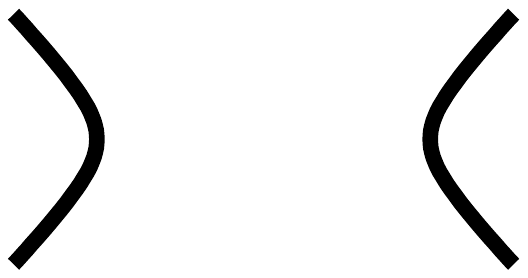} {0.3in} \}$
 of $Q_4$ is 
 $$\begin{pmatrix}
 \frac{(d-2)^2}{(d-1)^2}&\frac{d-3}{d-1}&0\\
 0&\frac{1}{(d-1)^2}&0\\
 \frac{d-2}{d-1}&\frac{d-2}{(d-1)^2}&2
 \end{pmatrix} $$
 
  \end{proposition}
  \begin{proof} Just calculate using the skein relations of $\mathcal Q$.
  \end{proof}
  
  The eigenvalues of $\mathcal S$ are thus $\displaystyle 2, \frac{(d-2)^2}{(d-1)^2}$ and $\displaystyle \frac{1}{(d-1)^2}$ with
  eigenvectors $$\vpic {idv} {0.3in} \quad,$$  $$(d^2-2) \vpic {h} {0.3in} +(d-1)(d-2) \vpic {idv} {0.3in} \quad , \mbox {  and  }$$
  $$  -(2d^2-4d+1) \vpic {h} {0.3in} +(2d^2-4d+1) \vpic {cupcap} {0.3in} +(d-2)^2 \vpic {idv} {0.3in} $$
  respectively.
  
  So provided $d\neq 2$ $\mathcal S$ is invertible. 
  
  The eigenvalue $2$ is of no interest since the eigenvector is just the identity
  and we would get the identity Hamiltonian. 
  
  For $d< 3$ the largest eigenvalue of $\mathcal S^{-1}$ is $\frac{(d-2)^2}{(d-1)^2}$ so any scale invariant spatially homogeneous nearest neighbour Hamiltonian will tend, modulo scalars, to that with $h=(d^2-2) \vpic {h} {0.3in} +(d-1)(d-2) \vpic {idv} {0.3in}$.
  
  And for $d>3$ the same limit Hamiltonian will have $h=  -(2d^2-4d+1) \vpic {h} {0.3in} +(2d^2-4d+1) \vpic {cupcap} {0.3in} +(d-2)^2 \vpic {idv} {0.3in} $.
  
  For both of these possibilities for $h$, as elements of the algebra $Q_4$, their eigenvalues are both positive and negative. 
  One can force the spectrum of the Hamiltonian to be positive on a given $Q_n$ by adding a large multiple of the identity, but
  on applying $\mathcal S^{-1}$ enough times, as required by scale invariance, the spectrum will eventually contain negative
  numbers as well as positive. 
   
   Thus, not surprisingly for topsey turvey systems, the spectrum of a scale invariant Hamiltonian can only be positive
   for the trivial Hamiltonian.
 %Hidden note-find calculation of spectrum of eigenvectors of S in Hamiltonianrenormalisation.nb

\end{document}